\documentclass{amsart}
\usepackage{amsmath}
\usepackage{amsfonts}
\usepackage[mathscr]{eucal}
\usepackage{amssymb}
\usepackage{marginnote}
\usepackage{hyperref}

\usepackage{tikz-cd} 
\usepackage{verbatim}

\newtheorem{thm}{Theorem}[section]

\newtheorem{cor}[thm]{Corollary}

\newtheorem{lemma}[thm]{Lemma}
\newtheorem{prop}[thm]{Proposition}

\theoremstyle{definition}
\newtheorem{defn}[thm]{Definition}
\newtheorem{remark}[thm]{Remark}
\newtheorem{exam}[thm]{Example}

\newtheorem{notation}[thm]{Notation}

\newcommand{\bb}[1]{\mathbb{#1}}
\newcommand{\cl}[1]{\mathcal{#1}}
\newcommand{\ff}[1]{\mathfrak{#1}}

\newcommand{\inner}[2]{\left\langle {#1}, {#2} \right\rangle}

\newcommand{\tr}{\text{Tr}}

\usepackage[all]{xy}
\begin{document}

\title[Absolutely Norming Operators on S.N. Ideals]{A Spectral Characterization of Absolutely Norming Operators on S.N. Ideals}

\author[Satish~K.~Pandey]{Satish K.~Pandey}
\date{July 5, 2016}

\begin{abstract}
The class of absolutely norming operators on complex Hilbert spaces of arbitrary dimensions was introduced in \cite{CN} and a spectral characterization theorem for these operators was established in \cite{VpSp}. In this paper we extend the concept of absolutely norming operators to various symmetric norms. We establish a few spectral characterization theorems for operators on complex Hilbert spaces that are absolutely norming with respect to various symmetric norms. It is also shown that for many symmetric norms the absolutely norming operators have the same spectral characterization as proven earlier for the class of operators that are absolutely norming with respect to the usual operator norm.
 Finally, we prove the existence of a symmetric norm on the algebra $\mathcal B (\mathcal H)$ with respect to which even the identity operator does not attain its norm. 
\end{abstract}

\maketitle

\tableofcontents
\section{Introduction} 
Throughout this paper $\cl{H}$ and $\cl{K}$ will denote complex Hilbert spaces and we write $\cl{B}(\cl{H},\cl{K})$ (respectively $\cl{B}(\cl{H}$)) for the set of all bounded linear operators from $\cl{H}$ to $\cl{K}$ (respectively from $\cl{H}$ to $\cl{H}$). We recall that $\cl{B}(\cl{H},\cl{K})$ is a complex Banach space with respect to the operator norm $\|T\| = \text{sup}\{\|Tx\|_{\cl{K}}: x\in \cl{H}, \|x\|_{\cl{H}} \leqslant 1\}.$
We recall the following definition.

\begin{defn}\cite[Definitions 1.1,1.2]{VpSp}\label{Def1}
An operator $T\in \mathcal{B}(\mathcal{H},\mathcal{K})$ is said to be \emph{norming} or \emph{norm attaining} if there is an element $x\in\mathcal{H}$ with $\|x\|=1$ such that $\|T\|=\|Tx\|$. We say that $T\in \mathcal{B}(\mathcal{H},\mathcal{K})$ is \emph{absolutely norming} if for every nontrivial closed subspace $\cl M$ of $\cl H$, $T|_\cl M$ is norming. 

We let $\mathcal N(\mathcal{H},\mathcal{K})$ (or $\cl N$) and $\mathcal{AN}(\mathcal{H},\mathcal{K})$ (or $\cl{AN}$) respectively denote the set of norming and absolutely norming operators in $\mathcal{B}(\mathcal{H},\mathcal{K})$.
\end{defn}

There is a wealth of information on norm attaining operators; see, for instance, \cite{AcRu02, AcRu98, AcAgPa, Acosta, Aguirre, Partington, Scha2, Scha, Shkarin} and references therein. The class of absolutely norming operators, however, was introduced recently in \cite{CN} and studied in \cite{CN},\cite{R}. Carvajal and Neves \cite{CN} proved a partial structure theorem \cite[Theorem 3.25]{CN} for the class of positive operators on complex Hilbert spaces that included an uncharacterized ``remainder'' operator. The result \cite[Theorem 5.1]{VpSp} established a spectral characterization for positive operators which asserts that a positive operator is absolutely norming if and only if it is the sum of a positive compact operator, a self-adjoint finite-rank operator, and a nonnegative scalar multiple of the identity operator. This theorem was then carried over to bounded operators which we recall here.

\begin{thm}[Spectral Theorem for Operators in $\cl{AN}(\cl H,\cl K)$]\cite[Theorem 6.4]{VpSp}\label{Main-result-VpSp}
Let $\cl{H}$ and $\cl{K}$ be complex Hilbert spaces of arbitrary dimensions, let $T\in \cl{B}(\cl{H},\cl{K})$ and let $T=U|T|$ be its polar decomposition. 
Then $T\in\cl{AN}$ if and only if $|T|$ is of the form $|T|= \alpha I +F+K$, where $\alpha \geq 0, K$ is a positive compact operator and $F$ is self-adjoint finite-rank operator. 
\end{thm}

The above theorem opened up new territories to explore. Can the concept of ``absolutely norming'' be carried over to norms other than the operator norm? If yes, then can we characterize such class of operators?

In this paper we extend the concept of absolutely norming operators to several particular symmetric norms that are equivalent to the operator norm. We do this with an eye towards the objective of characterizing these classes. We single out three of these symmetric norms for more detailed study: the Ky Fan $k$-norm in section 3, the weighted Ky Fan $\pi, k$-norm in section 6, and the $(p,k)$-singular norm in section 9. 

In section 3 and 4 we develop results for $[k]$-norming and absolutely $[k]$-norming operators (see Definitions \ref{N_[k]Def} and \ref{AN_[k]Def}) on $\cl B(\cl H,\cl K)$. These results parallel those for norming and absolutely norming operators on $\cl B(\cl H, \cl K)$. In subsections 4.1 and 4.2 we establish necessary and sufficient conditions for a positive operator to belong to the class $\cl {AN}_{[k]}(\cl H,\cl K)$ (see Definition \ref{AN_[k]Def}) and present a spectral characterization theorem (see Theorem \ref{SpThPAN_[k]}) for such operators. This leads us to extend the result to bounded operators in section 5 where we establish a spectral characterization theorem for the family $\{\cl {AN}_{[k]}(\cl H,\cl K):k\in \bb N\}$(see Theorem \ref{SpThAN}). The operators that belong to this family has the same form as that of those which belong to the class $\cl {AN}(\cl H,\cl K)$ or to the class $\cl {AN}_{[k]}(\cl H,\cl K)$ for any $k\in \bb N$.

Sections 6, 7 and 8 introduce weighted Ky Fan $\pi, k$-norm and study the family $\{\cl{AN}_{[\pi, k]}(\cl H, \cl K):\pi \in \Pi, k\in \bb N\}$ of operators (see Definition \ref{AN_[pi,k]Def}), where we use $\Pi$ to denote the set of all nonincreasing sequences of positive numbers with their first term equal to 1. The family of these classes is large. For instance, it contains $\cl {AN}_{[k]}(\cl H,\cl K)$ for every $k\in \bb N$. We develop the ground results for the operators that belong to this family and present a spectral characterization theorem for the entire family (see Theorems \ref{SpThPAN_[pi,k]} and \ref{SpThAN_[pi, k]}). These results can be thought of as a generalization of those we prove in sections 4 and 5. The operators $T$ that belong to either of these families have property that $|T|$ is also also of the form $\alpha I+K+F$, the notations being as before. 

In sections 9, 10 and 11 we work through the class of absolutely $(p,k)$-norming operators (see Definition \ref{AN_(p,k)Def}) in $\cl B(\cl H,\cl K)$ and establish a spectral characterization theorem for the family $\{\cl {AN}_{(p,k)}(\cl H,\cl K):p\in[1,\infty), k\in \bb N\}$ (see Definition \ref{AN_(p,k)Def}). This family too contains $\cl {AN}_{[k]}(\cl H,\cl K)$ for every $k\in \bb N$ and hence the results that we establish here may as well be considered as a generalization of those we prove in section 4 and 5. We establish a spectral characterization theorem for the operators that belong to this family (see Theorem \ref{SpThAN_(p, k)}) and find that these operators are too of the form $\alpha I+K+F$.

 As a corollary to the several spectral characterization theorems we prove, we see that every positive operator of the form $\alpha I + K + F$ belongs to each of the families $\cl{AN}_{[k]}\cl B(\cl H), \cl{AN}_{[\pi,k]}\cl B(\cl H)$ and $\cl{AN}_{(p,k)}\cl B(\cl H)$. So, it might appear at this stage that with respect to \textit{every} symmetric norm $\|\cdot\|_s$ on $\cl B(\cl H)$, the positive operators on $\cl B(\cl H)$ that are of the above form, are ``absolutely $s$-norming''. In section 12, we prove the following proposition that violates our intuition and renders the identity operator nonnorming. From section 12 onwards all Hilbert spaces are considered to be separable.

\begin{prop}\label{Identity-nonnorming}
There exists a symmetric norm $\|\cdot\|_{\Phi_\pi^*}$ on $\cl B(\ell^2(\bb N))$ such that $I\notin \cl N_{\Phi_\pi^*}(\ell^2(\bb N))$.
\end{prop}
 We collect some facts about symmetrically-normed ideals (s.n. ideals) from \cite{GK} and then establish the definition of ``$s$-norming" and ``absolutely $s$-norming" operators on the s.n. ideal $(\cl B(\cl H),\|\cdot\|_s)$. We also prove that every compact operator in the s.n. ideal $(\cl B(\cl H), \|\cdot\|_s)$ is ``absolutely $s$-norming", where $\cl H$ is a separable Hilbert space and $\|\cdot\|_s$ is an arbitrary symmetric norm on $\cl B(\cl H)$ (see Theorem \ref{compacts-are-norming}).

\section{Preliminaries}
Let $\cl{H}$ and $\cl{K}$ be complex Hilbert spaces and $T\in \cl{B}(\cl{H}, \cl{K})$. We define $|T|:=\sqrt{T^*T}$ --- this is conventionally known as the absolute value (or modulus) of the operator $T$ --- such that $|T|^2 =T^*T.$ If $T$ is compact, then $|T|$ is a positive compact operator on $\cl{H}$.
We use $\cl B_{0}(\cl H,\cl K)$ (respectively $\cl{B}_0(\cl{H}$)) to denote the set of all compact operators in $\cl B(\cl H,\cl K)$ (respectively from $\cl{H}$ to $\cl{H}$).
\begin{defn}[s-numbers of Compact Operators]\label{s-comp}
Let $\cl{H}$ and $\cl{K}$ be Hilbert spaces and $T\in \cl B_{0}(\cl{H}, \cl{K})$. The  \textit{singular values} or \textit{s-numbers} of $T$ are the eigenvalues of $|T|$.
Needless to say that we can enumerate the nonzero eigenvalues  $\lambda_1(|T|), \lambda_2(|T|),...$ of $|T|$ in decreasing order, taking account of their multiplicities, that is, $\lambda_1(|T|)\geq \lambda_2(|T|)\geq...$; and hence can enumerate the nonzero s-numbers $s_1(T), s_2(T),...$ of $T$ in decreasing order, taking account of their multiplicities as well, so that 
$$
s_j(T)=\lambda_j(|T|) \hspace{2cm} (j=1,2,...,rank(|T|)).
$$
If $rank(|T|)<\infty$ we define $s_j(T)=0$ for $j> rank(|T|)$.
\end{defn}

We now generalize this concept from compact operators to bounded linear operators. This requires us to define the numbers $\lambda_j(|T|)$ for $T\in \cl B(\cl H, \cl K)$ which parallel the definition in the case when $T\in \cl B_0(\cl H,\cl K)$. So, our next task is to define the numbers $\lambda_j(A)$ for a positive operator $A\in \cl{B}(\cl{H}).$ For this we need the following definition.

\begin{defn}[essential spectrum of a self-adjoint operator]
Let $T\in \cl{B}(\cl{H})$ be a self-adjoint operator. A point $\lambda$ in the spectrum $\sigma (T)$ of $T$ is said to be in the essential spectrum $\sigma_{e}(T)$ of $T$ if it is either an accumulation point of $\sigma(T)$ or an eigenvalue of $T$ with infinite multiplicity.
\end{defn}

\begin{defn}\label{singular}
Let $A\in \cl{B}(\cl{H})$ be a positive operator and let $\mu =\sup\{\nu:\nu\in \sigma(A)\}$. If $\mu \in \sigma_{e}(A)$ we define $\lambda_j (A):= \mu \hspace{0.2cm} (j=1,2,...,rank(A))$. 
If $\mu \notin \sigma_{e}(A)$ then it is an eigenvalue of $A$ with finite multiplicity, say $m$. In this case, we define 
\begin{align*}
\lambda_j (A)&:= \mu \hspace{1cm} (j=1,2,...,m).\\
\lambda_{m+j}(A)&:= \lambda_j(A_1) \hspace{1cm} (j=1,2,...,rank(A_1)).\\
\end{align*}
where $A_1 = A-\mu P_{E_\mu}$ with $P_{E_\mu}$ being the orthogonal projection of $\cl{H}$ onto the eigenspace $E_\mu$ corresponding to the eigenvalue $\mu$.
If $rank(A)<\infty$ we define $\lambda_j(T)=0$ for $j> rank(A)$.
\end{defn}

This notion agrees with the original definition if $A$ is compact. In the light of the above definition, the following definition makes sense.
\begin{defn}[s-numbers of arbitrary bounded linear operator]\label{s-bdd}
The s-numbers of an arbitrary operator $T\in \cl{B}(\cl{H},\cl{K})$ are defined as
$$
s_j(T)=\lambda_j(|T|) \hspace{2cm} (j=1,2,...,rank(|T|)).
$$
If $rank(|T|)<\infty$ we define $s_j(T)=0$ for $j> rank(|T|)$.
\end{defn}

This completes the formal description of the s-numbers of arbitrary bounded linear operators. We now define the notion of a symmetric norm on a two-sided ideal of $\cl B(\cl H)$. An ideal of $\cl B(\cl H)$ always means a two-sided ideal.

\begin{defn}[Symmetric Norm]
Let $\cl{I}$ be an ideal of the algebra $\cl{B}(\cl{H})$ of operators on a complex Hilbert space. A \textit{symmetric norm} on $\cl{I}$ is a function $\|.\|_s:\cl{I}\longrightarrow [0,\infty)$ which satisfies the following six conditions:\\
(1) $\|X\|_s\geq 0$ for each $X\in \cl{I}$.\\
(2) $\|X\|_s=0$ if and only if $X=0.$\\ 
(3) $\|\lambda X\|_s=|\lambda|\|X\|_s$ for every $X\in \cl{I}$ and $\lambda\in \bb{C}$.\\
(4) $\|X+Y\|_s\leq\|X\|_s+\|Y\|_s$ for every $X,Y\in \cl{I}$.\\
(5) $\|AXB\|_s\leq\|A\|\|X\|_s\|B\|$ for every $A,B \in \cl{B}(\cl{H})$ and $X\in \cl{I}$.\\
(6) $\|X\|_s=\|X\|=s_1(X)$ for every rank-one operator $X\in \cl{I}$.\\
\end{defn}

\begin{remark}
In the definition of symmetric norm, if we consider the ideal $\cl I$ to be $\cl B(\cl H)$, then it is said to be a \textit{symmetric norm on $\cl B(\cl H)$}. That is, this definition can be extended to the trivial ideals as well. 
Moreover, the following observations are obvious: 
\begin{enumerate}
\item the usual operator norm on any ideal $\cl{I}$ of $\cl{B}(\cl{H})$, including the trivial ideals, is a symmetric norm; and
\item every symmetric norm on $\cl B(\cl H)$ is topologically equivalent to the ordinary operator norm. 
\end{enumerate}
\end{remark}

\section {The Classes $\cl N_{[k]}$ and $\cl {AN}_{[k]}$}

\begin{defn}[Ky Fan $k$-norm]\cite{KF}
For a given natural number $k$, the Ky Fan $k$-norm $\|\cdot\|_{[k]}$ of an operator $T\in \cl{B}(\cl{H},\cl{K})$ is defined to be the sum of the $k$ largest singular values of $T$, that is, 
$$
\|T\|_{[k]}=\sum_{j=1}^ks_j(T).
$$
The Ky Fan $k$-norm on $\cl{B}(\cl{H},\cl{K})$ is, indeed, a norm. It is not difficult to see that it is, in fact, a symmetric norm on $\cl B(\cl H)$.
\end{defn}

\begin{remark}
Note that the smallest of Ky Fan norms, the Ky Fan 1-norm, is equal to the operator norm.
\end{remark}

\begin{defn}\label{N_[k]Def}
For any $k\in \bb{N}$, an operator $T\in \cl{B}(\cl{H},\cl{K})$ is said to be \emph{$[k]$-norming} if there are orthonormal elements $x_1,...,x_k \in \cl{H}$ such that $\|T\|_{[k]}=\|Tx_1\|+...+\|Tx_k\|$. If $\dim(\cl{H})=r<k$, we define $T$ to be \emph{$[k]$-norming} if there exist orthonormal elements $x_1,...,x_r\in \cl{H}$ such that $\|T\|_{[k]}=\|Tx_1\|+...+\|Tx_r\|$.
We let $\cl{N}_{[k]}(\cl{H},\cl{K})$ denote the set of $[k]$-norming operators in $\cl{B}(\cl{H},\cl{K})$. 
\end{defn}
 A generalization of the above property leads to a new class of operators in $\cl{B}(\cl{H},\cl{K})$. 

\begin{defn}\label{AN_[k]Def}
For any $k\in \bb{N}$, an operator $T\in \cl{B}(\cl{H},\cl{K})$ is said to be \emph{absolutely $[k]$-norming} if for every nontrivial closed subspace $\cl{M}$
of $\cl{H}$, $T|_{\cl{M}}$ is \emph{$[k]$-norming}. We let  $\cl{AN}_{[k]}(\cl{H},\cl{K})$ denote the set of absolutely $[k]$-norming operators in $\cl{B}(\cl{H},\cl{K})$. 
\end{defn}

Alternatively, an operator $T\in \cl{B}(\cl{H},\cl{K})$ is said to be
an absolutely $[k]$-norming operator if for every nontrivial closed subspace
$\cl{M}$ of $\cl{H}$ with dimension $k$ or more, there are orthonormal elements $x_1,...,x_k \in \cl{M}$ such that $\|T|_{\cl{M}}\|_{[k]}=\|T|_{\cl{M}}x_1\|+...+\|T|_{\cl{M}}x_k\|$. For a closed subspace $\cl{M}$ of $\cl{H}$ with $\dim(\cl{M})=r<k$, the definition implies that $T$ is absolutely $[k]$-norming if there exist orthonormal elements $x_1,...,x_r \in \cl{M}$ such that $\|T|_{\cl{M}}\|_{[k]}=\|T|_{\cl{M}}x_1\|+...+\|T|_{\cl{M}}x_r\|$. Needless to say,  every absolutely $[k]$-norming operator is $[k]$-norming, that is, $\cl{AN}_{[k]}(\cl{H},\cl{K})\subseteq \cl{N}_{[k]}(\cl{H},\cl{K})$.

\begin{remark}Since, in the finite-dimensional setting, the geometric multiplicity of an eigenvalue of a diagonalizable operator is the same as its algebraic multiplicity and the singular values of an operator $T$ are precisely the eigenvalues of the positive operator $|T|$, it immediately follows that \textit{every operator on a finite-dimensional Hilbert space is $[k]$-norming for any $k\in \bb N$.} This is not true when the Hilbert space in question is not finite-dimensional (see Example \ref{Counterex}). 
\end{remark}

 There is an important and useful criterion for an operator $T\in \cl{B}(\cl{H},\cl{K})$
to be absolutely $[k]$-norming, which is stated in the following lemma.

\begin{lemma}\label{Trick_[k]}
For a closed linear subspace $\cl{M}$ of a complex Hilbert space $\cl{H}$ let $V_{\cl{M}}:\cl{M}\longrightarrow \cl{H}$ be the inclusion map from $\cl{M}$ to $\cl{H}$ defined as $V_{\cl{M}}(x) = x$ for each $x\in \cl{M}$ and let $T\in \cl{B}(\cl{H},\cl{K})$.
For any $k\in\bb{N}$, $T\in \cl{AN}_{[k]}(\cl{H},\cl{K})$ if and only if for every nontrivial closed linear subspace $\cl{M}$ of $\cl{H}$, $TV_{\cl{M}}\in\cl{N}_{[k]}(\cl{M},\cl{K})$.
\end{lemma} 
  
  \begin{proof}   
     To prove this assertion we first observe that for any given nontrivial closed subspace $\cl{M}$ of $\cl{H}$, the maps $TV_\cl{M}$ and $T|_\cl{M}$ are identical and so are their singular values which implies $\|TV_\cl{M}\|_{[k]} = \|T|_\cl{M}\|_{[k]}$.

We next assume that $T\in\cl{AN}_{[k]}(\cl{H},\cl{K})$ and prove the forward implication. Let $\cl{M}$ be an arbitrary but fixed nontrivial closed subspace of $\cl{H}$. Either $\dim(\cl{M})=r<k$, in which case, there exist orthonormal elements $x_1,...,x_r \in \cl{M}$ such that $\|T|_\cl{M}\|_{[k]} = \|T|_{\cl{M}}x_1\|+...+\|T|_{\cl{M}}x_r\|$ which means that there exist orthonormal elements $x_1,...,x_r \in \cl{M}$ such that $\|TV_\cl{M}\|_{[k]} =\|T|_\cl{M}\|_{[k]}=\|T|_{\cl{M}}x_1\|+...+\|T|_{\cl{M}}x_r\| = \|TV_\cl{M} x_1\|+...+ \|TV_\cl{M} x_r\|$ proving that $TV_{\cl{M}}\in \cl{N}_{[k]}(\cl{M},\cl{K})$, or $\dim(\cl{M})\geq k$, in which case, there exist orthonormal elements $x_1,...,x_k \in \cl{M}$ such that $\|T|_\cl{M}\|_{[k]} = \|T|_{\cl{M}}x_1\|+...+\|T|_{\cl{M}}x_k\|$ which means that there exist orthonormal elements $x_1,...,x_k \in \cl{M}$ such that $\|TV_\cl{M}\|_{[k]} =\|T|_\cl{M}\|_{[k]}=\|T|_{\cl{M}}x_1\|+...+\|T|_{\cl{M}}x_k\| = \|TV_\cl{M} x_1\|+...+ \|TV_\cl{M} x_k\|$ proving that $TV_{\cl{M}}\in\cl{N}_{[k]}(\cl{M},\cl{K})$. Since $\cl{M}$ is arbitrary, it follows that $TV_\cl{M}\in\cl{N}_{[k]}(\cl{M},\cl{K})$ for every $\cl{M}$.

We complete the proof by showing that $T\in\cl{AN}_{[k]}(\cl{H},\cl{K})$ if $TV_\cl{M}\in\cl{N}_{[k]}(\cl{M},\cl{K})$ for every nontrivial closed subspace $\cl{M}$ of $\cl{H}$. We again fix $\cl{M}$ to be an arbitrary nontrivial closed subspace of $\cl{H}$. Since $TV_\cl{M}\in \cl{N}_{[k]}(\cl{M},\cl{K})$, either $\dim(\cl{M})=r<k$, in which case, there exist orthonormal elements $x_1,...,x_r \in \cl{M}$ such that $\|TV_\cl{M}\|_{[k]}=\|TV_\cl{M} x_1\|+...+ \|TV_\cl{M} x_r\|$ which means that there exist orthonormal elements $x_1,...,x_r \in \cl{M}$ such that
 $\|T|_\cl{M}\|_{[k]}=\|TV_\cl{M}\|_{[k]}=\|TV_\cl{M} x_1\|+...+ \|TV_\cl{M} x_r\|=\|T|_\cl{M} x_1\|+...+ \|T|_\cl{M} x_r\|$ proving that $T|_{\cl{M}}\in\cl{N}_{[k]}(\cl{M},\cl{K})$, or  $\dim(\cl{M})\geq k$, in which case, there exist  orthonormal elements $x_1,...,x_k \in \cl{M}$ such that $\|TV_\cl{M}\|_{[k]}=\|TV_\cl{M}x_1\|+...+\|TV_\cl{M}x_k\|$ which means that there exist orthonormal elements $x_1,...,x_k \in \cl{M}$ such that $\|T|_\cl{M}\|_{[k]}=\|TV_\cl{M}\|_{[k]}=\|TV_\cl{M}x_1\|+...+\|TV_\cl{M}x_k\|=\|T|_\cl{M}x_1\|+...+\|T|_\cl{M}x_k\|$ proving that $T|_{\cl{M}}\in \cl{N}_{[k]}(\cl{M},\cl{K})$. Because $\cl{M}$ is arbitrary, this essentially guarantees that $T\in\cl{AN}_{[k]}$. Since $k\in \bb{N}$ is arbitrary, the assertion holds for each $k\in \bb{N}.$
\end{proof}

 \begin{prop}\label{Tiff|T|_[k]}
Let $\cl{H}$ and $\cl{K}$ be complex Hilbert spaces and let $T\in \cl{B}(\cl{H},\cl{K})$. Then for every $k\in \bb{N}$, $T\in \cl{AN}_{[k]}(\cl{H},\cl{K})$ if and only if $|T|\in \cl{AN}_{[k]}(\cl{H})$.
 \end{prop}
 \begin{proof}
 Let $\cl{M}$ be an arbitrary nontrivial closed subspace of $\cl{H}$ and let $V_{\cl{M}}:\cl{M}\longrightarrow \cl{H}$ be the inclusion map from $\cl{M}$ to $\cl{H}$ defined as $V_{\cl{M}}(x) = x$ for each $x\in \cl{M}$. Notice that $|TV_{\cl{M}}|=|\,|T|V_{\cl{M}}\,|$; for
\begin{multline*}
|TV_{\cl{M}}|^2=V_{\cl{M}}^*T^*TV_{\cl{M}}=V_{\cl{M}}^*|T|^2V_{\cl{M}}\\=(V_{\cl{M}}^*|T|)(|T|V_{\cl{M}})=(|T|V_{\cl{M}})^*(|T|V_{\cl{M}})=|\,|T|V_{\cl{M}}\,|^2.
\end{multline*}
Consequently, for every $j$, $\lambda_j(|TV_{\cl{M}}|)=\lambda_j(|\,|T|V_{\cl{M}}\,|)$ and hence $s_j(TV_{\cl{M}})=s_j(|T|V_{\cl{M}})$. This implies that for each $k\in \bb{N}$, we have
$$
\|TV_{\cl{M}}\|_{[k]}=\|\,|T|V_{\cl{M}}\,\|_{[k]}.
$$
That for each $x\in \cl{H}$, $\|TV_{\cl{M}}x\| = \|\,|T|V_{\cl{M}}x\,\|$ is a trivial observation.
Since $\cl{M}$ is arbitrary, by Lemma \ref{Trick_[k]} the assertion follows.
\end{proof}

\begin{remark}
For the remaining part of this article, we use $\cl{N}_{[k]}$ and $\cl{AN}_{[k]}$ for $\cl{N}_{[k]}(\cl{H},\cl{K})$ and $\cl{AN}_{[k]}(\cl{H},\cl{K})$ respectively, as long as the domain and codomain spaces are obvious from the context.
\end{remark}

\section{Spectral Characterization of Positive Operators in $\cl{AN}_{[k]}$}
The purpose of this section is to study the necessary and sufficient conditions for a positive operator on complex Hilbert space of arbitrary dimension to be absolutely $[k]$-norming for any $k\in \bb N$ and to characterize such operators.

\begin{prop}\label{1.1}
Suppose $A\in \cl B(\cl H)$ be a positive operator, $\mu =\sup\{\nu:\nu\in \sigma(A)\}$, and $\mu \notin \sigma_{e}(A)$, in which case, it is an eigenvalue of $A$ with finite multiplicity, say $m$, so that for every $j\in \{1,...,m\}, s_j (A)= \mu$. Then the following statements are equivalent.
\begin{enumerate}
\item $s_{m+1}(A)$ is an eigenvalue of $A$.
\item $s_{m+1}(A)$ is an eigenvalue of $A-\mu P_{E_\mu}$, where $P_{E_\mu}$ is the orthogonal projection of $\cl{H}$ onto the eigenspace $E_\mu$ corresponding to the eigenvalue $\mu$.
\item $(A-\mu P_{E_\mu})|_{E_{\mu}^\perp}:E_{\mu}^\perp\longrightarrow E_\mu^\perp$ is norming, that is, $(A-\mu P_{E_\mu})|_{E_{\mu}^\perp}\in\cl N$.
\item $A|_{E_{\mu}^\perp}:E_{\mu}^\perp\longrightarrow E_\mu^\perp$ is norming, that is, $A|_{E_{\mu}^\perp}\in \cl N$.
\item $A\in\cl N_{[m+1]}$.
\end{enumerate}
\end{prop}
\begin{proof}
$(1)\iff(2)$: The backward implication is trivial. For the forward implication, let $\lambda=s_{m+1}(A):=\sup\{\nu:\nu\in \sigma(A-\mu P_{E_\mu})\}$. Assume that $\lambda$ is an eigenvalue of $A$. Then there exists some nonzero vector $x\in \cl H$ such that $Ax=\lambda x$. It suffices to prove that $x\perp E_{\mu}$, for then $(A-\mu P_{E_\mu})x=Ax=\lambda x$. But $A\geq 0$ and $\lambda\neq \mu$ which implies that $x\perp E_{\mu}$.

$(2)\iff(3)$: Since
$$
 A-\mu P_{E_\mu}(x)= \begin{cases} 
     0 & \text{ if } x\in E_\mu \\
      Ax & \text{ if } x\in E_\mu^\perp,\\
   \end{cases}
$$ 
$A-\mu P_{E_\mu}$ is a positive operator on $\cl B(\cl H)$ and $E_\mu^\perp$ is a closed subspace of $\cl H$ which is invariant under $A-\mu P_{E_\mu}$ which implies that $(A-\mu P)_{E_\mu}|_{E_{\mu}^\perp}:E_{\mu}^\perp\longrightarrow E_\mu^\perp$, viewed as an operator on $E_\mu^\perp$, is positive and $\|(A-\mu P_{E_\mu})|_{E_\mu^\perp}\|=s_{m+1}(A)$. By \cite[Theorem 2.3]{VpSp} we know that a positive operator $T$ belongs to $\cl N$ if and only if $\|T\|$ is an eigenvalue of $T$. Thus $(A-\mu P_{E_\mu})|_{E_{\mu}^\perp}\in\cl N$ if and only if $s_{m+1}(A)$ is an eigenvalue of $(A-\mu P_{E_\mu})|_{E_{\mu}^\perp}$ if and only if $s_{m+1}(A)$ is an eigenvalue of $A-\mu P_{E_\mu}$.

 $(3)\iff(4)$: This equivalence follows trivially from the fact that the maps $(A-\mu P_{E_\mu})|_{E_{\mu}^\perp},$ and $A|_{E_{\mu}^\perp}$ are identical on $E_{\mu}^\perp$.

$(3)\iff(5)$: Notice that $A\in \cl N_{[m]}$; for $\|A\|_{[m]}=m\mu$ and since the geometric multiplicity of $\mu$ is $m$, we can find a set $\{v_1,...,v_m\}$ of $m$ orthonormal vectors in $E_\mu\subseteq \cl H$ such that $\sum_{i=1}^m\|Av_i\|=m\mu=\|A\|_{[m]}.$ Also, it is not very difficult to observe that if there exists any set $\{w_1,...,w_m\}$ of $m$ orthonormal vectors in $\cl H$ such that $\sum_{i=1}^m\|Aw_i\|=m\mu$, then this set has to be contained in $E_\mu$. This observation implies that $A\in \cl N_{[m+1]}$ if and only if there exists a unit vector $x\in E_\mu^\perp$ such that $\|Ax\|=s_{m+1}(A)$ which is possible if and only if $A-\mu P_{E_\mu}|_{E_{\mu}^\perp}:E_{\mu}^\perp\longrightarrow E_\mu^\perp$ is norming because $\|A-\mu P_{E_\mu}\|=\|(A-\mu P_{E_\mu})|_{E_\mu^\perp}\|=s_{m+1}(A)$.
\end{proof}

\begin{remark}
The above proposition holds even if $\mu\in \sigma_e(A)$, $\mu$ is an accumulation point but not an eigenvalue; for we can consider it to be an eigenvalue with multiplicity $0$. If $\mu\in \sigma_e(A)$ is an accumulation point as well as an eigenvalue with finite multiplicity, say $m$, then one can still prove $(2)\iff(3)\iff(4)\iff(5)$; the condition $(1)$ no longer remains equivalent to other conditions.
\end{remark}

\begin{prop}\label{1.2}
If $A\in \cl B(\cl H)$ be a positive operator and $s_{m+1}(A)\neq s_m(A)$ for some $m\in \bb N$, then $A\in \cl N_{[m]}$. Moreover, in this case, $A\in \cl N_{[m+1]}$ if and only if $s_{m+1}(A)$ is an eigenvalue of $A$. 
\end{prop}
\begin{proof}
It is easy to see that for every $j\in \{1,...,m\}, s_j(A)\notin \sigma_{\text{e}}(A)$. Then the set $\{s_1(A),...,s_m(A)\}$ consists of eigenvalues (not necessarily distinct) of $A$, each having finite multiplicity. This guarantees the existence of an orthonormal set $\{v_1,...,v_m\}\subseteq K\subseteq \cl H$ such that $Av_j=s_j(A)v_j$ which yields $\|A\|_{[m]}=\|Av_1\|+...+\|Av_m\|$, where $K$ is the joint span of the eigenspaces corresponding to the eigenvalues $\{s_1(A),...,s_m(A)\}$, which implies that $A\in \cl N_{[m]}$. Furthermore, we observe that if there exists any orthonormal set $\{w_1,...,w_m\}$ of $m$ vectors in $\cl H$ such that $\sum_{i=1}^m\|Aw_i\|=\sum_{j=1}^ms_j(A)$, then this set has to be contained in $K$. Note that $K^\perp$ is invariant under $A$ and hence    $A|_{K^\perp}:K^\perp\longrightarrow K^\perp$, viewed as an operator on $K^\perp$, is positive. It follows then that $A\in \cl N_{m+1}$ if and only if there exists a unit vector $x\in K^\perp$ such that $\|Ax\|=s_{m+1}(A)$, which is possible if and only if $A|_{K^\perp}:K^\perp\longrightarrow K^\perp$, viewed as an operator on $K^\perp$, belongs to $\cl N$, which in turn happens if and only if $s_{m+1}(A)$ is an eigenvalue of $A|_{K^\perp}$, since $\|A|_{K^\perp}\|=s_{m+1}(A)$. But $s_{m+1}(A)\neq s_m(A)$ implies that $s_{m+1}(A)$ is an eigenvalue of $A|_{K^\perp}$ if and only if $s_{m+1}(A)$ is an eigenvalue of $A$. This proves the assertion.
\end{proof}

\subsection{Necessary Conditions for Positive Operators in $\cl{AN}_{[k]}$}

\begin{prop}\label{1.6}
Let $A\in \cl{B}(\cl{H})$ be a positive operator  and $k\in \bb{N}$. If $A\in \cl N_{[k]}$, then $s_1(A),...,s_k(A)$ are eigenvalues of $A$.
\end{prop}
\begin{proof}
The proof is by contrapositive. Assuming that at least one of the elements from the set $\{s_1(A),..., s_k(A)\}$ is not an eigenvalue of $A$, we show that $A\notin\cl N_{[k]}$. Suppose that $s_1(A)$ is not an eigenvalue of $A$. Then it must be an accumulation point of the spectrum of $A$ in which case none of the singular values of $A$ is an eigenvalue of $A$ and that  $s_j(A)=s_1(A)$ for every $j\geq 2$. Since $s_1(A)=\|A\|$, it follows from \cite[Theorem 2.3]{VpSp} that $A\notin\cl N$ which means that for every $x\in \cl H, \|x\|=1$, we have $\|Ax\|<\|A\|=s_1(A)$. Consequently, for every orthonormal set $\{x_1,...,x_k\}\subseteq \cl H$ we have 
$\sum_{j=1}^k\|Ax_j\|<k\|A\|=\sum_{j=1}^ks_j(A)$ so that $A\notin \cl N_{[k]}$.  

Next suppose that $s_1(A)$ is an eigenvalue of $A$ but $s_2(A)$ is not. Clearly then $s_1(A)$ is an eigenvalue with multiplicity $1$, $s_2(A)\neq s_1(A)$ and $s_j(A)=s_2(A)$ for every $j\geq 3$ in which case Proposition \ref{1.2} ascertains that $A\in \cl N$ but $A\notin \cl N_{[2]}$. This implies that there exists $y_1\in \cl H$ with $\|y_1\|=1$ such that $\|Ay_1\|=\|A\|$ and for every $y\in \text{span}\{y_1\}^\perp$ with $
\|y\|=1$ we have $\|Ay\|<s_2(A)$ which in turn implies that for every orthonormal set $\{y_2,...y_k\}\subseteq \text{span}\{y_1\}^\perp$ we have $\sum_{j=2}^k\|Ay_j\|<(k-1)s_2(A)=\sum_{j=2}^ks_j(A)$. This yields  $\sum_{j=1}^k\|Ay_j\|<\sum_{j=1}^ks_j(A)$ for every orthonormal set $\{y_1,...,y_k\}\subseteq \cl H$ which implies that $A\notin \cl N_{[k]}$.

If $s_1(A), s_2(A)$ are eigenvalues of $A$ but $s_3(A)$ is not, then we have 
$s_3(A)\neq s_2(A)$ and $s_j(A)=s_3(A)$ for every $j\geq 4$ in which case Proposition \ref{1.2} asserts that $A\in \cl N_{[2]}$ but $A\notin \cl N_{[3]}$. Consequently, there exists an orthonormal set $\{z_1,z_2\}\subseteq \cl H$ such that $\|Tz_1\|+\|Tz_2\|=\|T\|_{[2]}$ and that for every unit vector $z\in \text{span}\{z_1,z_2\}^\perp$ we have $\|Tz\|<s_3(A)$ which in turn implies that for every orthonormal set $\{z_3,...z_k\}\subseteq \text{span}\{z_1,z_2\}^\perp$ we have $\sum_{j=3}^k\|Az_j\|<(k-2)s_3(A)=\sum_{j=3}^ks_j(A)$. It then follows that $\sum_{j=1}^k\|Az_j\|<\sum_{j=1}^ks_j(A)$ for every orthonormal set $\{z_1,...,z_k\}\subseteq \cl H$ which implies that $A\notin \cl N_{[k]}$.

  If we continue in this way, we can show at every step that $A\notin \cl N_{[k]}$. We conclude the proof by discussing the final case when $s_1(A),...,s_{k-1}(A)$ are all eigenvalues of $A$ but $s_k(A)$ is not in which case $s_{k}(A)\neq s_{k-1}(A)$ and thus by Proposition \ref{1.2}, we infer that $A\notin \cl N_{[k]}$. This exhausts all the possibilities and the assertion is thus proved contrapositively.
\end{proof}

The converse of the above proposition is not necessarily true as the following example shows.
 \begin{exam}\label{Counterex}
Consider the operator
$$T=\begin{bmatrix}
 1\\
    & 1 & & & &\text{\huge0}&\\
    & & \frac{1}{2}\\
    & & & \frac{2}{3}\\
    & & & & \ddots&\\
    &\text{\huge0} & & & &1-\frac{1}{n} &\\
    & & &  & & & \ddots
\end{bmatrix} \in \cl{B}( \ell^2),$$
with respect to an orthonormal basis $B=\{v_i:i\in \bb N\}$.
That $T$ is positive diagonalizable operator with $\|T\|=1$ is obvious. The spectrum $\sigma (T)$ of $T$ is given by the set $\{1-\frac{1}{n}:n\in \bb N, n>1\}\cup \{1\}$ where $1\in \sigma (T)$ is an accumulation point of the spectrum as well as an eigenvalue of $T$ with multiplicity $2$ and hence $s_j(T)=1$ for each $j\in \bb N.$ Notice that $\{v_1,v_2\}\subseteq B$ serves to be an orthonormal set such that $\|T\|_{[2]}=\|Tv_1\|+\|Tv_2\|$ which implies that $T\in\cl N_{[2]}$. Also, if there exists an orthonormal set $\{w_1,w_2\}\subseteq \ell^2$ of two vectors such that $\|T\|_{[2]}=\|Tw_1\|+\|Tw_2\|$, then this set has to be contained in $\text{span}\{v_1,v_2\}$. $T$ is, however, not $[3]$-norming. To show that there does not exist a unit vector $x\in \text{span}\{v_1,v_2\}^\perp$ such that $\|Tx\|=1$, we consider the diagonal operator 
$$
A:=T-P_{\text{span}\{v_1,v_2\}}=\begin{bmatrix}
 0\\
    & 0 & & & &\text{\huge0}&\\
    & & \frac{1}{2}\\
    & & & \frac{2}{3}\\
    & & & & \ddots&\\
    &\text{\huge0} & & & &1-\frac{1}{n} &\\
    & & &  & & & \ddots
\end{bmatrix},
$$
 where $P_{\text{span}\{v_1,v_2\}}$ is the orthogonal
projection of $\ell^2$ onto the space
span$\{v_1,v_2\}.$ It is not very hard to see that there exists a unit 
vector $x\in \text{span}\{v_1,v_2\}^\perp$ 
with $\|Tx\|=1$ if and only if
$A|_{\text{span}\{v_1,v_2\}^\perp}:\text{span}\{v_1,v_2\}^\perp \longrightarrow \text{span}\{v_1,v_2\}^\perp$
achieves its norm on $\text{span}\{v_1,v_2\}^\perp$. Since $A|_{\text{span}\{v_1,v_2\}^\perp}$ is positive on span$\{v_1,v_2\}^\perp$, it follows that $A|_{\text{span}\{v_1,v_2\}^\perp}\in\cl N$ if and only if $\|A|_{\text{span}\{v_1,v_2\}^\perp}\|=1$ is an eigenvalue of $A|_{\text{span}\{v_1,v_2\}^\perp}$ which is indeed not the case. 
\end{exam}

\begin{prop}\label{1.7}
Let $A\in \cl{B}(\cl{H})$ be a positive operator and $k\in \bb{N}$. If $s_1(A),...,s_k(A)$ are mutually distinct eigenvalues of $A$, then there exists an orthonormal set $\{v_1,...,v_k\}\linebreak
\subseteq \cl H$ such that $Av_j=s_j(A)v_j$ for every $j\in \{1,...,k\}$. Thus $A\in\cl N_{[k]}$.
\end{prop}
\begin{proof}
This is a direct consequence of the fact that the eigenvectors of a normal operator corresponding to distinct eigenvalues are mutually orthogonal.
\end{proof}

An immediate question that arises here is the following: suppose that $s_1(A),...,\linebreak s_k(A)$ are eigenvalues of the positive operator $A$ with $s_1(A)=s_2(A)=...=s_k(A)$. Is it possible for $A$ to be in $\cl N_{[k]}$, and if yes, then under what circumstances? The answer is affirmative and it happens if and only if the geometric multiplicity of the eigenvalue $s_1(A)$ is at least $k$.

\begin{prop}\label{1.8}
Let $A\in \cl{B}(\cl{H})$ be a positive operator, $k\in \bb N$ and let $s_1(A),...,s_k(A)$ be the first $k$ singular values of $A$ that are also the eigenvalues of $A$ and are not necessarily distinct. Then either $s_1(A)=...=s_k(A)$, in which case, $A\in \cl N_{[k]}$ if and only if the multiplicity of $\alpha:=s_1(A)$ is at least $k$; or there exists $t\in \{2,...,k\}$ such that $s_{t-1}(A)\neq s_t(A)=s_{t+1}(A)=...=s_k(A)$, in which case, $A\in \cl N_{[k]}$ if and only if the multiplicity of $\beta:=s_t(A)$ is at least $k-t+1$.
\end{prop}

\begin{proof}
It suffices to establish the assertion of the first case; the second case follows similarly. We thus assume that $s_1(A)=...=s_k(A)$ and prove that $A\in \cl N_{[k]}$ if and only if the multiplicity of $\alpha:=s_1(A)$ is at least $k$. The backward implication is trivial. To see the forward implication, let us assume contrapositively that the geometric multiplicity of $\alpha$ is strictly less that $k$, that is, the dimension of the eigenspace $E_{\alpha}=\ker(A-\alpha I)$ associated with the eigenvalue $\alpha$ is $m<k$. Then $\alpha$ has to be an accumulation point of the spectrum $\sigma (A)$ of $A$ as well; for the number of times an eigenvalue with finite multiplicity appears in the sequence $(s_j(A))_{j\in \bb N}$ exceeds its multiplicity only when it is also an accumulation point of the spectrum. It is easy to see that $A\in\cl N_{[m]}$ since there exists an orthonormal set $\{v_1,...,v_m\}\subseteq E_\alpha$ such that $\|T\|_{[m]}=\|Tv_1\|+...+\|Tv_m\|$. Even more, if there exists any orthonormal set $\{w_1,...,w_m\}\subseteq \cl H$ such that $\|T\|_{[m]}=\|Tw_1\|+...+\|Tw_m\|$, then this set has to be contained in $E_\alpha$. We now show that $A\notin \cl N_{[k]}$.
Let $P_{E_\alpha}$ denote the orthogonal projection of $\cl H$ onto the eigenspace $E_\alpha$. Now consider the positive operator $A-\alpha P_{E_\alpha}$ on $\cl B(\cl H)$ and note that $E_\alpha^\perp$ is a closed subspace of $\cl H$ which is invariant under $A-\alpha P_{E_\alpha}$ which implies that $(A-\alpha P_{E_\alpha})|_{E_\alpha^\perp}:E_\alpha^\perp\longrightarrow E_\alpha^\perp$, viewed as an operator on $E_\alpha^\perp$, is positive and that $\|(A-\alpha P_{E_\alpha})|_{E_\alpha^\perp}\|=s_{m+1}(A)=\alpha$. It is easy to see that $\alpha$ is not an eigenvalue of the positive operator $(A-\alpha P_{E_\alpha})|_{E_\alpha^\perp}$ on $E_\alpha^\perp$. Consequently, this operator does not achieve its norm on $E_\alpha^\perp$ which means that for every $x\in E_\alpha^\perp$ with $\|x\|=1$ we have 
$\|(A-\alpha P_{E_\alpha})|_{E_\alpha^\perp}x\|<s_{m+1}(A)=\alpha$. Thus for every orthonormal set $\{v_{m+1},v_{m+2},...,v_k\}\subseteq E_\alpha^\perp$ we have $\|(A-\alpha P_{E_\alpha})|_{E_\alpha^\perp}v_j\|<s_j(A)=\alpha$, $m+1\leq j\leq k$ so that $\sum_{j=m+1}^k\|(A-\alpha P_{E_\alpha})|_{E_\alpha^\perp}v_j\|<\sum_{j=m+1}^ks_j(A)$. It now follows that for every orthonormal set $\{x_1,...,x_k\}\subseteq \cl H, \sum{j=1}^k\|Ax_j\|<\sum_{j=1}^ks_j(A)=\|A\|_{[k]}$ which implies that $A\notin \cl N_{[k]}$. This proves the proposition.
\end{proof}

 \begin{thm} \label{PositiveN_{[k]}}
Let $A\in \cl{B}(\cl{H})$ be a positive operator and $k\in \bb{N}$. Then the following statements are equivalent.
\begin{enumerate}
\item $A\in\cl{N}_{[k]}.$
\item $s_1(A),...,s_k(A)$ are eigenvalues of $A$ and there exists an orthonormal set $\{v_1,...,v_k\}\subseteq \cl H$ such that $Av_j=s_j(A)v_j$ for every $j\in \{1,...,k\}.$
\end{enumerate} 
\end{thm} 

\begin{proof}
$(1)$ follows from $(2)$ trivially. Assume that $A\in \cl N_{[k]}$. Since $A\geq 0$,
by Proposition \ref{1.6}, $s_1(A),..., s_k(A)$ are all eigenvalues of $A$. If $s_1(A),..., s_k(A)$ are mutually distinct, then by Proposition \ref{1.7} $A\in \cl N_{[k]}$. However, if $s_1(A),...,s_k(A)$ are not necessarily distinct then the Proposition \ref{1.8} yields $A\in \cl N_{[k]}$. This completes the proof.
\end{proof}
 The following corollary is an immediate consequence of the above theorem.
\begin{cor}\label{N_{[k+1]}impliesN_[k]}
Let $k\in \bb{N}$. If $A\in\cl{N}_{[k+1]}(\cl{H},\cl{K})$ is positive, then $A\in \cl{N}_{[k]}(\cl{H},\cl{K})$.  
\end{cor}

\begin{thm}\label{TiffT*T_[k]}
Let $\cl{H}$ and $\cl{K}$ be complex Hilbert spaces, $T\in \cl{B}(\cl{H},\cl{K})$ and $k\in\bb{N}$. Then the following statements are equivalent.\\
$(1)$ $T\in\cl{N}_{[k]}$.\\
$(2)$ $|T|\in\cl{N}_{[k]}$.\\
$(3)$ $T^*T\in\cl{N}_{[k]}$.
\end{thm}

\begin{proof}
The equivalence of $(1)$ and $(2)$ follows from facts that for every $j, s_j(T)=\lambda_j(|T|)=s_j(|T|)$ and for every $x\in \cl H, \|Tx\|=\||T|x\|$.
To prove the equivalence of $(2)$ and $(3)$, we first assume that $|T|\in\cl N_{[k]}$. Since $|T|$ is positive, Theorem \ref{PositiveN_{[k]}} guarantees that $s_1(|T|),...,s_k(|T|)$ are eigenvalues of $|T|$ and there exists an orthonormal set $\{v_1,...,v_k\}\subseteq \cl H$ such that $|T|v_j=s_j(|T|)v_j$ for every $j\in \{1,...,k\}.$ Consequently, for every $j\in  \{1,...,k\}$, we have 
$\||T|v_j\|=s_j(|T|)$. Using this we deduce that for every $j$,
\begin{multline*}
s_j(T^*T)=s_j(|T|^2)=s_j^2(|T|)=\||T|v_j\|^2=\inner{|T|v_j}{|T|v_j}
=\inner{|T|^2v_j}{v_j}\\
=\inner{T^*Tv_j}{v_j}\leq \|T^*Tv_j\|= \||T|^2v_j\|=\|s^2_j(|T|)v_j\|=s_j^2(|T|)=s_j(T^*T),
\end{multline*}
and so we have equality throughout which implies that $s_j(T^*T)=\|T^*Tv_j\|$ for every $j\in \{1,...,k\}$. This yields
$\|T^*T\|_{[k]}=\sum_{j=1}^ks_j(T^*T)=\sum_{j=1}^k\|T^*Tv_j\|$
which implies that $T^*T\in\cl N_{[k]}$. 

Conversely, if $T^*T\in\cl N_{[k]}$, then again by Theorem \ref{PositiveN_{[k]}} $s_1(T^*T),...,s_k(T^*T)$ are eigenvalues of $T^*T$ and there exists an orthonormal set $
\{w_1,...,w_k\}\subseteq \cl H$ such that $T^*Tw_j=s_j(T^*T)w_j$ for every $j\in \{1,...,k\}$.  This gives 
\begin{multline*}
\||T|w_j\|^2=\inner{|T|w_j}{|T|w_j}=\inner{T^*Tw_j}{w_j}=\inner{s_j(T^*T)w_j}{w_j}\\
=s_j(T^*T)\inner{w_j}{w_j}=s_j(T^*T)=s_j(|T|^2)=s_j^2(|T|),
\end{multline*}
which in turn gives $\||T|w_j\|=s_j(|T|)$ for every $j\in \{1,...,k\}$. Then 
$\||T|\|_{[k]}=\sum_{j=1}^ks_j(|T|)=\sum_{j=1}^k\||T|w_j\|$, and the result follows.
\end{proof}

\begin{thm}\label{AN_{[k+1]}impliesAN_[k](general)}
Let $k\in \bb{N}$. Then $\cl{AN}_{[k+1]}(\cl H, \cl K)\subseteq \cl{AN}_{[k]}(\cl H, \cl K)$.
\end{thm}
\begin{proof}
If $A\in \cl{AN}(\cl H, \cl K)$, then Theorem \ref{TiffT*T_[k]} along with Corollary \ref{N_{[k+1]}impliesN_[k]} implies that for every nontrivial closed subspace $\cl M$ of $\cl H$,
\begin{align*}
AV_{\cl M} \in \cl N_{[k+1]}&\iff|AV_{\cl M}| \in \cl N_{[k+1]}\\
& \implies |AV_{\cl M}| \in \cl N_{[k]}\\
&\iff AV_{\cl M} \in \cl N_{[k]}.
\end{align*}
Since the above implications (and both way implications) hold for every $\cl M$, the assertion is proved. 
\end{proof}
 \begin{cor}\label{AN_{[k+1]}impliesAN_[k]}
 Let $k\in \bb{N}$. Then every  positive operator in $\cl{AN}_{[k+1]}$ belongs to $\cl{AN}_{[k]}$.
 \end{cor}

\begin{thm}\label{forward_[k]}
Let $\cl{H}$ be a complex Hilbert space, $A$ be a positive operator on $\cl{H}$, and $k\in \bb N$. If $A\in \cl{AN}_{[k]}$, then $A$ is of the form $A= \alpha I +K+F$, where $\alpha \geq 0, K$ is a positive compact operator and $F$ is self-adjoint finite-rank operator.
\end{thm}

\begin{proof}
Since $A\in \cl{AN}_{[k]}$, $A\in \cl{AN}$. The forward implication of \cite[Theorem 5.1]{VpSp}, hence, implies the assertion.
\end{proof}

We finish this subsection by proving a result which will be useful later in section \ref{symmetrically-normed-ideals} for establishing the notion of absolutely norming operators in symmetrically-normed ideals.

\begin{thm}\label{TVM-TPM-k}
For a closed linear subspace $\cl{M}$ of a complex Hilbert space $\cl{H}$ let $V_{\cl{M}}:\cl{M}\longrightarrow \cl{H}$ be the inclusion map from $\cl{M}$ to $\cl{H}$ defined as $V_{\cl{M}}(x) = x$ for each $x\in \cl{M}$, let $P_{\cl M}\in \cl B(\cl H)$ be the orthogonal projection onto $\cl M$, and let $T\in \cl{B}(\cl{H},\cl{K})$.
For any $k\in\bb{N}$, the following statements are equivalent.
\begin{enumerate}
\item $T\in \cl{AN}_{[k]}(\cl{H},\cl{K})$.
\item $TV_{\cl{M}}\in\cl{N}_{[k]}(\cl{M},\cl{K})$ for every nontrivial closed linear subspace $\cl{M}$ of $\cl{H}$.
\item $TP_{\cl{M}}\in\cl{N}_{[k]}(\cl{H},\cl{K})$ for every nontrivial closed linear subspace $\cl{M}$ of $\cl{H}$.
\end{enumerate} 
\end{thm} 
\begin{proof}
The equivalence of $(1)$ and $(2)$ follows from Lemma \ref{Trick_[k]}. We will prove $(1)\iff (3)$. Fix $\cl M$ to be a nontrivial closed subspace of $\cl H$. A trivial verification shows that $\sigma(|T|_{\cl M}|)\setminus\{0\}=\sigma(|TP_{\cl M}|)\setminus\{0\}$ which implies that the singular values of $T|_{\cl M}$ and $TP_{\cl M}$ are identical, which gives $\|T|_{\cl M}\|_{[k]}=\|TP_{\cl M}\|_{[k]}$. Of course, $\|\,T|_{\cl M}x\,\|=\|TP_{\cl M}x\|$ for each $x\in \cl M$. This establishes the implication $(1)\implies (3)$. 
All that remains is to prove $(3)\implies (1)$.
 Assume that $TP_{\cl M}\in \cl N_{[k]}(\cl H,\cl K)$. Then by Theorem \ref{TiffT*T_[k]} $|TP_{\cl M}|\in \cl N_{[k]}(\cl H)$. Theorem \ref{PositiveN_{[k]}} guarantees the existence of an orthonormal set 
 $\{x_1,...,x_k\}\subseteq \cl H$ with $|TP_{\cl M}|x_j=s_j(|TP_{\cl M}|)x_j$ for every $j\in \{1,...,k\}$ which implies that $|TP_{\cl M}|^2x_j=s_j^2(|TP_{\cl M}|)x_j=s_j(|TP_{\cl M}|^2)x_j$ for every $j\in \{1,...,k\}$. Without loss of generality we assume that for each $j, s_j(|TP_{\cl M}|^2)\neq 0$. Under this assumption it is obvious that $x_j\in \cl M$ for each $j$; for if $x_j$ were not in $\cl M$, then it can't be an eigenvector of $|TP_{\cl M}|^2$ corresponding to the eigenvalue $s_j(|TP_{\cl M}|^2)$. It then follows immediately that 
\begin{multline*}
\|T|_{\cl M}\|_{[k]}
=\|TP_{\cl M}\|_{[k]}
=\|\,|TP_{\cl M}|\,\|_{[k]}
=\sum_{j=1}^k\|\,|TP_{\cl M}|x_j\|\\
=\sum_{j=1}^k\|TP_{\cl M}x_j\|
=\sum_{j=1}^k\|Tx_j\|
=\sum_{j=1}^k\|\,T|_{\cl M}x_j\|,
\end{multline*}
where $\{x_1,...,x_k\}$ is an orthonormal set contained in $\cl M$. Using the fact $s_j(|TP_{\cl M}|)=s_j(TP_{\cl M})$ we conclude $T|_{\cl M}\in \cl N_{[k]}$. But $\cl M$ is arbitrary, so $T\in \cl {AN}_{[k]}(\cl H,\cl K).$ Since $k\in \cl N$ is arbitrary, the assertion holds for each $k\in \cl N.$  
\end{proof}

\subsection{Sufficient Conditions for Operators in $\cl{AN}_{[k]}$}

In this subsection, we discuss the sufficient conditions for an operator (not necessarily positive) to be absolutely $[k]$-norming for every $k\in \bb N.$
We begin with a relatively easy proposition, the proof of which is trivial and thus omitted, that gives a sufficient condition for a positive diagonalizable operator to be in $\cl{N}_{[k]}$.  

\begin{prop}\label{SCforPDO}
Let $A\in \cl{B}(\cl{H})$ be a positive diagonalizable operator on the complex Hilbert space $\cl{H}$, $B=\{v_\beta:\beta \in \Lambda\}$ be an orthonormal basis of $\cl{H}$ corresponding to which $A$ is diagonalizable, and $k\in \bb{N}$. 
If there exists a subset $\{\beta_1,...,\beta_k\}\subseteq \Lambda$ of cardinality $k$ such that for every $j\in \{1,...,k\}, A(v_{\beta_j})=\lambda_j(A)v_{\beta_j}$, then $A\in\cl{N}_{[k]}$. If $\dim (\cl{H})=r<k$, then the existence of a subset  $\{\beta_1,...,\beta_r\}\subseteq \Lambda$ of cardinality $r$ is required with the condition that for every $j\in \{1,...,r\}, A(v_{\beta_j})=\lambda_j(A)v_{\beta_j}$, for the operator $A$ to be in $\cl{N}_{[k]}$. Here $\lambda_j(A)$ is as introduced in the Definition \ref{singular}.
\end{prop}


\begin{thm}[The Courant-Fischer Theorem for Positive Compact Operators]\label{CFTforPCompact}
Let $A\in \cl{B}(\cl{H})$ be a positive compact operator on $\cl{H}$ and let $\lambda_1(A)\geq \lambda_2(A)\geq...$ be its algebraically ordered eigenvalues (counting multiplicities) in nonincreasing sense. Let $k\in \bb{N}$ and let $S$ denote a subspace of $\cl{H}$.Then 

\begin{equation}\label{CF-A}
\lambda_1(A)= \max_{\{x:x\in \cl{H} \text{ and } \|x\|=1\}}\inner {Ax}{x}
\end{equation}

\begin{equation}\label{CF-B}
\lambda_{k+1}(A)= \min_{\{S:\dim (S)=k\}}\left( \max_{\{x:x\in S^\perp \text{ and } \|x\|=1\}}\inner {Ax}{x}\right)
\end{equation}
where the maximum in (\ref{CF-A}) is attained only at those eigenvectors of $A$ which correspond to $\lambda_1(A)$ and the minimum in (\ref{CF-B}) is attained when $S$ coincides with the $k$-dimensional subspace spanned by the eigenvectors $\{u_j:1\leq j\leq k\}$ of $A$ corresponding to the eigenvalues $\{\lambda_j:1\leq j\leq k\}$, so that 
\begin{equation}\label{CF-C}
\lambda_{k+1}(A)=\max_{\{x:x\in S^\perp \text{ and } \|x\|=1\}}\inner {Ax}{x}.
\end{equation}
\end{thm}

\begin{prop}\label{compacts-are-absolutely[k]norming}
If $T\in \cl{B}(\cl{H},\cl{K})$ is a compact operator, then $T\in\cl{AN}_{[k]}$ for every $k\in \bb{N}$.
\end{prop}

\begin{proof}
If $T$ is a compact operator from $\cl H$ to $\cl K$ then the restriction of $T$ to any closed subspace $\cl M$ is a compact operator from $\cl M$ to $\cl K.$ So it will be sufficient to prove that if $T$ is a compact operator then $T\in\cl N_{[k]}$ for each $k\in \bb{N}$.

The assertion is trivial if $\cl H$ is finite-dimensional; for then $|T|\in M_n(\bb C)$ for some $n\in \bb N$ is a positive diagonal matrix. We thus assume $\cl H$ to be infinite-dimensional. Let us fix $k\in \bb N$.  Since $|T|$ is positive compact operator the singular values of $T$ are precisely the eigenvalues of $|T|$, the Courant Fisher theorem (\ref{CFTforPCompact}) guarantees the existence of an orthonormal set $\{v_1,...,v_k\}\subseteq \cl H$ such that for every $j\in \{1,...,k\}, |T|v_j=\lambda_j(|T|)v_j$ which implies that $\|T\|_{[k]}=\sum_{j=1}^ks_j(T)=\sum_{j=1}^k\lambda_j(|T|)=\sum_{j=1}^k\||T|v_j\|=\sum_{j=1}^k\|Tv_j\|$ so that $T\in\cl N_{[k]}$. Since $k\in \bb N$ is arbitrary, it follows that $T\in\cl N_{[k]}$ for every $k\in \bb N.$
\end{proof}
\begin{lemma}\label{F}
If $F\in\cl{B}(\cl{H})$ is a self-adjoint finite-rank operator and $\alpha \geq 0$, then $\alpha I+F\in\cl{N}_{[k]}$ for every $k\in \bb{N}$.
\end{lemma}

\begin{proof}
The assertion is trivial if $\alpha =0$; for then $F$ is a compact operator which belongs to $\cl{N}_{[k]}$ for every $k\in \bb{N}$. We fix $k$ and assume that $\alpha > 0.$
There is no loss of generality in assuming that $\cl{H}$ is infinite-dimensional, for if it is not, then the operator is compact and hence belongs to $\cl{N}_{[k]}$. Let the range of $F$ be $m$-dimensional. It suffices to show that $|\alpha I +F|\in\cl{N}_{[k]}$ operator.\\

\emph{Case I}: If $k\leq m$.
 Since $F$ is self-adjoint, there exists an orthonormal basis $B = \{v_\beta:\beta\in \Lambda\}$ of $\cl{H}$ corresponding to which the matrix $M_B(F)$ is a diagonal matrix with $m$ nonzero real diagonal entries, say $\{\mu_1, \mu_2,...,\mu_m\}$ which are not necessarily distinct. 
There is, then, a subset $\{\beta_1,...,\beta_m\}\subseteq \Lambda$ of cardinality $m$ such that for every $j\in \{1,...,m\}$, we have $F(v_{\beta_j})=\mu_jv_{\beta_j}$.
  Clearly then, $M_B(|\alpha I+F|)$ is also a diagonal matrix with respect to the basis $B$ such that the spectrum $\sigma(|\alpha I +F|)$ of $|\alpha I +F|$ is given by $\sigma(|\alpha I +F|)=\sigma_p(|\alpha I +F|)=\{|\alpha + \mu_1|,...,|\alpha + \mu_m|, \alpha\}$ where $\alpha$ is the only eigenvalue with infinite multiplicity.
Let $i_1,...,i_m$ be a permutation of the integers $1,...,m$ which forces $|\alpha+\mu_{i_1}|\geq ...\geq|\alpha+\mu_{i_m}|$. Let us rename and denote by $v_{\beta_1},...,v_{\beta_m}$ the eigenvectors of $|\alpha I + F|$ corresponding to the eigenvalues  $|\alpha+\mu_{i_1}|, ...,|\alpha+\mu_{i_m}|$ respectively. We can further rename and denote each $|\alpha+\mu_{i_j}|$  by $|\alpha+\mu_{j}|$  so that we have $|\alpha+\mu_{1}|\geq ...\geq|\alpha+\mu_{m}|$ and a subset $\{\beta_1,...,\beta_m\}\subseteq \Lambda$ of cardinality $m$ such that for every $j\in \{1,...,m\}$, we have $|\alpha I + F|(v_{\beta_j})=|\alpha +\mu_j|v_{\beta_j}$.
Notice that $\sup\{|\alpha+\mu_1|,...,|\alpha+\mu_m|,\alpha\}
          =\max\{|\alpha+\mu_1|,...,|\alpha+\mu_m|,\alpha\}=\max\{|\alpha+\mu_1|,\alpha\}$. If $\max\{|\alpha+\mu_1|,...,|\alpha+\mu_m|,\alpha\}=\alpha$, then we have $\lambda_j(|\alpha I + F|)=\alpha$ for every $j \in \{1,...,k\}$, in which case, we can choose any $k$ distinct eigenvectors from $B\setminus\{v_{\beta_1},...,v_{\beta_m}\}$, say $\{w_{\beta_1},...,w_{\beta_k}\}$, so that $\||\alpha I + F|\|_{[k]}= k\alpha = \||\alpha I + F|w_{\beta_1}\|+...+\||\alpha I + F|w_{\beta_k}\|$ thereby implying that $|\alpha I +F|\in\cl{N}_{[k]}$. 
          
          Otherwise, we have $|\alpha+\mu_1|= \max\{|\alpha+\mu_1|,...,|\alpha+\mu_m|,\alpha\}$, so that  $\lambda_1(|\alpha I + F|)=|\alpha+\mu_1|$. Further, if $\alpha = \max\{|\alpha+\mu_2|,...,|\alpha+\mu_m|,\alpha\}$, then we have $\lambda_j(|\alpha I + F|)=\alpha$ for every $j \in \{2,...,k\}$ , in which case, we can choose the eigenvector $v_{\beta_1}$ and any $k-1$ distinct eigenvectors from $B\setminus\{v_{\beta_2},...,v_{\beta_m}\}$, say $\{w_{\beta_2},...,w_{\beta_k}\}$ so that $\||\alpha I + F|\|_{[k]}= |\alpha+\mu_1|+(k-1)\alpha = \||\alpha I + F|v_{\beta_1}\|+|\alpha I + F|w_{\beta_2}\|+...+\||\alpha I + F|w_{\beta_k}\|$ thereby implying that $|\alpha I +F|\in\cl{N}_{[k]}$; but if $\alpha \neq \max\{|\alpha+\mu_2|,...,|\alpha+\mu_m|,\alpha\}$, then we have $|\alpha+\mu_2|= \max\{|\alpha+\mu_2|,...,|\alpha+\mu_m|,\alpha\}$, so that $\lambda_1(|\alpha I + F|)=|\alpha+\mu_1|$, $
\lambda_2(|\alpha I + F|)=|\alpha +\mu_2|.$  Then, if $\alpha = \max\{|\alpha+\mu_3|,...,|\alpha+\mu_m|,\alpha\}$, we get $\lambda_j(|\alpha I + F|)=\alpha$ for every $j \in \{3,...,k\}$, in which case, we can choose the vectors $v_{\beta_1}, v_{\beta_2}$ and any $k-2$ distinct eigenvectors from $B\setminus\{v_{\beta_3},...,v_{\beta_m}\}$, say $\{w_{\beta_3},...,w_{\beta_k}\}$ which yields $|\alpha I + F|\in\cl{N}_{[k]}$. Carrying out this process of selecting appropriate eigenvectors from $B$ depending upon the value $\lambda_j(|\alpha I + F|)$ takes,  until we select $k$ of those,  establishes the fact that $|\alpha I + F|\in\cl{N}_{[k]}$.\\

\emph{Case II}: If $k\geq m$.
The proof goes the same way except that now we terminate the process once we find a subset $\{\beta_1,...,\beta_m\}\subseteq \Lambda$ of cardinality $m$ such that for every $j\in \{1,...,m\}$, we have $|\alpha I + F|(v_{\beta_j})=|\alpha +\mu_j|v_{\beta_j}$; for $\lambda_j(|\alpha I + F|)=0$ for $j> m.$ 

Since $k\in \bb N$ is arbitrary, it follows that $\alpha I + F\in\cl N_{[k]}$ for every $k\in \bb N.$
\end{proof}

\begin{prop}\label{alphaI+K+FisN_[k]}
Let $K\in\cl{B}(\cl{H})$ be a positive compact operator, $F\in\cl{B}(\cl{H})$ be a self-adjoint finite-rank operator, and $\alpha \geq 0$. Then $\alpha I+K+F\in\cl{N}_{[k]}$ for every $k\in \bb{N}$.
\end{prop}

\begin{proof}
The assertion is trivial if $\alpha =0$; for then $K+F$ is a compact operator which sits in $\cl{N}_{[k]}$ for every $k\in \bb{N}$. We fix $k$ and assume that $\alpha>0$.
There is no loss of generality in assuming that $\cl{H}$ is infinite-dimensional, for if it is not, then the operator is compact and thus belongs to $\cl{N}_{[k]}$. We can also assume, without loss of generality, that $\dim(\text{ran }  K)>n$ for every $n\in \bb{N}$, for if $K$ is a finite-rank operator then the operator is $[k]$-norming by the previous lemma. Due to the equivalence of $(1)$ and $(2)$ of Theorem \ref{TiffT*T_[k]}, it suffices to show that $|\alpha I +K+F|\in\cl{N}_{[k]}$.

 Notice that $K+F$ is a self-adjoint compact operator on $\cl{H}$ and thus there exists an orthonormal basis $B$ of $\cl{H}$ consisting entirely of eigenvectors of $K+F$ corresponding to which it is diagonalizable. From \cite[Lemma 4.8]{VpSp}, $K+F$ can have at most finitely many negative eigenvalues. Let $\{\nu_1, \nu_2,...,\nu_m\}$ be the set of all negative eigenvalues of $K+F$ with $\{v_1,v_2,...,v_m\}$ as the corresponding eigenvectors in basis $B$; and let $\{\mu_\beta:\beta\in \Lambda\}$ be the set of all remaining nonnegative eigenvalues of $K+F$ with $\{w_\beta:\beta\in \Lambda\}$ as the corresponding eigenvectors in $B$. We have $B:= \{v_1,v_2,...,v_n\}\cup \{w_\beta:\beta\in \Lambda\}$ and the matrix $M_B(K+F)$ of $K+F$ with respect to $B$ is given by
$$
K+F = \begin{bmatrix}
\nu_1 & & &\vdots & & &\\
& \ddots & &\vdots & &\text{\huge 0}  &\\
 &  & \nu_m &\vdots & & &\\
\hdots& \hdots & \hdots & \hdots&\hdots &\hdots&\hdots\\
&  &  &\vdots & \ddots& &\\
&  \text{\huge 0}&  &\vdots &  & \mu_\beta&\\
&  &  &\vdots &  & &\ddots \\
\end{bmatrix}
$$
Because $K+F$ is compact, the multiplicity of each nonzero eigenvalue is finite and there are at most countably many nonzero eigenvalues, counting multiplicities. In fact, we can safely assume that there are countably infinite nonzero eigenvalues (counting multiplicities) of $K+F$; for if there are only finitely many nonzero eigenvalues, then $K+F$ would be a self-adjoint finite-rank operator which, by Lemma \ref{F}, belongs to $\cl{N}_{[k]}$. With this observation, the set $\Gamma :=\Lambda \setminus\{\beta \in \Lambda:\mu_\beta =0\}$ is countably infinite and can be safely replaced by $\bb{N}$.  This essentially redefines the spectrum $\sigma (K+F)=\{\nu_n\}_{n=1}^m\cup \{\mu_n\}_{n=1}^\infty\cup\{0\}$ of $K+F$ and allows us to enumerate the positive eigenvalues $\{\mu_n\}_{n=1}^\infty$ in nonincreasing order $\mu_1\geq \mu_2\geq...$ so that each eigenvalue appears as many times as is its multiplicity. This ensures that the set of all positive eigenvalues of $K+F$ has been exhausted in the process of constructing the sequence $\{\mu_n\}_{n\in\bb{N}}$. That the sequence $\{\mu_n\}$ converges to $0$ is a trivial observation. So, $0$ is an accumulation point of the spectrum. However, it can also be an eigenvalue with infinite multiplicity. At this point, we rename and denote by $\{v_n\}_{n=1}^m $, $ \{w_n\}_{n=1}^\infty $, and $\{z_\beta\}_{\beta\in \Lambda \setminus \Gamma} $ the eigenvectors corresponding to the eigenvalues $\{\nu_n\}_{n=1}^m $, $\{\mu_n\}_{n=1}^\infty $, and $\{0\}$ respectively.
With the reordering, we now have $B:= \{v_n\}_{n=1}^m \cup \{w_n\}_{n=1}^\infty \cup\{z_\beta\}_{\beta\in \Lambda \setminus \Gamma} $ and the matrix $M_B(K+F)$ of $K+F$ with respect to $B$ is given by
$$
K+F = \begin{bmatrix}
\nu_1 & & &\vdots & & &&\vdots&&\\
& \ddots & &\vdots & &\text{\huge 0}  &&\vdots&&\\
 &  & \nu_m &\vdots & & &&\vdots&&\\
\hdots& \hdots & \hdots & \hdots&\hdots &\hdots&\hdots&\vdots&\text{\huge 0}\\
&  &  &\vdots & \ddots& &&\vdots&&\\
&  \text{\huge 0}&  &\vdots &  & \mu_n&&\vdots&&\\
&  &  &\vdots &  & &\ddots &\vdots&&\\
\hdots&\hdots&\hdots&\hdots&\hdots&\hdots&\hdots&\hdots&\hdots&\hdots\\
&&&&&&&\vdots&&\\
&&&\text{\huge 0}&&&&\vdots&\text{\huge 0}&\\
&&&&&&&\vdots&&\\
\end{bmatrix}
$$

We now consider the operator $|\alpha I +K+F|$. With respect to the basis $B$, the matrix $M_B(|\alpha I +K+F|)$ of $|\alpha I +K+F|$  is given by

$$
 \begin{bmatrix}
|\alpha +\nu_1| &	 & & \vdots & & & & \vdots & & & \\
& \ddots & &\vdots & &\text{\huge 0}  &&\vdots & & &\\
 &  & |\alpha +\nu_m| &\vdots & & &&\vdots&& & \\
\hdots& \hdots & \hdots & \hdots&\hdots &\hdots&\hdots&\vdots & &\text{\huge 0}&\\
&  &  &\vdots & \ddots& &&\vdots&&&\\
&  \text{\huge 0}&  &\vdots &  & \alpha +\mu_n&&\vdots&& &\\
&  &  &\vdots &  & &\ddots &\vdots&& &\\
\hdots&\hdots &\hdots &\hdots &\hdots&\hdots&\hdots&\hdots&\hdots &\hdots&\hdots \\
&&&&&&&\vdots&\ddots& & \\
&&& \text{\huge 0} &&&&\vdots &&\alpha & \\
&&&&&&&\vdots&& &\ddots \\
\end{bmatrix}.
$$

Observe that $\sigma_{e}(|\alpha I + K + F|)$ of $|\alpha I + K + F|$ is the singleton $\{\alpha\}$ and that for any given $k\in \bb{N} $, $\lambda_j(|\alpha I + K + F|) > \alpha $ for every $j\in \{1,...,k\}$. In fact, $\lambda_j(|\alpha I + K + F|)\in \{|\alpha +\nu_n|\}_{n=1}^m \cup\{\alpha+\mu_n\}_{n=1}^\infty $ for every $j\in \{1,...,k\}$. It then immediately follows that there exist $k$ orthogonal eigenvectors in $\{v_n\}_{n=1}^m \cup \{w_n\}_{n=1}^\infty$ with $\lambda_j(|\alpha I + K + F|), j\in \{1,...,k\}$ being their correspoding eigenvalues. This proves the assertion. Since $k\in \bb N$ is arbitrary in the above proof, the propostion holds for every $k\in \bb N$ and thus an operator of the form $\alpha I +K+F$ belongs to $\cl N_{[k]}$ for every $k\in \bb N.$ 
\end{proof}

This result is the key to the following theorem.

\begin{thm}\label{backward_[k]}
Let $K\in\cl{B}(\cl{H})$ be a positive compact operator, $F\in\cl{B}(\cl{H})$ be a self-adjoint finite-rank operator, and $\alpha \geq 0$. Then $\alpha I+K+F\in\cl{AN}_{[k]}$ for every $k\in \bb{N}$.
\end{thm}

\begin{proof}
Let us define $T:= \alpha I+K+F$ so that we have $|T|= |\alpha I+K+F|$ and  $|T|^*|T| =|T|^2= (\alpha I+K+F)^2 = (\alpha^2 I)+(2\alpha K+K^2) + (2\alpha F + FK +KF +F^2) = \beta I+\tilde K + \tilde F$ where $\beta=\alpha^2 \geq 0$, $\tilde K = 2\alpha K+K^2$ and $\tilde F = 2\alpha F + FK +KF +F^2$ are respectively positive compact and self-adjoint finite-rank operators. Further, let $\cl M$ be an arbitrary nonempty closed linear subspace of the Hilbert space $\cl H$ and $V_{\cl M}:\cl M\longrightarrow \cl H$ be the inclusion map from $\cl M$ to $\cl H$ defined as $V_{\cl M}(x)=x$ for each $x\in \cl M$.
We fix $k\in \bb N$ and observe that
\begin{multline*}
 |T|V_\cl{M} \in \cl{N}_{[k]} \iff  (|T|V_\cl{M})^*(|T|V_\cl{M}) \in \cl{N}_{[k]}\\ \iff V_\cl{M}^*(|T|^*|T|)V_\cl{M} \in \cl{N}_{[k]} \iff  V_\cl{M}^*( \beta I+\tilde K + \tilde F)V_\cl{M} \in \cl{N}_{[k]}.
\end{multline*}
It suffices to show that $V_\cl{M}^*( \beta I+\tilde K + \tilde F)V_\cl{M}\in\cl{N}_{[k]}$; for then, since $\cl{M}$ is arbitrary, it immediately follows from lemma \ref{Trick_[k]}  that $|T|\in\cl{AN}_{[k]}$ and so does $T$ due to the equivalence of $(1)$ and $(2)$ of Theorem \ref{TiffT*T_[k]}. To this end, notice that $V_\cl{M}^*( \beta I+\tilde K + \tilde F)V_\cl{M}:\cl{M}\longrightarrow \cl{M}$ is an operator on $\cl{M}$ and 
$$
V_\cl{M}^*( \beta I+\tilde K + \tilde F)V_\cl{M} = V_\cl{M}^*\beta IV_\cl{M} +V_\cl{M}^*\tilde K V_\cl{M} +V_\cl{M}^* \tilde FV_\cl{M} = \beta I_\cl{M} +\tilde K_\cl{M}+\tilde F_\cl{M}
$$
is the sum of a nonnegative scalar multiple of Identity, a positive compact operator and a self-adjoint finite-rank operator on a Hilbert space $\cl{M}$ which, by the preceding proposition, belongs to $\cl{N}_{[k]}$. This proves the assertion. Moreover, since $k\in \bb N$ is arbitrary, the result holds for every $k\in \bb N$ and thus an operator of the above form belongs to $\cl {AN}_{[k]}$ for every $k\in \bb N.$
\end{proof}

We are now ready to establish the spectral theorem for positive operators that belong to $\cl {AN}_{[k]}$ for every $k\in \bb N.$
Note that the Theorem \ref{backward_[k]} we just proved ------ that for every $\alpha\geq 0, \alpha I + K +F\in\cl{AN}_{[k]}$ where $K$ and $F$ are respectively positive compact and self-adjoint finite-rank operators ------ is the stronger version of the backward implication of our spectral theorem for positive $\cl{AN}_{[k]}$ operators. If the operator $ \alpha I + K +F$ is also positive then the implication can be reversed and the two conditions are equivalent. This is what the next theorem states.

\begin{thm}[Spectral Theorem for Positive Operators in $\cl{AN}_{[k]}$]\label{SpThPAN_[k]}
Let $\cl{H}$ be a complex Hilbert space of arbitrary dimension and $P$ be a positive operator on $\cl{H}$. Then the following statements are equivalent.
\begin{enumerate}
\item $P\in\cl{AN}_{[k]}$ for every $k\in \bb N$.
\item $P\in\cl{AN}_{[k]}$ for some $k\in \bb N$.
\item $P$ is of the form $P= \alpha I +K+F$, where $\alpha \geq 0, K$ is a positive compact operator and $F$ is self-adjoint finite-rank operator.
\end{enumerate} 
\end{thm}

\begin{proof}
$(1)$ implies $(2)$ trivially. $(2)$ implies $(3)$ is due to Theorem \ref{forward_[k]}. $(1)$ follows from $(3)$ due to Theorem \ref{backward_[k]}.
\end{proof}

\section{Spectral Characterization of Operators in $\cl{AN}_{[k]}$}
We first state the polar decomposition theorem. Let $W$ be a subspace of $\cl H$. We use clos$[W]$ to denote the norm closure of $W$ in $\cl H$.
\begin{thm}[Polar Decomposition Theorem]\cite[Page 15]{ConwayOT}\label{PDT}
Let $\cl{H}, \cl{K}$ be  complex Hilbert spaces. If $T\in \cl{B}(\cl{H},\cl{K})$, then there exists a  unique partial 
isometry $U:\cl{H}\longrightarrow \cl{K}$ with final space $\text{clos}[\text{ran }{T}]$ and initial space $\text{clos}[\text{ran }{|T|}]$ such that $T=U|T|$ and $|T|=U^*T$. If $T$ is invertible, then $U$ is unitary.
\end{thm}

\begin{thm}[Spectral Theorem for Operators in $\cl{AN}_{[k]}$]\label{SpThAN}
Let $\cl{H}$ and $\cl{K}$ be complex Hilbert spaces of arbitrary dimensions, let $T\in \cl{B}(\cl{H},\cl{K})$ and let $T=U|T|$ be its polar decomposition. 
Then the following statements are equivalent.
\begin{enumerate}
\item $T\in\cl{AN}_{[k]}$ for every $k\in \bb N$.
\item $T\in\cl{AN}_{[k]}$ for some $k\in \bb N$.
\item $|T|$ is of the form $|T|= \alpha I +F+K$, where $\alpha \geq 0, K$ is a positive compact operator and $F$ is self-adjoint finite-rank operator. 
\end{enumerate} 
\end{thm}

\begin{proof}
The proof follows from the Proposition \ref{Tiff|T|_[k]}, the polar decomposition theorem and the spectral theorem for positive $\cl{AN}_{[k]}$ operators.
\end{proof}

\section{The Classes $\cl N_{[\pi,k]}$ and $\cl {AN}_{[\pi,k]}$}

\begin{defn}[Weighted Ky Fan $\pi, k$-norm]
Let $(\pi_j)_{j\in \bb N}$ be a nonincreasing sequence of positive numbers with $\pi_1=1$ and let $k\in \bb N$. The weighted Ky Fan $\pi, k$-norm $\|\cdot\|_{[\pi, k]}$ of an operator $T\in \cl{B}(\cl{H},\cl{K})$ is defined to be the weighted sum of the $k$ largest singular values of $T$, the weights being the first $k$ terms of the sequence $(\pi_j)_{j\in \bb N}$, that is, 
$$
\|T\|_{[\pi, k]}=\sum_{j=1}^k\pi_js_j(T).
$$
The weighted Ky Fan $\pi, k$-norm on $\cl{B}(\cl{H},\cl{K})$ is, indeed, norm; the proof is similar to that of the Ky Fan $k$-norm. It can be easily shown that on $\cl B(\cl H)$ it is a symmetric norm.
\end{defn}

\begin{remark}
If we choose $(\pi_j)_{j\in \bb N}$ to be a constant sequence with each term equals to $1$, then weighted Ky Fan $\pi, k$-norm $\|\cdot\|_{[\pi, k]}$ is simply the Ky Fan $k$-norm $\|\cdot\|_{[k]}$. If in addition, we also choose $k=1$, we get the operator norm.
\end{remark}

\begin{defn}\label{N_[pi,k]Def}
Let $(\pi_j)_{j\in \bb N}$ be a nonincreasing sequence of positive numbers with $\pi_1=1$ and let $k\in \bb N$. An operator $T\in \cl{B}(\cl{H},\cl{K})$ is said to be \emph{$[\pi, k]$-norming} if there are orthonormal elements $x_1,...,x_k \in \cl{H}$ such that $\|T\|_{[\pi, k]}=\|Tx_1\|+\pi_2\|Tx_2\|+...+\pi_k\|Tx_k\|$.
 If $\dim(\cl{H})=r<k$, we define $T$ to be \emph{$[\pi, k]$-norming} if there exist orthonormal elements $x_1,...,x_r\in \cl{H}$ such that $\|T\|_{[\pi, k]}=\|Tx_1\|+\pi_2\|Tx_2\|+...+\pi_r\|Tx_r\|$. We let $\cl N_{[\pi, k]}(\cl H, \cl K)$ denote the set of $[\pi, k]$-norming operators in $\cl B(\cl H,\cl K)$.
\end{defn}

\begin{defn}\label{AN_[pi,k]Def}
Let $(\pi_j)_{j\in \bb N}$ be a nonincreasing sequence of positive numbers with $\pi_1=1$ and let $k\in \bb N$. An operator $T\in \cl{B}(\cl{H},\cl{K})$ is said to be \emph{absolutely $[\pi,k]$-norming} if for every nontrivial closed subspace $\cl{M}$ of $\cl{H}$, $T|_{\cl{M}}$ is  $[\pi,k]$-norming. We let $\cl{AN}_{[\pi, k]}(\cl H, \cl K)$ denote the set of absolutely $[\pi,k]$-norming operators in $\cl B(\cl H,\cl K)$. Note that $\cl{AN}_{[\pi, k]}(\cl H,\cl K)\subseteq \cl{N}_{[\pi, k]}(\cl H,\cl K)$.
\end{defn}


\begin{remark}
We let $\Pi$ denote the set of all nonincreasing sequences of positive numbers with their first term equal to $1$. Every operator on a finite-dimensional Hilbert space is $[\pi, k]$-norming for any $\pi\in \Pi$ and for any $k\in \bb N$. However, this is not true when the Hilbert space in question is not finite-dimensional. The operator in Example \ref{Counterex} is one such operator. There exists $\tilde \pi =(1,1,1,1,...)$ such that $A\in \cl N_{[\tilde \pi, 2]}(\cl H,\cl K)$ but $A\notin \cl N_{[\tilde \pi, 3]}(\cl H,\cl K)$.
\end{remark}

 We now mention that Lemma \ref{Trick_[k]} and Proposition \ref{Tiff|T|_[k]} carries over word for word to operators in $\cl{AN}_{[\pi, k]}(\cl H,\cl K)$.

\begin{lemma}\label{Trick_[pi, k]}
For a closed linear subspace $\cl{M}$ of a complex Hilbert space $\cl{H}$ let $V_{\cl{M}}:\cl{M}\longrightarrow \cl{H}$ be the inclusion map from $\cl{M}$ to $\cl{H}$ defined as $V_{\cl{M}}(x) = x$ for each $x\in \cl{M}$ and let $T\in \cl B(\cl H, \cl K)$.
For any sequence $\pi\in \Pi$ and for any $k\in\bb{N}$, $T\in \cl{AN}_{[\pi, k]}(\cl H,\cl K)$ if and only if for every nontrivial closed linear subspace $\cl{M}$ of $\cl{H}$, $TV_{\cl{M}}\in\cl{N}_{[\pi, k]}(\cl H,\cl K)$.
\end{lemma}

\begin{proof}
From the proof of Lemma \ref{Trick_[k]} we know that 
for any given nontrivial closed subspace $\cl{M}$ of $\cl{H}$, the maps $TV_\cl{M}$ and $T|_\cl{M}$ are identical and so are their singular values which implies $\|TV_\cl{M}\|_{[\pi, k]} = \|T|_\cl{M}\|_{[\pi, k]}$.
We next assume that $T\in\cl{AN}_{[\pi, k]}(\cl H,\cl K)$ and prove the forward implication. Let $\cl{M}$ be an arbitrary but fixed nontrivial closed subspace of $\cl{H}$. Either $\dim(\cl{M})=r<k$, in which case, there exist orthonormal elements $x_1,...,x_r \in \cl{M}$ such that $\|T|_\cl{M}\|_{[\pi, k]} = \|T|_{\cl{M}}x_1\|+\pi_2\|T|_{\cl{M}}x_2\|+...+\pi_r\|T|_{\cl{M}}x_r\|$ which means that there exist orthonormal elements $x_1,...,x_r \in \cl{M}$ such that 
$\|TV_\cl{M}\|_{[\pi, k]} =\|T|_\cl{M}\|_{[\pi, k]}=\|T|_{\cl{M}}x_1\|+ \pi_2\|T|_{\cl{M}}x_2\|+...+ \pi_r\|T|_{\cl{M}}x_r\| = \|TV_\cl{M} x_1\|+ \pi_2\|TV_\cl{M} x_2\|+...+\pi_r\|TV_\cl{M} x_r\|$ proving that $TV_{\cl{M}}\in\cl{N}_{[\pi, k]}(\cl M,\cl K)$, or $\dim(\cl{M})\geq k$, in which case, there exist orthonormal elements $x_1,...,x_k \in \cl{M}$ such that $\|T|_\cl{M}\|_{[\pi, k]} = \|T|_{\cl{M}}x_1\|+\pi_2\|T|_{\cl{M}}x_2\|...+\pi_k\|T|_{\cl{M}}x_k\|$ which means that there exist orthonormal elements $x_1,...,x_k \in \cl{M}$ such that $\|TV_\cl{M}\|_{[\pi, k]} =\|T|_\cl{M}\|_{[\pi, k]}=\|T|_{\cl{M}}x_1\|+\pi_2\|T|_{\cl{M}}x_2\|+...+\pi_k\|T|_{\cl{M}}x_k\| = \|TV_\cl{M} x_1\|+\pi_2\|TV_\cl{M} x_2\|+...+ \pi_k\|TV_\cl{M} x_k\|$ proving that $TV_{\cl{M}}\in\cl{N}_{[\pi, k]}(\cl M,\cl K)$. Since $\cl{M}$ is arbitrary, it follows that $TV_\cl{M}\in\cl{N}_{[\pi, k]}(\cl M,\cl K)$ for every $\cl M$.

We complete the proof by showing that $T\in\cl{AN}_{[\pi, k]}(\cl H,\cl K)$ operator if $TV_\cl{M}\in\cl{N}_{[\pi, k]}(\cl M,\cl K)$ for every nontrivial closed subspace $\cl{M}$ of $\cl{H}$. We again fix $\cl{M}$ to be an arbitrary nontrivial closed subspace of $\cl{H}$. Since $TV_\cl{M}\in\cl{N}_{[\pi, k]}(\cl M,\cl K)$, either $\dim(\cl{M})=r<k$, in which case, there exist orthonormal elements $x_1,...,x_r \in \cl{M}$ such that $\|TV_\cl{M}\|_{[\pi, k]}=\|TV_\cl{M} x_1\|+\pi_2\|TV_\cl{M} x_2\|+...+\pi_r\|TV_\cl{M} x_r\|$ which means that there exist orthonormal elements $x_1,...,x_r \in \cl{M}$ such that
 $\|T|_\cl{M}\|_{[\pi, k]}=\|TV_\cl{M}\|_{[\pi, k]}=\|TV_\cl{M} x_1\|+\pi_2\|TV_\cl{M} x_2\|+...+ \pi_r\|TV_\cl{M} x_r\|=\|T|_\cl{M} x_1\|+\pi_2\|T|_\cl{M} x_1\|+...+ \pi_r\|T|_\cl{M} x_r\|$ proving that $T|_{\cl{M}}\in\cl{N}_{[\pi, k]}(\cl M,\cl K)$, or  $\dim(\cl{M})\geq k$, in which case, there exist  orthonormal elements $x_1,...,x_k \in \cl{M}$ such that $\|TV_\cl{M}\|_{[\pi, k]}=\|TV_\cl{M}x_1\|+\pi_2\|TV_\cl{M}x_2\|+...+\pi_k\|TV_\cl{M}x_k\|$ which means that there exist orthonormal elements $x_1,...,x_k \in \cl{M}$ such that $\|T|_\cl{M}\|_{[\pi, k]}=\|TV_\cl{M}\|_{[\pi, k]}=\|TV_\cl{M}x_1\|+\pi_2\|TV_\cl{M}x_2\|+...+\pi_k\|TV_\cl{M}x_k\|=\|T|_\cl{M}x_1\|+\pi_2\|T|_\cl{M}x_2\|+...+\pi_k\|T|_\cl{M}x_k\|$ proving that $T|_{\cl{M}}\in\cl{N}_{[\pi, k]}(\cl M,\cl K)$. Because $\cl{M}$ is arbitrary, $T\in\cl{AN}_{[\pi, k]}$. It is worthwhile noticing that since $\pi\in \Pi$ and $k\in \bb{N}$ are arbitrary, the assertion holds for every sequence $\pi\in \Pi$ and for every $k\in \bb{N}.$

\end{proof}

 \begin{prop}\label{Tiff|T|_[pi, k]}
Let $\cl{H}$ and $\cl{K}$ be complex Hilbert spaces and let $T\in \cl{B}(\cl{H},\cl{K})$. Then for every $\pi\in \Pi$ and for every $k\in \bb{N}$, $T\in\cl{AN}_{[\pi, k]}(\cl H,\cl K)$ if and only if $|T|\in\cl{AN}_{[\pi, k]}(\cl H,\cl K)$.
 \end{prop}
 
 \begin{proof}
 Let $\cl{M}$ be an arbitrary nontrivial closed subspace of $\cl{H}$ and let $V_{\cl{M}}:\cl{M}\longrightarrow \cl{H}$ be the inclusion map from $\cl{M}$ to $\cl{H}$ defined as $V_{\cl{M}}(x) = x$ for each $x\in \cl{M}$. From the proof of the Proposition \ref{Tiff|T|_[k]}, it is easy to see that $|TV_{\cl{M}}|=|\,|T|V_{\cl{M}}\,|$. Consequently, for every $j$, $\lambda_j(|TV_{\cl{M}}|)=\lambda_j(||T|V_{\cl{M}}|)$ and hence $s_j(TV_{\cl{M}})=s_j(|T|V_{\cl{M}})$. This implies that for every $\pi\in \Pi$ and for each $k\in \bb{N}$, we have
$$
\|TV_{\cl{M}}\|_{[\pi, k]}=\||T|V_{\cl{M}}\|_{[\pi, k]}.
$$
Furthermore, for every $x\in \cl{H}$, we have $\|TV_{\cl{M}}x\| = \||T|V_{\cl{M}}x\|$.
Since $\cl{M}$ is arbitrary, by Lemma \ref{Trick_[pi, k]} the assertion follows.
 \end{proof}
\begin{remark}
For the remaining part of this article, we use $\cl{N}_{[\pi, k]}$ and $\cl{AN}_{[\pi, k]}$ for $\cl{N}_{[\pi, k]}(\cl{H},\cl{K})$ and $\cl{AN}_{[\pi, k]}(\cl{H},\cl{K})$ respectively, as long as the domain and codomain spaces are obvious from the context.
\end{remark}

\section{Spectral Characterization of Positive Operators in $\cl{AN}_{[\pi,k]}$}

This section discusses the necessary and sufficient conditions for a positive operator on complex Hilbert space of arbitrary dimension to be absolutely $[\pi, k]$-norming for every $\pi \in \Pi$ and for every $k \in \bb N$. 
We first mention an easy proposition, the proof of which is left to the reader.

\begin{prop}\label{EasyPeasy}
Let $A\in \cl{B}(\cl{H})$ be a positive operator. Then the following statements are equivalent.
\begin{enumerate}
\item $A\in\cl{N}.$
\item $A\in\cl{N}_{[\pi, 1]}$ for some $\pi\in \Pi$.
\item $A\in\cl{N}_{[\pi, 1]}$ for every $\pi\in \Pi$.
\end{enumerate} 
\end{prop}

The following result may be considered as an analogue of Proposition \ref{1.2} and can be proved in much the same way. 

\begin{prop}\label{Analogue-1-1.2}
Let $\pi\in \Pi$ and $A\in \cl B(\cl H)$ be a positive operator. If $s_{m+1}(A)\neq s_m(A)$ for some $m\in \bb N$, then $A\in \cl N_{[\pi, m]}$. Moreover, in this case, $A\in \cl N_{[\pi,m+1]}$ if and only if $s_{m+1}(A)$ is an eigenvalue of $A$.
\end{prop}
\begin{proof}
It is easy to see that for every $j\in \{1,...,m\}, s_j(A)\notin \sigma_{\text{e}}(A)$. Then the set $\{s_1(A),...,s_m(A)\}$ consists of eigenvalues (not necessarily distinct) of $A$, each having finite multiplicity. This guarantees the existence of an orthonormal set $\{v_1,...,v_m\}\subseteq K\subseteq \cl H$ such that $Av_j=s_j(A)v_j$ so that $\|Av_j\|=s_j(A)$ and thus $\|A\|_{[\pi, m]}=\sum_{j=1}^m\pi_js_j(A)=\sum_{j=1}^m\pi_j\|Av_j\|$, where $K$ is the closure of the joint span of the eigenspaces corresponding to the eigenvalues $\{s_1(A),...,s_m(A)\}$, which implies that $A\in \cl N_{[\pi, m]}$. Furthermore, we observe that if there exists any orthonormal set $\{w_1,...,w_m\}$ of $m$ vectors in $\cl H$ such that $\sum_{i=1}^m\pi_i\|Aw_i\|=\sum_{j=1}^m\pi_js_j(A)$, then this set has to be contained in $K$. Note that $K^\perp$ is invariant under $A$ and hence    $A|_{K^\perp}:K^\perp\longrightarrow K^\perp$, viewed as an operator on $K^\perp$, is positive. Since $s_{m+1}(A)\neq s_m(A)$, it follows that $s_{m+1}(A)$ is an eigenvalue of $A\iff$ $s_{m+1}(A)$ is an eigenvalue of $A|_{K^\perp}:K^\perp\longrightarrow K^\perp$ which is possible $\iff A|_{K^\perp}:K^\perp\longrightarrow K^\perp$, viewed as an operator on $K^\perp$, belongs to $\cl N$, that is, there is a unit vector $x\in K^\perp$ such that $\|Ax\|=s_{m+1}(A)$, which is possible if and only if $\pi_{m+1}\|Ax\|=\pi_{m+1}s_{m+1}(A)\iff A\in\cl N_{[\pi,m+1]}$. This proves the assertion. 
\end{proof}

\subsection{Necessary Conditions for Positive Operators in $\cl{AN}_{[\pi,k]}$}

The last proposition in the previous section can be used to establish results analogous to Propositions \ref{1.6}, \ref{1.7} and \ref{1.8} (see \ref{analogous-1-1.6}, \ref{analogous-1-1.7}, and \ref{analogous-1-1.8} respectively) for a given $\pi \in \Pi$ and a given $k\in \bb N$. The proofs for these are similar so we only prove one of these results and leave the rest for the reader.
\begin{prop}\label{analogous-1-1.6}
Let $A\in \cl{B}(\cl{H})$ be a positive operator, $\pi\in \Pi$, and $k\in \bb{N}$. If $A\in \cl N_{[\pi, k]}$, then $s_1(A),...,s_k(A)$ are eigenvalues of $A$.
\end{prop}

\begin{proof}
The proof is by contrapositive. Assuming that at least one of the elements from the set $\{s_1(A),..., s_k(A)\}$ is not an eigenvalue of $A$, we show that $A\notin \cl N_{[\pi, k]}$. Suppose that $s_1(A)$ is not an eigenvalue of $A$. Then it must be an accumulation point of the spectrum of $A$ in which case none of the singular values of $A$ is an eigenvalue of $A$ and that  $s_j(A)=s_1(A)$ for every $j\geq 2$. Since $s_1(A)=\|A\|$, it follows from \cite[Theorem 2.3]{VpSp} that $A\notin\cl N$ which means that for every $x\in \cl H, \|x\|=1$, we have $\|Ax\|<\|A\|=s_1(A)=s_2(A)=...=s_k(A)$. Consequently, for every orthonormal set $\{x_1,...,x_k\}\subseteq \cl H$ we have 
$\sum_{j=1}^k\pi_j\|Ax_j\|<\sum_{j=1}^k\pi_js_j(A)$ so that $A\notin \cl N_{[\pi, k]}$.  

Next suppose that $s_1(A)$ is an eigenvalue of $A$ but $s_2(A)$ is not. Clearly then $s_1(A)$ is an eigenvalue with multiplicity $1$, $s_2(A)\neq s_1(A)$ and $s_j(A)=s_2(A)$ for every $j\geq 3$ in which case Proposition \ref{Analogue-1-1.2} ascertains that $A\in \cl N_{[\pi, 1]}$ but $A\notin \cl N_{[\pi, 2]}$. This implies that there exists $y_1\in \cl H$ with $\|y_1\|=1$ such that $\|Ay_1\|=\|A\|$ and for every $y\in \text{span}\{y_1\}^\perp$ with $
\|y\|=1$ we have $\|Ay\|<s_2(A)$ which in turn implies that for every orthonormal set $\{y_2,...y_k\}\subseteq \text{span}\{y_1\}^\perp$ we have $\sum_{j=2}^k\pi_j\|Ay_j\|<\sum_{j=2}^k\pi_js_j(A)$ so that  $\sum_{j=1}^k\pi_j\|Ay_j\|<\sum_{j=1}^k\pi_js_j(A)$ which implies that $A\notin \cl N_{[\pi, k]}$.

If $s_1(A), s_2(A)$ are eigenvalues of $A$ but $s_3(A)$ is not, then we have 
$s_3(A)\neq s_2(A)$ and $s_j(A)=s_3(A)$ for every $j\geq 4$ in which case Proposition \ref{Analogue-1-1.2} asserts that $A\in \cl N_{[\pi, 2]}$ but $A\notin \cl N_{[\pi,3]}$. Consequently, there exists an orthonormal set $\{z_1,z_2\}\subseteq \cl H$ such that $\|Tz_1\|+\pi_2\|Tz_2\|=\|T\|_{[\pi, 2]}$ and that for every unit vector $z\in \text{span}\{z_1,z_2\}^\perp$ we have $\|Tz\|<s_3(A)$ which in turn implies that for every orthonormal set $\{z_3,...z_k\}\subseteq \text{span}\{z_1,z_2\}^\perp$ we have $\sum_{j=3}^k\pi_j\|Az_j\|<\sum_{j=3}^k\pi_js_j(A)$. It then follows that $\sum_{j=1}^k\pi_j\|Az_j\|<\sum_{j=1}^k\pi_js_j(A)$ for every orthonormal set $\{z_1,...,z_k\}\subseteq \cl H$ which implies that $A\notin \cl N_{[\pi, k]}$.

  If we continue in this way, we can show at every step that $A\notin \cl N_{[\pi, k]}$. We conclude the proof by discussing the final case when $s_1(A),...,s_{k-1}(A)$ are all eigenvalues of $A$ but $s_k(A)$ is not in which case $s_{k}(A)\neq s_{k-1}(A)$ and thus Proposition \ref{Analogue-1-1.2} again implies that $A\notin \cl N_{[\pi, k]}$. This exhausts all the possibilities and the assertion is thus proved contrapositively.
\end{proof}

\begin{prop}\label{analogous-1-1.7}
Let $A\in \cl{B}(\cl{H})$ be a positive operator, $\pi\in \Pi$, and $k\in \bb{N}$. If $s_1(A),...,s_k(A)$ are mutually distinct eigenvalues of $A$,  then there exists an orthonormal set $\{v_1,...,v_k\}\subseteq \cl H$ such that $Av_j=s_j(A)v_j$ for every $j\in \{1,...,k\}$. Thus $A\in\cl N_{[\pi, k]}$.
\end{prop}

\begin{prop}\label{analogous-1-1.8}
Let $A\in \cl{B}(\cl{H})$ be a positive operator, $\pi\in \Pi$, $k\in \bb N$ and let $s_1(A),...,s_k(A)$ be the first $k$ singular values of $A$ that are also the eigenvalues of $A$ and are not necessarily distinct. Then either $s_1(A)=...=s_k(A)$, in which case, $A\in \cl N_{[\pi, k]}$ if and only if the multiplicity of $\alpha:=s_1(A)$ is at least $k$; or there exists $t\in \{2,...,k\}$ such that $s_{t-1}(A)\neq s_t(A)=s_{t+1}(A)=...=s_k(A)$, in which case, $A\in \cl N_{[\pi,k]}$ if and only if the multiplicity of $\beta:=s_t(A)$ is at least $k-t+1$.
\end{prop}

The above propositions leads us to prove the following result that adds another equivalent condition to the Theorem \ref{PositiveN_{[k]}}.
\begin{thm}\label{PositiveN_{[pi, k]}}
Let $A\in \cl{B}(\cl{H})$ be a positive operator, $\pi \in \Pi$, and $k\in \bb{N}$. Then the following statements are equivalent.
\begin{enumerate}
\item $A\in\cl{N}_{[k]}.$
\item $A\in\cl{N}_{[\pi, k]}$.
\item $s_1(A),...,s_k(A)$ are eigenvalues of $A$ and there exists an orthonormal set $\{v_1,...,v_k\}\subseteq \cl H$ such that $Av_j=s_j(A)v_j$ for every $j\in \{1,...,k\}.$
\end{enumerate} 
\end{thm}

\begin{proof}
$(1)\iff(3)$ has been established in Theorem \ref{PositiveN_{[k]}} and $(3)\implies (2)$ is trivial.  To establish $(2)\implies (3)$, note that by the
Proposition \ref{analogous-1-1.6}, $s_1(A),..., s_k(A)$ are all eigenvalues of $A$. If $s_1(A),..., s_k(A)$ are mutually distinct eigenvalues, then by Proposition \ref{analogous-1-1.7} there exists an orthonormal set $\{v_1,...,v_k\}\subseteq \cl H$ such that $Av_j=s_j(A)v_j$ for every $j\in \{1,...,k\}$. However, if $s_1(A),...,s_k(A)$ are all eigenvalues but not necessarily distinct then also the existence of an orthonormal set $\{v_1,...,v_k\}\subseteq \cl H$ with $Av_j=s_j(A)v_j$ for every $j\in \{1,...,k\}$ is guaranteed by the Proposition \ref{analogous-1-1.8}. This completes the proof.
\end{proof}
The above theorem leads us immediately to the following rather obvious corollary.

\begin{cor}\label{N_{[pi, k+1]}impliesN_{[pi, k]}}
Let $A\in \cl B(\cl H)$ be a positive operator, $\pi \in \Pi$, and $k\in \bb{N}$. 
\begin{enumerate}
\item If $A\in \cl N_{[\pi, k+1]}$, then $A\in \cl N_{[\pi, k]}$.
\item If $A\in \cl N_{[\pi, k]}$, then $A\in \cl N$.
\end{enumerate}
\end{cor}

Theorem \ref{TiffT*T_[k]} extends word for word to the family $\cl N_{[\pi, k]}$ (see \ref{TiffT*T_[pi, k]}) and Theorem \ref{AN_{[k+1]}impliesAN_[k](general)} alongwith the Corollary \ref{AN_{[k+1]}impliesAN_[k]} extend to the family $\cl{AN}_{[\pi, k]}$ (see \ref{AN_{[pi, k+1]}impliesAN_[pi, k](general)} and \ref{AN_{[pi, k+1]}impliesAN_[pi, k]} respectively). 

\begin{thm}\label{TiffT*T_[pi, k]}
Let $\cl{H}$ and $\cl{K}$ be complex Hilbert spaces, $T\in \cl{B}(\cl{H},\cl{K})$, $\pi \in \Pi$, and $k\in \bb{N}$. Then the following statements are equivalent.
\begin{enumerate}
\item $T\in\cl{N}_{[\pi, k]}$.
\item $|T|\in\cl{N}_{[\pi, k]}$.
\item $T^*T\in\cl{N}_{[\pi, k]}$. 
\end{enumerate}
\end{thm}

\begin{proof}
It suffices to establish $(1)\iff (2)$; for then $|T|$ and $T^*T$ are positive and since the sets $\cl{N}_{[k]}$ and $\cl{N}_{[\pi, k]}$ coincide for positive operators,  Theorem \ref{TiffT*T_[k]} yields the equivalence of $(2)$ and $(3)$. But $s_j(T)=s_j(|T|)$ for every $j$ and $\|Tx\|=\|\, |T|x\, \|$ for every $x\in \cl H$ which establishes the equivalence of $(1)$ and $(2)$. 
\end{proof}

\begin{thm}\label{AN_{[pi, k+1]}impliesAN_[pi, k](general)}
Let $A\in \cl B(\cl H, \cl K)$, $\pi \in \Pi$, and $k\in \bb{N}$. If $A\in \cl{AN}_{[\pi, k+1]}$, then $A\in \cl{AN}_{[\pi, k]}$.
\end{thm}

The proof of the above theorem is similar to that of the Theorem \ref{AN_{[k+1]}impliesAN_[k](general)} and hence omitted. It yields the following obvious corollary.
 \begin{cor}\label{AN_{[pi, k+1]}impliesAN_[pi, k]}
Let $A\in \cl B(\cl H)$ be a positive operator, $\pi \in \Pi$, and $k\in \bb{N}$. If $A\in \cl{AN}_{[\pi, k+1]}$, then $A\in \cl{AN}_{[\pi, k]}$. In particular, if $A\in \cl{AN}_{[\pi, k]}$, then $A\in\cl{AN}$.
 \end{cor}
The above corollary along with the forward implication of \cite[Theorem 5.1]{VpSp} yields the following theorem.

\begin{thm}\label{forward_[pi, k]}
Let $\cl{H}$ be a complex Hilbert space, $A$ be a positive operator on $\cl{H}$, $\pi\in \Pi$, and $k\in \bb N$. If $A\in\cl{AN}_{[\pi, k]}$, then $A$ is of the form $A= \alpha I +K+F$, where $\alpha \geq 0, K$ is a positive compact operator and $F$ is self-adjoint finite-rank operator.
\end{thm}

The following theorem is an analogue of Theorem \ref{TVM-TPM-k} and its proof may be handled in much the same way. This result will not be needed until section \ref{symmetrically-normed-ideals}.

\begin{thm}\label{TVM-TPM-Pi-k}
For a closed linear subspace $\cl{M}$ of a complex Hilbert space $\cl{H}$ let $V_{\cl{M}}:\cl{M}\longrightarrow \cl{H}$ be the inclusion map from $\cl{M}$ to $\cl{H}$ defined as $V_{\cl{M}}(x) = x$ for each $x\in \cl{M}$, let $P_{\cl M}\in \cl B(\cl H)$ be the orthogonal projection of $\cl H$ onto $\cl M$, and let $T\in \cl{B}(\cl{H},\cl{K})$.
For any sequence $\pi \in \Pi$ and for any $k\in\bb{N}$, the following statements are equivalent.
\begin{enumerate}
\item $T\in \cl{AN}_{[\pi, k]}(\cl{H},\cl{K})$.
\item $TV_{\cl{M}}\in\cl{N}_{[\pi, k]}(\cl{M},\cl{K})$ for every nontrivial closed linear subspace $\cl{M}$ of $\cl{H}$.
\item $TP_{\cl{M}}\in\cl{N}_{[\pi, k]}(\cl{H},\cl{K})$ for every nontrivial closed linear subspace $\cl{M}$ of $\cl{H}$.
\end{enumerate} 
\end{thm}

\subsection{Sufficient Conditions for Positive Operators in $\cl{AN}_{[\pi,k]}$}

We now mention the sufficient conditions for a positive operator to be absolutely $[\pi, k]$-norming for every $\pi \in \Pi$ and for every $k \in \bb N$. We begin by stating a proposition that gives a sufficient condition for a positive operator to be $[\pi, k]$-norming for every $\pi \in \Pi$ and for every $k \in \bb N$, the proof of which is easy to see.
\begin{prop}
Let $K\in\cl{B}(\cl{H})$ be a positive compact operator, $F\in\cl{B}(\cl{H})$ be a self-adjoint finite-rank operator, and $\alpha \geq 0$ such that $\alpha I+K+F\geq 0$. Then $\alpha I+K+F\in\cl{N}_{[\pi, k]}$ for every $\pi \in \Pi$ and for every $k\in \bb{N}$.
\end{prop}

This proposition serves to be the key to the following theorem.

\begin{thm}\label{backward_[pi, k]}
Let $K\in\cl{B}(\cl{H})$ be a positive compact operator, $F\in\cl{B}(\cl{H})$ be a self-adjoint finite-rank operator, and $\alpha \geq 0$ such that $\alpha I+K+F\geq 0$. Then $\alpha I+K+F\in\cl{AN}_{[\pi, k]}$ for every $\pi \in \Pi$ and for every $k\in \bb{N}$.
\end{thm}
\begin{proof}
Let us fix $\pi \in \Pi$ and $k\in \bb{N}$, and let us define $T:= \alpha I+K+F$. Due to Proposition \ref{Tiff|T|_[pi, k]}, $T\in\cl{AN}_{[\pi, k]}$ if and only if $|T|\in\cl{AN}_{[\pi, k]}$, which due to Lemma \ref{Trick_[pi, k]}, is possible if and only if for every nontrivial closed linear subspace $\cl M$ of $\cl H,\ |T|V_{\cl M}\in\cl{N}_{[\pi, k]}$, where $V_{\cl M}:\cl M\longrightarrow \cl H$ is the inclusion map defined as $V_{\cl M}(x)=x$ for each $x\in \cl M$. We show the last of these equivalent statements. 

Notice that $|T|= |\alpha I+K+F|$ and  $|T|^*|T| = \beta I+\tilde K + \tilde F$ where $\beta=\alpha^2 \geq 0$, and, $\tilde K = 2\alpha K+K^2$ and $\tilde F = 2\alpha F + FK +KF +F^2$ are respectively positive compact and self-adjoint finite-rank operators. It is easy to see that $\beta I+\tilde K + \tilde F\geq 0$.  Next we fix a closed linear subspace $\cl M$ of $\cl H$ and observe that 
\begin{align*}
 &|T|V_\cl{M}\in\cl{N}_{[\pi, k]}\\
 \iff  &(|T|V_\cl{M})^*(|T|V_\cl{M}) \in \cl{N}_{[\pi, k]} \\ 
 \iff &V_\cl{M}^*(|T|^*|T|)V_\cl{M} \in \cl{N}_{[\pi, k]} \\
  \iff & V_\cl{M}^*( \beta I+\tilde K + \tilde F)V_\cl{M} \in \cl{N}_{[\pi, k]} ,
\end{align*}
where the first equivalence is due to the Theorem \ref{TiffT*T_[pi, k]}.
It suffices to show that $V_\cl{M}^*( \beta I+\tilde K + \tilde F)V_\cl{M}\in\cl{N}_{[\pi, k]}$;for then, since $\cl M$ is arbitrary, the assertion immediately follows. To this end, notice that $V_\cl{M}^*( \beta I+\tilde K + \tilde F)V_\cl{M}:\cl{M}\longrightarrow \cl{M}$ is an operator on $\cl{M}$ and 
$$
V_\cl{M}^*( \beta I+\tilde K + \tilde F)V_\cl{M} = V_\cl{M}^*\beta IV_\cl{M} +V_\cl{M}^*\tilde K V_\cl{M} +V_\cl{M}^* \tilde FV_\cl{M} = \beta I_\cl{M} +\tilde K_\cl{M}+\tilde F_\cl{M}
$$
is the sum of a nonnegative scalar multiple of Identity, a positive compact operator and a self-adjoint finite-rank operator on the fixed Hilbert space $\cl{M}$ such that this sum is a positive operator on this Hilbert space $\cl{M}$ which, by the preceding proposition, belongs to $\cl{N}_{[\pi, k]}$. Moreover, since $\pi\in \Pi$ and $k\in \bb N$ are arbitrary, the result holds for every $\pi \in \Pi$ and for every $k\in \bb N$. This completes the proof.
\end{proof}

As an immediate consequence of the Theorem \ref{forward_[pi, k]} and Theorem \ref{backward_[pi, k]}, we get the following theorem which completely characterizes positive operators that are absolutely $[\pi, k]$-norming for any and every $\pi \in \Pi$ and $k\in \bb{N}$.

\begin{thm}[Spectral Theorem for Positive Operators in $\cl{AN}_{[\pi, k]}$]\label{SpThPAN_[pi,k]}
Let $\cl{H}$ be a complex Hilbert space of arbitrary dimension and $P$ be a positive operator on $\cl{H}$. Then the following statements are equivalent.
\begin{enumerate}
\item $P\in\cl{AN}_{[\pi, k]}$ for every $\pi \in \Pi$ and for every $k\in \bb N$.
\item $P\in\cl{AN}_{[\pi, k]}$ for some $\pi \in \Pi$ and for some $k\in \bb N$.
\item $P$ is of the form $P= \alpha I +K+F$, where $\alpha \geq 0, K$ is a positive compact operator and $F$ is self-adjoint finite-rank operator.
\end{enumerate}
\end{thm}

At this point, readers can move on to the result (see Theorem \ref{SpThAN_[pi, k]}) in the next section which completes the proposed motive of characterizing bounded operators on complex Hilbert spaces of arbitrary dimensions that attain their weighted Ky Fan $\pi,k$-norm on every closed subspace. However, it is perhaps worth a short digression to address the following question before closing this section: What can be said along the lines of Theorem \ref{backward_[pi, k]} in the case of an operator in the same form of $\alpha I +K+F$ which is not necessarily positive? We still have our other hypotheses, that is, $K\in \cl B(\cl H)$ is a positive compact operator, $F\in \cl B(\cl H)$ is self-adjoint operator, and $\alpha \geq 0.$ We address this question in the Proposition \ref{backward_[pi, k]strong}, the proof of which is left to the reader. The proof essentially requires the following lemma. 

\begin{lemma}\label{backward_[pi, k]strong(prereq)}
Let $K\in\cl{B}(\cl{H})$ be a positive compact operator, $F\in\cl{B}(\cl{H})$ be a self-adjoint finite-rank operator, and $\alpha \geq 0$. Then $\alpha I+K+F\in\cl{N}_{[\pi, k]}$ for every $\pi \in \Pi$ and for every $k\in \bb{N}$.
\end{lemma}

\begin{proof}
Fix $\pi\in \Pi$ and $k\in \bb N$. For any bounded operator $T\in \cl B(\cl H, \cl K)$ we observe that
\begin{align*}
T\in \cl N_{[k]} &\iff |T| \in \cl N_{[k]}\\
&\iff |T|\in \cl N_{[\pi, k]} \\
&\iff T\in \cl N_{[\pi, k]},
\end{align*}
where the first equivalence is due to the Theorem \ref{TiffT*T_[k]}, the second equivalence is due to the Theorem \ref{PositiveN_{[pi, k]}} and the last equivalence is due to the Theorem \ref{TiffT*T_[pi, k]}. 
This observation when applied to the Proposition \ref{alphaI+K+FisN_[k]} proves that $\alpha I +K+F\in \cl N_{[\pi, k]}$. Since $\pi\in \Pi$ and $k\in \bb N$ are arbitrary, it follows that $\alpha I +K+F\in \cl N_{[\pi, k]}$ for every $\pi \in \Pi$ and for every $k\in \bb N$ thereby proving the assertion.
\end{proof}

\begin{remark}\label{N_[k]iffN_[pi, k]}
The proof of the Lemma \ref{backward_[pi, k]strong(prereq)} uses a rather interesting result which deserves to be stated for its intrinsic interest. If $T\in \cl B(\cl H, \cl K),\, \pi\in \Pi$, and $k\in \bb N$, then $T\in \cl N_{[k]} 
\iff T\in \cl N_{[\pi, k]}$.
\end{remark}

\begin{prop}\label{backward_[pi, k]strong}
Let $K\in\cl{B}(\cl{H})$ be a positive compact operator, $F\in\cl{B}(\cl{H})$ be a self-adjoint finite-rank operator, and $\alpha \geq 0$. Then $\alpha I+K+F\in\cl{AN}_{[\pi, k]}$ for every $\pi \in \Pi$ and for every $k\in \bb{N}$.
\end{prop}

\begin{proof}This proof is very similar to the proof of the Theorem \ref{backward_[pi, k]}. As before, let us define $T:= \alpha I+K+F$. We need to show that $T\in\cl{AN}_{[\pi, k]}$ for every $\pi \in \Pi$ and for every $k\in \bb{N}$. Let us fix $\pi\in \Pi$ and $k\in \bb N$. 
The Proposition \ref{Tiff|T|_[pi, k]}, together with the Lemma \ref{Trick_[pi, k]} shows that it suffices to show that for every nontrivial closed linear subspace $\cl M$ of $\cl H,\ |T|V_{\cl M}\in\cl{N}_{[\pi, k]}$, where $V_{\cl M}:\cl M\longrightarrow \cl H$ is the inclusion map as defined earlier. Next we fix a closed linear subspace $\cl M$ of $\cl H$ and observe that
\begin{align*}
 &|T|V_\cl{M} \in\cl{N}_{[\pi, k]}\\
  \iff & V_\cl{M}^*( \beta I+\tilde K + \tilde F)V_\cl{M} \in\cl{N}_{[\pi, k]},
\end{align*}
where 
$\beta I+\tilde K + \tilde F=|T|^*|T|$ with $\beta=\alpha^2 \geq 0$, $\tilde K = 2\alpha K+K^2$ and $\tilde F = 2\alpha F + FK +KF +F^2$. 
All that remains to be shown is that $V_\cl{M}^*( \beta I+\tilde K + \tilde F)V_\cl{M}\in\cl{N}_{[\pi, k]}$; for then, since $\cl M$ is arbitrary, the assertion immediately follows. To this end, notice that $V_\cl{M}^*( \beta I+\tilde K + \tilde F)V_\cl{M}:\cl{M}\longrightarrow \cl{M}$ is an operator on $\cl{M}$ and 
$
V_\cl{M}^*( \beta I+\tilde K + \tilde F)V_\cl{M} = \beta I_\cl{M} +\tilde K_\cl{M}+\tilde F_\cl{M}
$
is the sum of a nonnegative scalar multiple of Identity, a positive compact operator and a self-adjoint finite-rank operator on the fixed Hilbert space $\cl{M}$ which, by the preceding lemma, belongs to $\cl{N}_{[\pi, k]}$. Since $\pi\in \Pi$ and $k\in \bb N$ are arbitrary, the result holds for every $\pi \in \Pi$ and for every $k\in \bb N$ and thus an operator of the above form belongs to $\cl{AN}_{[\pi, k]}$ for every $\pi \in \Pi$ and for every $k\in \bb N$. This completes the proof.
\end{proof}

\section{Spectral Characterization of Operators in $\cl{AN}_{[\pi,k]}$}
By Proposition \ref{Tiff|T|_[pi, k]}, the polar decomposition theorem (see Theorem \ref{PDT}) and the spectral theorem for positive operators $\cl{AN}_{[\pi, k]}$ (see Theorem \ref{SpThPAN_[pi,k]}), we can safely consider the following theorem to be fully proved.

\begin{thm}[Spectral Theorem for Operators in $\cl{AN}_{[\pi, k]}$]\label{SpThAN_[pi, k]}
Let $\cl{H}$ and $\cl{K}$ be complex Hilbert spaces of arbitrary dimensions, let $T\in \cl{B}(\cl{H},\cl{K})$ and let $T=U|T|$ be its polar decomposition.
Then the following statements are equivalent.
\begin{enumerate}
\item $T\in\cl{AN}_{[\pi, k]}$ for every $\pi \in \Pi$ and for every $k\in \bb N$.
\item $T\in\cl{AN}_{[\pi, k]}$ for some $\pi \in \Pi$ and for some $k\in \bb N$.
\item $|T|$ is of the form $|T|= \alpha I +K+F$, where $\alpha \geq 0, K$ is a positive compact operator and $F$ is self-adjoint finite-rank operator.
\end{enumerate}  
\end{thm}

\section{The Classes $\cl N_{(p,k)}$ and $\cl {AN}_{(p,k)}$}
Govind S. Mudholkar and Marshall Freimer focussed on a particular class of norms in \cite{M} --- the vector $p$ norm of the first $k$ singular values --- and found specific results about these norms.
Nathaniel Johnston, in one of his blogs \textit{Ky Fan Norms, Schatten Norms, and Everything in Between}, discusses these norms as the natural generalization of two well known families of norms, the Ky Fan norms and the Schatten norms. He coined in the term ``$(p,k)$-singular norm'' for this class of norms. 

\begin{defn}[$(p, k)$-singular norm]\cite{M}
Let $p\in [1,\infty)$ and let $k\in \bb N$. The $(p,k)$-singular norm $\|\cdot\|_{(p, k)}$ of an operator $T\in \cl{B}(\cl{H},\cl{K})$ is defined to be the vector $p$ norm of the $k$ largest singular values of $T$, that is, 
$$
\|T\|_{(p,k)}=\left(\sum_{j=1}^ks_j^p(T)\right)^{1/p}.
$$
The $(p,k)$-singular norm on $\cl{B}(\cl{H},\cl{K})$ is, indeed, a norm. When $\cl K=\cl H$, it can be shown that this norm is symmetric.
\end{defn}

\begin{remark}
If we choose $p=1$, then the $(1,k)$-singular norm $\|\cdot\|_{(1, k)}$ is simply the Ky Fan $k$-norm $\|\cdot\|_{[k]}$. If in addition, we also choose $k=1$, we get the operator norm.
\end{remark}

\begin{defn}\label{N_(p,k)Def}
Let $p\in [1,\infty)$ and let $k\in \bb N$. An operator $T\in \cl{B}(\cl{H},\cl{K})$ is said to be \emph{$(p,k)$-norming} if there are orthonormal elements $x_1,...,x_k \in \cl{H}$ such that $$\|T\|_{(p,k)}^p=\sum_{j=1}^k\|Tx_j\|^p.$$ If $\dim(\cl{H})=r<k$, we define $T$ to be $(p,k)$-norming if there exist orthonormal elements $x_1,...,x_r\in \cl{H}$ such that $\|T\|^p_{(p,k)}=\sum_{j=1}^r\|Tx_j\|^p$. We let $\cl{N}_{(p,k)}(\cl H, \cl K)$ denote the set of $(p,k)$-norming operators in $\cl B(\cl H, \cl K)$.
\end{defn}

\begin{defn}\label{AN_(p,k)Def}
Let $p\in [1,\infty)$ and let $k\in \bb N$. An operator $T\in \cl{B}(\cl{H},\cl{K})$ is said to be \emph{absolutely $(p,k)$-norming} if for every nontrivial closed subspace $\cl{M}$ of $\cl{H}$, $T|_{\cl{M}}$ is $(p,k)$-norming. We let $\cl{AN}_{(p,k)}(\cl H,\cl K)$ denote the set of absolutely $(p,k)$-norming operators in $\cl B(\cl H,\cl K)$. Note that $\cl{AN}_{(p,k)}(\cl H,\cl K)\subseteq\cl{N}_{(p,k)}(\cl H,\cl K)$.
\end{defn}

\begin{remark} Every operator on a finite-dimensional Hilbert space is $(p, k)$-norming for each $p\in [1,\infty)$ and for each $k\in \bb N$. However, this is not true when the Hilbert space in question is not finite-dimensional. The operator in Example \ref{Counterex} is one such operator. There exists $p_0=1$ such that $A\notin\cl N_{(p_0, 3)}(\cl H, \cl K)$.
\end{remark}

The following lemma can be considered as an analogue of Lemma \ref{Trick_[k]}. 
\begin{lemma}\label{Trick_(p, k)}
For a closed linear subspace $\cl{M}$ of a complex Hilbert space $\cl{H}$ let $V_{\cl{M}}:\cl{M}\longrightarrow \cl{H}$ be the inclusion map from $\cl{M}$ to $\cl{H}$ defined as $V_{\cl{M}}(x) = x$ for each $x\in \cl{M}$ and let $T\in \cl B(\cl H, \cl K)$.
For any real number $p\in [1,\infty)$ and for any $k\in\bb{N}$, $T\in \cl{AN}_{(p,k)}(\cl{H},\cl{K})$ if and only if for every nontrivial closed linear subspace $\cl{M}$ of $\cl{H}$, $TV_{\cl{M}}\in\cl{N}_{(p, k)}(\cl H, \cl K)$.
\end{lemma}

\begin{proof}
To prove this assertion we first observe that for any given nontrivial closed subspace $\cl{M}$ of $\cl{H}$, the maps $TV_\cl{M}$ and $T|_\cl{M}$ are identical and so are their singular values which implies $\|TV_\cl{M}\|_{(p, k)} = \|T|_\cl{M}\|_{(p, k)}$.
We next assume that $T\in\cl{AN}_{(p, k)}(\cl H, \cl K)$ and prove the forward implication. Let $\cl{M}$ be an arbitrary but fixed nontrivial closed subspace of $\cl{H}$. Either $\dim(\cl{M})=r<k$, in which case, there exist orthonormal elements $x_1,...,x_r \in \cl{M}$ such that $\|T|_\cl{M}\|^p_{(p, k)} = \sum_{j=1}^r\|T|_{\cl{M}}x_j\|^p$ which means that there exist orthonormal elements $x_1,...,x_r \in \cl{M}$ such that $\|TV_\cl{M}\|^p_{(p, k)} =\|T|_\cl{M}\|^p_{(p, k)}=\sum_{j=1}^r\|T|_{\cl{M}}x_j\|^p = \sum_{j=1}^r\|TV_\cl{M} x_j\|^p$ proving that $TV_{\cl{M}}\in\cl{N}_{(p, k)}(\cl M, \cl K)$, or $\dim(\cl{M})\geq k$, in which case, there exist orthonormal elements $x_1,...,x_k \in \cl{M}$ such that $\|T|_\cl{M}\|^p_{(p, k)} =\sum_{j=1}^k\|T|_{\cl{M}}x_j\|^p$ which means that there exist orthonormal elements $x_1,...,x_k \in \cl{M}$ such that $\|TV_\cl{M}\|^p_{(p, k)} =\|T|_\cl{M}\|^p_{(p, k)}=\sum_{j=1}^k\|T|_{\cl{M}}x_j\|^p= \sum_{j=1}^k\|TV_\cl{M} x_j\|^p$ proving that $TV_{\cl{M}}\in\cl{N}_{(p, k)}(\cl M, \cl K)$. Since $\cl{M}$ is arbitrary, it follows that $TV_\cl{M}\in\cl{N}_{(p, k)}(\cl M, \cl K)$.

We complete the proof by showing that $T$ is an $\cl{AN}_{(p, k)}(\cl H, \cl K)$ operator if $TV_\cl{M}\in\cl{N}_{(p, k)}(\cl M, \cl K)$ for every nontrivial closed subspace $\cl{M}$ of $\cl{H}$. We again fix $\cl{M}$ to be an arbitrary nontrivial closed subspace of $\cl{H}$. Since $TV_\cl{M}\in\cl{N}_{(p, k)}(\cl M, \cl K)$, either $\dim(\cl{M})=r<k$, in which case, there exist orthonormal elements $x_1,...,x_r \in \cl{M}$ such that $\|TV_\cl{M}\|^p_{(p, k)}=\sum_{j=1}^r\|TV_\cl{M} x_j\|^p$ which means that there exist orthonormal elements $x_1,...,x_r \in \cl{M}$ such that
 $\|T|_\cl{M}\|^p_{(p, k)}=\|TV_\cl{M}\|^p_{(p, k)}=\sum_{j=1}^r\|TV_\cl{M} x_j\|^p\linebreak=\sum_{j=1}^r\|T|_\cl{M} x_j\|^p$ proving that $T|_{\cl{M}}\in\cl{N}_{(p, k)}(\cl M, \cl K)$, or  $\dim(\cl{M})\geq k$, in which case, there exist  orthonormal elements $x_1,...,x_k \in \cl{M}$ such that $\|TV_\cl{M}\|^p_{(p, k)}=\sum_{j=1}^k\|TV_\cl{M}x_j\|^p$ which means that there exist orthonormal elements $x_1,...,x_k \in \cl{M}$ such that $\|T|_\cl{M}\|^p_{(p, k)}=\|TV_\cl{M}\|^p_{(p, k)}=\sum{j=1}^k\|TV_\cl{M}x_j\|^p =\sum_{j=1}^k\|T|_\cl{M}x_j\|^p$ proving that $T|_{\cl{M}}\in\cl{N}_{(p, k)}(\cl M, \cl K)$. Because $\cl{M}$ is arbitrary, this essentially guarantees that $T\in\cl{AN}_{(p, k)}(\cl H, \cl K)$. It is worthwhile noticing that since $p\in [1,\infty)$ and $k\in \bb{N}$ are arbitrary, the assertion holds for every $p\in [1,\infty)$ and for every $k\in \bb{N}.$
\end{proof}

Without going into details, we mention that Proposition \ref{Tiff|T|_[k]} carries over word for word to operators in $\cl{AN}_{(p, k)}(\cl H,\cl K)$.

 \begin{prop}\label{Tiff|T|_(p, k)}
Let $\cl{H}$ and $\cl{K}$ be complex Hilbert spaces and let $T\in \cl{B}(\cl{H},\cl{K})$. Then for every $p\in [1,\infty)$ and for every $k\in \bb{N}$, $T\in\cl{AN}_{(p, k)}(\cl H, \cl K)$ if and only if $|T|\in\cl{AN}_{(p, k)}(\cl H, \cl K)$.
 \end{prop}
\begin{remark}
Henceforth, we use $\cl{N}_{(p, k)}$ and $\cl{AN}_{(p, k)}$ for the sets   $\cl{N}_{(p, k)}(\cl{H},\cl{K})$ and $\cl{AN}_{(p, k)}(\cl{H},\cl{K})$ respectively.
\end{remark}

\section{Spectral Characterization of Positive Operators in $\cl{AN}_{(p,k)}$}

This section discusses the necessary and sufficient conditions for a positive operator on complex Hilbert space of arbitrary dimension to be absolutely $(p,k)$-norming for every $p\in [1,\infty)$ and for every $k \in \bb N$.  We state the following proposition that adds few equivalent conditions to the Theorem \ref{EasyPeasy} 

\begin{prop}\label{EasyPeasy_(p, k)}
Let $A\in \cl{B}(\cl{H})$ be a positive operator. Then the following statements are equivalent.
\begin{enumerate}
\item $A\in\cl{N}.$
\item $A\in\cl{N}_{(p,1)}$ for some $p\in [1,\infty)$.
\item $A\in\cl{N}_{(p,1)}$ for every $p\in [1,\infty)$.
\end{enumerate} 
\end{prop}

The following proposition is analogous to Proposition \ref{1.2} and can be proved in much the same way. 

\begin{prop}\label{Analogue-2-1.2}
Let $p\in [1,\infty)$ and $A\in \cl B(\cl H)$ be a positive operator. If $s_{m+1}(A)\neq s_m(A)$ for some $m\in \bb N$, then $A\in \cl N_{(p, m)}$. Moreover, in this case, $A\in \cl N_{(p,m+1)}$ if and only if $s_{m+1}(A)$ is an eigenvalue of $A$.
\end{prop}
\begin{proof}
From the proof of the Proposition \ref{Analogue-1-1.2}, we deduce that there exists an orthonormal set $\{v_1,...,v_m\}\subseteq K\subseteq \cl H$ such that $Av_j=s_j(A)v_j$ so that $\|Av_j\|^p=s_j^p(A)$ and thus $\|A\|_{(p, m)}^p=\sum_{j=1}^ms_j^p(A)=\sum_{j=1}^m\|Av_j\|^p$, where $K$ is the closure of the joint span of the eigenspaces corresponding to the eigenvalues $\{s_1(A),...,s_m(A)\}$, which implies that $A\in \cl N_{(p, m)}$. Furthermore, we observe that if there exists any orthonormal set $\{w_1,...,w_m\}$ of $m$ vectors in $\cl H$ such that $\sum_{i=1}^m\|Aw_i\|^p=\sum_{j=1}^ms_j^p(A)$, then this set has to be contained in $K$. Note that $A|_{K^\perp}:K^\perp\longrightarrow K^\perp$, viewed as an operator on $K^\perp$, is positive. Since $s_{m+1}(A)\neq s_m(A)$, it follows that $s_{m+1}(A)$ is an eigenvalue of $A\iff$ $s_{m+1}(A)$ is an eigenvalue of $A|_{K^\perp}:K^\perp\longrightarrow K^\perp$ which is possible $\iff A|_{K^\perp}:K^\perp\longrightarrow K^\perp$, viewed as an operator on $K^\perp$, belongs to $\cl N$, that is, there is a unit vector $x\in K^\perp$ such that $\|Ax\|=s_{m+1}(A)$, which is possible if and only if $\|Ax\|^p=s_{m+1}^p(A)\iff A\in\cl N_{(p, m+1)}$. This proves the assertion. 
\end{proof}
\subsection{Necessary Conditions for Positive Operators in $\cl {AN}_{(p,k)}$}

Using the above proposition, it is not too hard to establish results analogous to Propositions \ref{1.6}, \ref{1.7} and \ref{1.8} (see \ref{analogous-2-1.6}, \ref{analogous-2-1.7}, and \ref{analogous-2-1.8} respectively) for a given $p\in [1,\infty)$ and a given $k\in \bb N$. 
\begin{prop}\label{analogous-2-1.6}
Let $A\in \cl{B}(\cl{H})$ be a positive operator, $p\in [1,\infty)$, and $k\in \bb{N}$. If $A\in \cl N_{(p,k)}$, then $s_1(A),...,s_k(A)$ are eigenvalues of $A$.
\end{prop}
The proof of the above proposition is along the lines of the proof of Proposition  \ref{1.6} or Propositoin \ref{analogous-1-1.6}. The following two propositions are also not too difficult to see and hence we omit their proofs.
\begin{prop}\label{analogous-2-1.7}
Let $A\in \cl{B}(\cl{H})$ be a positive operator, $p\in [1,\infty)$, and $k\in \bb{N}$. If $s_1(A),...,s_k(A)$ are mutually distinct eigenvalues of $A$,  then there exists an orthonormal set $\{v_1,...,v_k\}\subseteq \cl H$ such that $Av_j=s_j(A)v_j$ for every $j\in \{1,...,k\}$. Thus $A\in\cl N_{(p,k)}$.
\end{prop}

\begin{prop}\label{analogous-2-1.8}
Let $A\in \cl{B}(\cl{H})$ be a positive operator, $p\in [1,\infty)$, $k\in \bb N$ and let $s_1(A),...,s_k(A)$ be the first $k$ singular values of $A$ that are also the eigenvalues of $A$ and are not necessarily distinct. Then either $s_1(A)=...=s_k(A)$, in which case, $A\in \cl N_{(p,k)}$ if and only if the multiplicity of $\alpha:=s_1(A)$ is at least $k$; or there exists $t\in \{2,...,k\}$ such that $s_{t-1}(A)\neq s_t(A)=s_{t+1}(A)=...=s_k(A)$, in which case, $A\in \cl N_{(p,k)}$ if and only if the multiplicity of $\beta:=s_t(A)$ is at least $k-t+1$.
\end{prop}

These propositions yield the following result that adds yet another equivalent condition to the Theorem \ref{PositiveN_{[pi, k]}}.

\begin{thm}\label{PositiveN_{(p, k)}}
Let $A\in \cl{B}(\cl{H})$ be a positive operator, $p\in [1,\infty)$, and $k\in \bb{N}$. Then the following statements are equivalent.
\begin{enumerate}
\item $A\in\cl{N}_{[k]}.$
\item $A\in\cl{N}_{(p,k)}$.
\item $s_1(A),...,s_k(A)$ are eigenvalues of $A$ and there exists an orthonormal set $\{v_1,...,v_k\}\subseteq \cl H$ such that $Av_j=s_j(A)v_j$ for every $j\in \{1,...,k\}.$
\item $A\in\cl{N}_{[\pi,k]}$.
\end{enumerate} 
\end{thm}

\begin{proof}
$(1)\iff(3)\iff (4)$ has been established in Theorem \ref{PositiveN_{[pi, k]}} and $(3)\implies (2)$ is trivial.  All that remains to show is $(2)\implies (3)$. By the
Proposition \ref{analogous-2-1.6}, $s_1(A),..., s_k(A)$ are all eigenvalues of $A$. If $s_1(A),..., s_k(A)$ are mutually distinct eigenvalues, then by Proposition \ref{analogous-2-1.7} there exists an orthonormal set $\{v_1,...,v_k\}\subseteq \cl H$ such that $Av_j=s_j(A)v_j$ for every $j\in \{1,...,k\}$. However, if $s_1(A),...,s_k(A)$ are all eigenvalues but not necessarily distinct then also the existence of an orthonormal set $\{v_1,...,v_k\}\subseteq \cl H$ with $Av_j=s_j(A)v_j$ for every $j\in \{1,...,k\}$ is guaranteed by the Proposition \ref{analogous-2-1.8}. This completes the proof.
\end{proof}

The following corollary is easy to deduce from the above theorem. 
\begin{cor}\label{N_{(p, k+1)}impliesN_{(p, k)}}
Let $A\in \cl B(\cl H)$ be a positive operator, $p\in [1,\infty)$, and let $k\in \bb{N}$. 
\begin{enumerate}
\item If $A\in \cl N_{(p, k+1)}$, then $A\in \cl N_{(p,k)}$.
\item If $A\in \cl N_{(p,k)}$, then $A\in \cl N$.
\end{enumerate}
\end{cor}
Theorem \ref{TiffT*T_[k]} extends word for word to the family $\cl N_{(p, k)}$ (see \ref{TiffT*T_(p, k)}) and Theorem \ref{AN_{[k+1]}impliesAN_[k](general)} alongwith the Corollary \ref{AN_{[k+1]}impliesAN_[k]} extend to the family $\cl{AN}_{(p, k)}$ (see \ref{AN_{(p, k+1)}impliesAN_(p, k)(general)} and \ref{AN_{(p, k+1)}impliesAN_(p, k)} respectively).

\begin{thm}\label{TiffT*T_(p, k)}
Let $\cl{H}$ and $\cl{K}$ be complex Hilbert spaces, $T\in \cl{B}(\cl{H},\cl{K})$, $p\in [1,\infty)$, and $k\in\bb{N}$. Then the following statements are equivalent.
\begin{enumerate}
\item $T\in\cl{N}_{(p,k)}$.
\item $|T|\in\cl{N}_{(p,k)}$.
\item $T^*T\in\cl{N}_{(p,k)}$.
\end{enumerate}
\end{thm}

\begin{proof}
It suffices to establish $(1)\iff (2)$; for then $|T|$ and $T^*T$ being positive and the sets $\cl{N}_{[k]}$ and $\cl{N}_{(p, k)}$ being identical for positive operators,  Theorem \ref{TiffT*T_[k]} yields the equivalence of $(2)$ and $(3)$. But $s_j(T)=s_j(|T|)$ for every $j$ and $\|Tx\|=\|\, |T|x\, \|$ for every $x\in \cl H$ which implies that $\sum_{j=1}^k\|Tx_j\|^p=\sum_{j=1}^ks_j^p(T)\iff\sum_{j=1}^k\|\, |T|x_j\, \|^p=\sum_{j=1}^ks_j^p(|T|)$. This establishes the equivalence of $(1)$ and $(2)$. 
\end{proof}

\begin{thm}\label{AN_{(p, k+1)}impliesAN_(p, k)(general)}
Let $A\in \cl B(\cl H, \cl K)$, $p\in [1,\infty)$, and $k\in \bb{N}$. If $A\in \cl{AN}_{(p,k+1)}$, then $A\in \cl{AN}_{(p,k)}$.
\end{thm}
The proof of the above theorem is similar to that of the Theorem \ref{AN_{[k+1]}impliesAN_[k](general)} and the following corollary is easy to deduce from it.

 \begin{cor}\label{AN_{(p, k+1)}impliesAN_(p, k)}
Let $A\in \cl B(\cl H)$ be a positive operator, $p\in [1,\infty)$, and $k\in \bb{N}$. If $A\in \cl{AN}_{(p, k+1)}$, then $A\in \cl{AN}_{(p, k)}$. In particular, if $A\in \cl{AN}_{(p, k)}$, then $A\in\cl{AN}$.
 \end{cor}
The above corollary alongwith the forward implication of \cite[Theorem 5.1]{VpSp} yields the main result of this subsection as the following theorem.

\begin{thm}\label{forward_(p, k)}
Let $\cl{H}$ be a complex Hilbert space, $A$ be a positive operator on $\cl{H}$, $p\in [1,\infty)$, and $k\in \bb{N}$. If $A\in\cl{AN}_{(p, k)}$, then $A$ is of the form $A= \alpha I +K+F$, where $\alpha \geq 0, K$ is a positive compact operator and $F$ is self-adjoint finite-rank operator.
\end{thm}

Theorem \ref{TVM-TPM-k} extends word for word to the family $\cl {AN}_{(p,k)}(\cl H, \cl K)$ and its proof is similar. This result will be needed in section \ref{symmetrically-normed-ideals}.

\begin{thm}\label{TVM-TPM-p-k}
For a closed linear subspace $\cl{M}$ of a complex Hilbert space $\cl{H}$ let $V_{\cl{M}}:\cl{M}\longrightarrow \cl{H}$ be the inclusion map from $\cl{M}$ to $\cl{H}$ defined as $V_{\cl{M}}(x) = x$ for each $x\in \cl{M}$, let $P_{\cl M}\in \cl B(\cl H)$ be the orthogonal projection of $\cl H$ onto $\cl M$, and let $T\in \cl{B}(\cl{H},\cl{K})$.
For any real number $p\in [1,\infty)$ and for any $k\in\bb{N}$, the following statements are equivalent.
\begin{enumerate}
\item $T\in \cl{AN}_{(p,k)}(\cl{H},\cl{K})$.
\item $TV_{\cl{M}}\in\cl{N}_{(p,k)}(\cl{M},\cl{K})$ for every nontrivial closed linear subspace $\cl{M}$ of $\cl{H}$.
\item $TP_{\cl{M}}\in\cl{N}_{(p,k)}(\cl{H},\cl{K})$ for every nontrivial closed linear subspace $\cl{M}$ of $\cl{H}$.
\end{enumerate} 
\end{thm}

\subsection{Sufficient Conditions for Positive Operators in $\cl {AN}_{(p,k)}$}

We now discuss the sufficient conditions for a positive operator to be absolutely $(p,k)$-norming for every $p\in[1,\infty)$ and for every $k \in \bb N$. We begin by stating a proposition that gives a sufficient condition for a positive operator to be $(p,k)$-norming for every $p\in[1,\infty)$ and for every $k \in \bb N$.
\begin{prop}
Let $K\in\cl{B}(\cl{H})$ be a positive compact operator, $F\in\cl{B}(\cl{H})$ be a self-adjoint finite-rank operator, and $\alpha \geq 0$ such that $\alpha I+K+F\geq 0$. Then $\alpha I+K+F\in\cl{N}_{(p,k)}$ for every $p\in[1,\infty)$ and for every $k\in \bb{N}$.
\end{prop}
 The above result follows immediately from the Proposition \ref{alphaI+K+FisN_[k]} and Theorem \ref{PositiveN_{(p, k)}}. In fact, this proposition is a special case of the following theorem.

\begin{thm}\label{backward_(p, k)}
Let $K\in\cl{B}(\cl{H})$ be a positive compact operator, $F\in\cl{B}(\cl{H})$ be a self-adjoint finite-rank operator, and $\alpha \geq 0$ such that $\alpha I+K+F\geq 0$. Then $\alpha I+K+F\in\cl{AN}_{(p,k)}$ for every $p\in[1,\infty)$ and for every $k\in \bb{N}$.
\end{thm}

\begin{proof}
Let us fix $p\in [1,\infty)$ and $k\in \bb{N}$, and let us define $T:= \alpha I+K+F$. Due to Proposition \ref{Tiff|T|_(p, k)}, $T\in\cl{AN}_{(p,k)}$ if and only if $|T|\in\cl{AN}_{(p,k)}$, which due to Lemma \ref{Trick_(p, k)}, is possible if and only if for every nontrivial closed linear subspace $\cl M$ of $\cl H,\ |T|V_{\cl M}\in\cl{N}_{(p,k)}$ for every $p\in[1,\infty)$ and for every $k\in \bb{N}$, where $V_{\cl M}:\cl M\longrightarrow \cl H$ is the inclusion map as defined earlier. We show the last of these equivalent statements. 

Notice that $|T|= |\alpha I+K+F|$ and  $|T|^*|T| = \beta I+\tilde K + \tilde F$ where $\beta=\alpha^2 \geq 0$, and, $\tilde K = 2\alpha K+K^2$ and $\tilde F = 2\alpha F + FK +KF +F^2$ are respectively positive compact and self-adjoint finite-rank operators. It is easy to see that $\beta I+\tilde K + \tilde F\geq 0$.  Next we fix a closed linear subspace $\cl M$ of $\cl H$ and observe that 
\begin{align*}
 &|T|V_\cl{M}\in\cl{N}_{(p,k)} \\
 \iff  &(|T|V_\cl{M})^*(|T|V_\cl{M}) \in \cl{N}_{(p,k)} \\ 
 \iff &V_\cl{M}^*(|T|^*|T|)V_\cl{M} \in\cl{N}_{(p,k)} \\
  \iff & V_\cl{M}^*( \beta I+\tilde K + \tilde F)V_\cl{M} \in \cl{N}_{(p,k)},
\end{align*}
where the first equivalence is due to the Theorem \ref{TiffT*T_(p, k)}.
It suffices to show that $V_\cl{M}^*( \beta I+\tilde K + \tilde F)V_\cl{M}\in\cl{N}_{(p,k)}$;for then, since $\cl M$ is arbitrary, the assertion immediately follows. To this end, notice that $V_\cl{M}^*( \beta I+\tilde K + \tilde F)V_\cl{M}:\cl{M}\longrightarrow \cl{M}$ is an operator on $\cl{M}$ and 
$$
V_\cl{M}^*( \beta I+\tilde K + \tilde F)V_\cl{M} = V_\cl{M}^*\beta IV_\cl{M} +V_\cl{M}^*\tilde K V_\cl{M} +V_\cl{M}^* \tilde FV_\cl{M} = \beta I_\cl{M} +\tilde K_\cl{M}+\tilde F_\cl{M}
$$
is the sum of a nonnegative scalar multiple of Identity, a positive compact operator and a self-adjoint finite-rank operator on the fixed Hilbert space $\cl{M}$ such that this sum is a positive operator on this Hilbert space $\cl{M}$ which, by the preceding proposition, belongs to $\cl{N}_{(p,k)}$. Since $p\in [1,\infty)$ and $k\in \bb{N}$ are arbitrary, it follows that an operator of the above form belongs to $\cl{AN}_{(p,k)}$ for every  $p\in [1,\infty)$ and for every $k\in \bb{N}$. This completes the proof.
\end{proof}

As an immediate consequence of the Theorem \ref{forward_(p, k)} and Theorem \ref{backward_(p, k)}, we get the following theorem which completely characterizes positive operators that are absolutely $(p,k)$-norming for any and every $p\in[1,\infty)$ and $k\in \bb{N}$.

\begin{thm}[Spectral Theorem for Positive Operators in $\cl{AN}_{(p, k)}$ ]\label{SpThPAN_(p, k)}
Let $\cl{H}$ be a complex Hilbert space of arbitrary dimension and $P$ be a positive operator on $\cl{H}$. 
Then the following statements are equivalent.
\begin{enumerate}
\item $P\in\cl{AN}_{(p,k)}$ for every $p\in[1,\infty)$ and for every $k\in \bb N$.
\item $P\in\cl{AN}_{(p,k)}$ for some $p\in[1,\infty)$ and for some $k\in \bb N$.
\item $P$ is of the form $P= \alpha I +K+F$, where $\alpha \geq 0, K$ is a positive compact operator and $F$ is self-adjoint finite-rank operator.
\end{enumerate}
\end{thm}

We have all that is required to move on to the next section and establish the result (see Theorem \ref{SpThAN_(p, k)}) which characterizes bounded operators that attain their $(p,k)$-singular norm on every closed subspace. 
Lemma \ref{backward_[pi, k]strong(prereq)} and Proposition \ref{backward_[pi, k]strong} carry over word for word to the operators in $\cl N_{(p,k)}$ and operators in $\cl{AN}_{(p,k)}$ respectively (see \ref{backward_(p, k)strong(prereq)} and \ref{backward_(p, k)strong}). The Proposition \ref{backward_(p, k)strong} addresses the question of whether an operator of the form $\alpha I +K+F$, which is not necessarily positive, belongs to $\cl{AN}_{(p,k)}$.  

\begin{lemma}\label{backward_(p, k)strong(prereq)}
Let $K\in\cl{B}(\cl{H})$ be a positive compact operator, $F\in\cl{B}(\cl{H})$ be a self-adjoint finite-rank operator, and $\alpha \geq 0$. Then $\alpha I+K+F\in\cl{N}_{(p,k)}$ for every $p\in[1,\infty)$ and for every $k\in \bb{N}$.
\end{lemma}

\begin{proof}
Fix $p\in [1,\infty)$ and $k\in \bb N$. For any bounded operator $T\in \cl B(\cl H, \cl K)$ we observe that
\begin{align*}
T\in \cl N_{[k]} &\iff |T| \in \cl N_{[k]}\\
&\iff |T|\in \cl N_{(p,k)} \\
&\iff T\in \cl N_{(p,k)},
\end{align*}
where the first equivalence is due to the Theorem \ref{TiffT*T_[k]}, the second equivalence is due to the Theorem \ref{PositiveN_{(p, k)}} and the last equivalence is due to the Theorem \ref{TiffT*T_(p, k)}. 
This observation when applied to the Proposition \ref{alphaI+K+FisN_[k]} proves that $\alpha I +K+F\in \cl N_{(p,k)}$. Since $p\in [1,\infty)$ and $k\in \bb N$ are arbitrary, it follows that $\alpha I +K+F\in \cl N_{(p,k)}$ for every $p\in[1,\infty)$ and for every $k\in \bb N$ which proves the assertion.
\end{proof}

\begin{prop}\label{backward_(p, k)strong}
Let $K\in\cl{B}(\cl{H})$ be a positive compact operator, $F\in\cl{B}(\cl{H})$ be a self-adjoint finite-rank operator, and $\alpha \geq 0$. Then $\alpha I+K+F\in\cl{AN}_{(p, k)}$ for every $p\in[1,\infty)$ and for every $k\in \bb{N}$.
\end{prop}

\begin{proof}
This proof is very similar to the proof of the Theorem \ref{backward_(p, k)}. As before, let us define $T:= \alpha I+K+F$. We need to show that $T\in\cl{AN}_{(p,k)}$ for every $p\in[1,\infty)$ and for every $k\in \bb{N}$. Let us fix $p\in[1,\infty)$ and $k\in \bb{N}$.
The Proposition \ref{Tiff|T|_(p, k)}, together with the Lemma \ref{Trick_(p, k)} shows that it suffices to show that for every nontrivial closed linear subspace $\cl M$ of $\cl H,\ |T|V_{\cl M}\in\cl{N}_{(p,k)}$ for every $p\in[1,\infty)$ and for every $k\in \bb{N}$, where $V_{\cl M}:\cl M\longrightarrow \cl H$ is the inclusion map as defined earlier. Next we fix a closed linear subspace $\cl M$ of $\cl H$ and observe that
\begin{align*}
 |T|V_\cl{M} \in\cl{N}_{(p,k)} 
  \iff  V_\cl{M}^*( \beta I+\tilde K + \tilde F)V_\cl{M} \in \cl{N}_{(p,k)},
\end{align*}
where 
$\beta I+\tilde K + \tilde F=|T|^*|T|$ with $\beta=\alpha^2 \geq 0$, $\tilde K = 2\alpha K+K^2$ and $\tilde F = 2\alpha F + FK +KF +F^2$. 
All that remains to be shown is that $V_\cl{M}^*( \beta I+\tilde K + \tilde F)V_\cl{M}\in\cl{N}_{(p,k)}$. To this end, notice that $V_\cl{M}^*( \beta I+\tilde K + \tilde F)V_\cl{M}:\cl{M}\longrightarrow \cl{M}$ is an operator on $\cl{M}$ and 
$
V_\cl{M}^*( \beta I+\tilde K + \tilde F)V_\cl{M} = \beta I_\cl{M} +\tilde K_\cl{M}+\tilde F_\cl{M}
$
is the sum of a nonnegative scalar multiple of Identity, a positive compact operator and a self-adjoint finite-rank operator on the fixed Hilbert space $\cl{M}$ which, by the preceding lemma, belongs to $\cl{N}_{(p,k)}$. Finally, since $p\in[1,\infty)$ and $k\in \bb{N}$ are arbitrary, the result holds for every $p\in[1,\infty)$ and for every $k\in \bb{N}$ and thus an operator of the above form belongs to $\cl{AN}_{(p,k)}$ for every $p\in[1,\infty)$ and for every $k\in \bb{N}$. This completes the proof.
\end{proof}

\begin{remark}
The proof of the Lemma \ref{backward_(p, k)strong(prereq)} uses an interesting result which deserves to be stated for its intrinsic interest. If $T\in \cl B(\cl H, \cl K),\,p\in[1,\infty)$, and $k\in \bb N$, then $T\in \cl N_{[k]} 
\iff T\in \cl N_{(p,k)}$. This result, together with the result stated in Remark \ref{N_[k]iffN_[pi, k]} yields the following:

Suppose $T\in \cl B(\cl H, \cl K),\,\pi\in \Pi,\,p\in[1,\infty)$, and $k\in \bb N$. Then the following statements are equivalent. 
\begin{enumerate}
\item $T\in \cl N_{[k]}$.
\item $T\in \cl N_{(p,k)}$.
\item $T\in \cl N_{[\pi, k]}$.
\end{enumerate}
\end{remark}

\section{Spectral Characterization of Operators in $\cl {AN}_{(p,k)}$}
By Proposition \ref{Tiff|T|_(p, k)}, the polar decomposition theorem (see Theorem \ref{PDT}) and the spectral theorem for positive operators $\cl{AN}_{(p, k)}$ (see Theorem \ref{SpThPAN_(p, k)}), we can safely consider the following theorem to be fully proved.

\begin{thm}[Spectral Theorem for Operators in $\cl{AN}_{(p, k)}$]\label{SpThAN_(p, k)}
Let $\cl{H}$ and $\cl{K}$ be complex Hilbert spaces of arbitrary dimensions, let $T\in \cl{B}(\cl{H},\cl{K})$ and let $T=U|T|$ be its polar decomposition.
Then the following statements are equivalent.
\begin{enumerate}
\item $T\in\cl{AN}_{(p,k)}$ for every $p\in[1,\infty)$ and for every $k\in \bb N$.
\item $T\in\cl{AN}_{(p,k)}$ for some $p\in[1,\infty)$ and for some $k\in \bb N$.
\item $|T|$ is of the form $|T|= \alpha I+K+F$, where $\alpha \geq 0, K$ is a positive compact operator and $F$ is self-adjoint finite-rank operator. 
\end{enumerate}
\end{thm}

\section{Absolutely Norming Operators on Symmetrically-Normed Ideals}\label{symmetrically-normed-ideals}

All the spectral characterization theorems we established in previous sections exhibit a common phenomenon: if an operator in $\cl B(\cl H, \cl K)$ belongs to one of the families, it belongs to all of them and its absolute value is of the form $\alpha I + K + F$, where $\alpha \geq 0, K$ is a positive compact operator and $F$ is self-adjoint finite-rank operator. Conversely, any operator in $\cl B(\cl H, \cl K)$ with its absolute value of the form $\alpha I + K + F$ is absolutely norming with respect to each of the norms we discussed. As a corollary of these results, we have that every positive operator of the form $\alpha I + K + F$ belongs to each of the families $\cl{AN}_{[k]}\cl B(\cl H), \cl{AN}_{[\pi,k]}\cl B(\cl H)$ and $\cl{AN}_{(p,k)}\cl B(\cl H)$. So, it might appear at this stage that with respect to \textit{every} symmetric norm $\|\cdot\|_s$ on $\cl B(\cl H)$, the positive operators on $\cl B(\cl H)$, that are of the above form,  are ``absolutely $s$-norming''. 


If this were true, if the end result of the analysis of absolutely norming operators with respect to various symmetric norms were that they are all of the same form, then this theory would be relatively straightforward. But this is not the case, for we prove the existence of a symmetric norm $\|\cdot\|_{\Phi_\pi^*}$ on $\cl B(\ell^2(\bb N))$ with respect to which the identity operator does not attain its norm. 

In order to discover this not-so-usual symmetric norm we need to put down some definitions and collect some facts that we will be using in the remaining portion of this paper. Henceforth, we assume $\cl H$ to be a separable Hilbert space. 

\begin{defn}[Symmetrically-Normed Ideals]
An ideal $\ff S$ of the algebra $\cl{B}(\cl{H})$ of operators on a complex Hilbert space is said to be a \textit{symmetrically-normed ideal} (or an \textit{s.n.ideal}) of $\cl{B}(\cl{H})$ if there is defined on it a symmetric norm $\|.\|_{\cl{\ff S}}$ which makes $\ff S$ a Banach space.
\end{defn}

We say that two ideals $\ff S_I$ and $\ff S_{II}$ \textit{coincide elementwise} if $\ff S_I$ and $\ff S_{II}$ consist of the same elements.
If the s.n. ideals $(\ff S_I,\|\cdot\|_I)$ and $(\ff S_{II}, \|\cdot\|_{II})$ coincide elementwise, then their norms are topologically equivalent.

\begin{defn}[Symmetric norming function]\cite[Chapter 3, Page $71$]{GK}
A function $\Phi:c_{00}\longrightarrow [0,\infty)$ is said to be \textit{symmetric norming function} (or, \textit{s.n. function}) if it satisfies the following properties.
 \begin{enumerate}
\item $\Phi(\xi)\geq 0$ for every $\xi:=(\xi_j)_j\in c_{00}$.
\item $\Phi(\xi)= 0 \iff \xi=0$ 
\item $\Phi(\alpha \xi)=|\alpha| \Phi(\xi)$ for every scalar $\alpha\in \bb R$ and $\xi\in c_{00}.$
\item $\Phi(\xi+\psi)\leq \Phi(\xi)+\Phi(\psi)$ for every $\xi,\psi\in c_{00}.$
\item $\Phi(1,0,0,...)=1.$
\item $\Phi(\xi_1,\xi_2,...,\xi_n,0,0,...)=\Phi(|\xi_{j_1}|,|\xi_{j_2}|,...,|\xi_{j_n}|,0,0,...)$ for every\linebreak $\xi\in c_{00}$ and $n\in \bb N$, where $j_1,j_2,...,j_n$ is any permutation of the integers $1,2,...,n$.
\end{enumerate}
 \end{defn}
 
 \begin{defn}[Equivalence of s.n. functions]\cite[Chapter 3, Page $76$]{GK}
Two s.n. functions $\Phi$ and $\Psi$ are said to be \emph{equivalent} if 
$$
\sup_{\xi\in c_{00}}\frac{\Phi(\xi)}{\Psi(\xi)}<\infty \, \, \text{ and }\, \, \sup_{\xi\in c_{00}} \frac{\Psi(\xi)}{\Phi(\xi)}<\infty.
$$
\end{defn}

\begin{defn}
Let $\Phi$ and $\Psi$ be two s.n. functions.
We say that $\Phi\leq \Psi$ if for every every $\xi \in c_{00}$, we have $\Phi(\xi)\leq \Psi(\xi)$.
\end{defn}

In Gohberg and Krein's text \cite{GK}, the reader can find a nice exposition on the theory of s.n. ideals, and since we will be dealing with this theory extensively, we have adopted their text as a reader's companion to the remaining portion of this article. Consequently, we have attempted to duplicate their notation wherever possible. We also will frequently use objects that are not defined here but whose definitions may be found in \cite{GK}.

 \begin{notation}[Notations and Terminologies]
 
 Consider the algebra $\cl B(\cl H)$ of operators on a separable Hilbert space $\cl H$. We use $\cl{B}_{00}(\cl{H})$ to denote the set of all finite-rank operators on $\cl{H}$. $\cl B_0(\cl{H})$ and $\cl B_1(\cl{H})$ are, respectively, used to denote the set of all compacts and trace class operators on $\cl H$. The trace norm is denoted by $\|\cdot\|_1$. These are indeed s.n. ideals. We denote by $c_0$ the space of all convergent sequences of real numbers with limit $0$. We let $c_{00}\subseteq c_0$ denote the linear subspace of  $c_0$ consisting of all sequences with a finite number of nonzero terms. By $c_{00}^+$ we denote the positive cone of $c_{00}$. Finally, we let $c_{00}^*\subseteq c_{00}^+$ denote the cone of all nonincreasing nonnegative sequences from $c_{00}$. To every vector $\xi=(\xi_j)_j\in c_{00}$, we associate the unique vector $\xi^*=(\xi^*_j)_j\in c_{00}^*$, where $\xi^*_j=|\xi_{n_j}|$ for every $j\in \bb N$ and $n_1,n_2,...,n_j,...$ is a permutation of the positive integers such that the sequence $(|\xi_{n_j}|)_j$ is nonincreasing. Since for any s.n. function $\Phi$,  we have 
$\Phi(\xi)=\Phi(\xi^*)\text{ for every }\xi\in c_{00},$ it follows that an s.n. function can be uniquely defined by its values on the cone $c_{00}^*$.
Consider the function $\Phi_\infty:c_{00}^*\longrightarrow [0,\infty)$ defined by $
\Phi_\infty(\xi)=\xi_1 \text{ for every } \xi=(\xi_j)_j\in c_{00}^*.$
This is an s.n. function and is called the \textit{minimal} s.n. function. Next we consider the function $\Phi_1:c_{00}^*\longrightarrow [0,\infty)$ defined by $
\Phi_1(\xi)=\sum_{j}\xi_j \text{ for every } \xi=(\xi_j)_j\in c_{00}^*.$
This is also an s.n. function and is called the \textit{maximal} s.n. function. If $\Phi$ is any s.n. function, then it has been shown that $\Phi_\infty\leq \Phi \leq \Phi_1$  (see \cite[Chapter 3, Section 3, Relation $3.12$, Page $76$]{GK}), which justifies the name ``minimal'' and ``maximal'' given to the s.n. functions $\Phi_\infty$ and $\Phi_1$ respectively.

   We would like to mention few classical results on s.n. ideals generated by an s.n. function. Let $\Phi$ be an arbitrary s.n. function and $c_{\Phi}$ be its natural domain (see \cite[Chapter 3, Page $80$]{GK} for the definition of natural domain of an s.n. function). To this function we associate the set $\ff S_{\Phi}$ of all operators $X\in \cl B_0(\cl{H})$ for which $s(X)=(s_j(X))_j\in c_\Phi$. Next we define a (symmetric) norm $\|\cdot\|_\Phi$ on $\ff S_\Phi$ by $\|X\|_\Phi:=\Phi(s(X))$ for every $X\in \ff S_\Phi$. $\ff S_{\Phi}$ thus denotes the s.n. ideal associated to $\Phi$. If $\Phi,\Psi$ are s.n. functions and $\ff S_\Phi,\ff S_\Psi$ are the s.n. ideals generated by these s.n. functions respectively, then $\ff S_\Phi$ and $\ff S_\Psi$ \emph{coincide elementwise} (that is, consist of the same elements) if and only if $\Phi$ and $\Psi$ are equivalent. In particular, if $\Phi$ is an s.n. function equivalent to  $\Phi_1$, then  $\ff S_{\Phi}$ and $\cl B_1(\cl H)$ coincide elementwise and when $\Phi$ is equivalent to $\Phi_\infty$, $\ff S_{\Phi}$ and $\cl B_0(\cl H)$ coincide elementwise. For a given s.n. function there is a notion of its adjoint. The \emph{adjoint} $\Phi^*$ of the s.n. function $\Phi$ is given
 by 
$$
\Phi^*(\eta)=\max \left\{\sum_{j}\eta_j\xi_j: \xi\in c_{00}^*, \Phi(\xi)=1 \right\}, \text{ for every }\eta \in c_{00}^*,
$$ 
and is itself an s.n. function. The adjoint of $\Phi^*$ is $\Phi$. In particular, the minimal and maximal s.n. functions are the adjoint of each other, that is, 
$\Phi_1^*=\Phi_{\infty}$ and $\Phi_\infty^*=\Phi_1.$ Therefore, when an s.n. function is equivalent to the maximal(minimal) one, its adjoint is equivalent to the minimal(maximal) one. In \cite[Chapter 3, section 14]{GK} there are examples of s.n. ideals in which the set $\cl B_{00}(\cl H)$ of finite-rank operators is not dense. This circumstance suggests the necessity of introducing the subspace $\ff S^{(0)}_\Phi$, the norm closure of the set $\cl B_{00}(\cl H)$ in the norm of $\ff S_\Phi$, that is, 
$$\ff S^{(0)}_\Phi:=\text{clos}_{\|\cdot\|_{\Phi}}[\cl B_{00}(\cl H)].$$
 \end{notation}
 
In our exposition we will need the following elementary piece of folklore from \cite{GK}, whose proof we leave to the reader.

 \begin{prop}\cite[Chapter 3, Theorems 12.2 and 12.4]{GK}\label{Gohb-Krein}
 Let $\Phi$ be an arbitrary s.n. function. 
 \begin{enumerate} 
 \item If $\Phi$ is not equivalent to the maximal s.n. function, then the general form of a continuous linear functional $f$ on the separable space $\ff S^{(0)}_\Phi$ is given by $f(X)=\tr(AX)$ for some $A\in \ff S_{\Phi^*}$ and 
 $$\|f\|:=\sup\{|\tr(AX)|:X\in \ff S^{(0)}_\Phi, \|X\|_\Phi\leq 1\}=\|A\|_{\Phi^*}.$$ Thus, the space adjoint to the space $\ff S^{(0)}_{\Phi}$ is isometrically isomorphic to $\ff S_{\Phi^*}$, that is, $\ff S_\Phi^{(0)^*}\cong \ff S_{\Phi^*}.$ In particular, if both functions $\Phi$ and $\Phi^*$ are mononormalizing, the space $\ff S_{\Phi}$ is reflexive.\\
  \item If $\Phi$ is equivalent to the maximal s.n. function, then the general form of a continuous linear functional $f$ on the separable space $\ff S_\Phi$ is given by $f(X)=\tr(AX)$ for some $A\in \cl B(\cl H)$ and 
 $$\|f\|:=\sup\{|\tr(AX)|:X\in \ff S_\Phi, \|X\|_\Phi\leq 1\}=\|A\|_{\Phi^*}.$$ Thus, the dual space $\ff S_\Phi^*$ is isometrically isomorphic to
$(\cl B(\cl H),\|\cdot\|_{\Phi^*}),$ that is, 
$\ff S_\Phi^*\cong (\cl B(\cl H),\|\cdot\|_{\Phi^*}).$
 \end{enumerate}
 \end{prop}

Let $\Phi$ be arbitrary s.n. function equivalent to the maximal one. The remaining part of this section is intended to establish the notion of ``$\Phi^*$-norming'' operators on the s.n. ideal $(\cl B(\cl H),\|\cdot\|_{\Phi^*})$ which agrees with the Definitions \ref{N_[k]Def}, \ref{N_[pi,k]Def}, and \ref{N_(p,k)Def}, and in essentially the same spirit, generalizes the concept. To meet this purpose we need to establish a sequence of propositions which provide us with the machinery required to convert the concept of norming operators in the language of s.n. ideals. 
 
\subsection{Norming operators on $\cl B(\cl H)$} 

\begin{thm}\label{NisDone}
Let $T\in \cl B(\cl H)$. Then $T\in\cl N$ in the sense of the Definition \ref{Def1} if and only if there exists an operator $K\in \cl B_1(\cl H)$ with $\|K\|_1=1$ such that $|\tr(TK)|=\|T\|$.
\end{thm}
\begin{proof}
We first assume that $T\in\cl N$. Then there exists $x$ in the unit sphere of $\cl H$ such that $\|Tx\|=\|T\|.$ Let $y=\frac{Tx}{\|Tx\|}$ and define a rank-one operator $K_0:=x\otimes y\in \cl B_1(\cl H)$. Notice that $\|K_0\|_1=1$ and $
|\tr(TK_0)|=|\tr(T(x\otimes y))|=|\tr(Tx\otimes y)|=|\inner{Tx}{y}|=\|Tx\|=\|T\|$ which proves the forward implication. To see the backward implication, we assume that there exists an operator $K\in \cl B_1(\cl H)$ such that $\| K\|_1=1$ and $|\tr(TK)|=\|T\|$. Since $\cl B_1(\cl H)\subseteq \cl B_0(\cl H)$, Schmidt expansion allows us to write $ K=\sum_{j=1}^{\text{rank } K}s_j(K)(x_j\otimes y_j)$, where $\{x_j\}$
is an orthonormal basis of $\text{clos}[\text{ran }{K}]$ and $\{y_j\}$ is an orthonormal basis of $\text{clos}[\text{ran }{|K|}]$. We now have
\begin{multline*}
\|T\|=|\tr(TK)|=\left|\tr\left(T\left(\sum_{j=1}^{\text{rank } K}s_j(K)(x_j\otimes y_j)\right)\right)\right|\\
=\left|\sum_{j=1}^{\text{rank } K}s_j(K)\tr(T x_j\otimes y_j)\right|=\left|\sum_{j=1}^{\text{rank } K}s_j(K)\inner{Tx_j}{y_j}\right|\\
\leq \sum_{j=1}^{\text{rank } K}s_j(K)|\inner{Tx_j}{y_j}|\leq \sum_{j=1}^{\text{rank } K}s_j(K)\|T\|=\|T\|\|K\|_1=\|T\|,
\end{multline*}
which forces all inequalities to be equalities, so 
$$
\sum_{j=1}^{\text{rank } K}s_j(K)\|Tx_j\|= \sum_{j=1}^{\text{rank } K}s_j(K)\|T\|,
$$
which implies that $\sum_{j=1}^{\text{rank } K}s_j(K)(\|T\|-\|Tx_j\|)=0$. Notice that, for every $j$, $s_j(K)>0$ and $\|T\|-\|Tx_j\|\geq 0$. Thus for every $j$ we have $\|T\|=\|Tx_j\|$ which implies that $T\in \cl N.$ This completes the proof.
\end{proof}
\subsection{$[2]$-norming operators on $\left(\cl B(\cl H), \|\cdot\|_{[2]}\right)$}

\begin{lemma}
If $\Phi$ and $\Psi$ are s.n. functions defined as
\begin{align*}
\Phi(\eta)&=\max \left\{ \eta_1, \frac{\sum_{j}\eta_j}{2}\right\};\\
\Psi(\xi) &=\xi_1+\xi_2;
\end{align*}
with $\eta=(\eta_i)_{i\in \bb N}, \text{ and } \xi=(\xi_j)_{j\in \bb N} \in c_{00}^*$, then $\Phi$ and $\Psi$ are mutually adjoint, that is, $\Phi^*=\Psi$ and $\Psi^*=\Phi$.
\end{lemma}
\begin{proof}
Recall that the adjoint $\Psi^*$ of the s.n. function $\Psi$ is given
 by $$\Psi^*(\eta)=\max \left\{\sum_{j}\eta_j\xi_j: \xi\in c_{00}^*, \Psi(\xi)=1 \right\}, \text{ for every }\eta \in c_{00}^*,$$
 which can be rewritten as
\begin{align*}
\Psi^*(\eta)
=\max \left\{\sum_{j}\eta_j\xi_j: \xi\in c_{00}^*, \xi_1+\xi_2=1 \right\}.
\end{align*}
Since $\xi \in c_{00}^*$, it is a nonincreasing sequence which allows us to infer that  
\begin{align*}
\Psi^*(\eta)&=\max \left\{\eta_1\xi_1+\left(\sum_{j \neq 1}\eta_j\right)\xi_2:  \xi_1\geq \xi_2, \xi_1+\xi_2=1 \right\}\\
&=\max \left\{\left(\eta_1-\left(\sum_{j\neq 1}\eta_j\right)\right)\xi_1+\left(\sum_{j\neq 1}\eta_j\right):  \xi_1\in \left[\frac{1}{2}, 1\right] \right\}\\
&=\begin{cases} 
     \eta_1 & \text{ if } \eta_1\geq \sum_{j\neq 1}\eta_j  \\
      \frac{\sum_{j}\eta_j}{2} & \text{ if } \eta_1\leq \sum_{j\neq 1}\eta_j
   \end{cases}\\
   &=\max \left\{ \eta_1, \frac{\sum_{j}\eta_j}{2}\right\},
\end{align*}  
which is indeed equal to $\Phi(\eta)$; here the penultimate equality is a direct consequence of the fact that a function of the form $ax+b, x\in [\frac{1}{2},1]$ achieves its maximum at $x=1$ if $a> 0$ and at $x=\frac{1}{2}$ if $a<0$. The final equality arrives from the following simple calculation: $\eta_1\geq \sum_{j\neq 1}\eta_j \iff \eta_1 \geq \frac{\sum_{j}\eta_j}{2}$ and  $\eta_1\leq \sum_{j\neq 1}\eta_j \iff \eta_1 \leq \frac{\sum_{j}\eta_j}{2}$. This completes the proof.
\end{proof}

\begin{remark}
   It is easy to see that $\Psi$ is equivalent to the minimal s.n. function and that it corresponds to the Ky Fan $2$-norm on $\cl B(\cl H)$.   
Notice that
$$
\sup_n\left\{\frac{n}{\Psi^*(\underbrace{1,1,...,1}_{n },0,0,...)}\right\}=
\sup_n\left\{\frac{n}{\max\{1, \frac{n}{2}\}}\right\}=2<\infty,
$$
which implies that $\Phi=\Psi^*$ is equivalent to the maximal s.n. function $\Phi_1$.
Consequently, the dual $\ff{S}_{\Phi}^*$ of the s.n. ideal $\ff{S}_{\Phi}$ is isometrically isomorphic to the space  ($\cl B(\cl H), \|\cdot\|_{\Phi^*}$), that is, $\ff{S}_{\Phi}^*\cong  (\cl B(\cl H), \|\cdot\|_{\Phi^*}).$
Moreover, the s.n. ideal $\ff S_{\Phi}$ generated by $\Phi$ and the ideal $\cl B_1(\cl H)$ of trace class operators coincide elementwise. Clearly, $\Phi$ and $\Phi^*$ are s.n. functions considered on their natural domain instead of merely $c_{00}^*$.
\end{remark}

\begin{thm}\label{KyFan2IsDone}
Let $T\in \cl B(\cl H), \Phi$ be an s.n. function equivalent to the maximal s.n. function, defined by 
$$
\Phi(\eta)=\max \left\{ \eta_1, \frac{\sum_{j}\eta_j}{2}\right\},
$$ 
where $\eta=(\eta_i)_{i\in \bb N}\in c_{\Phi}$, and let $\Phi^*$ be its dual norm so that
$$
\ff{S}_{\Phi}^*\cong  (\cl B(\cl H), \|\cdot\|_{\Phi^*}),
$$
 with $\|T\|_{\Phi^*}=\|T\|_{[2]}$ for every $T\in \cl B(\cl H)$. 
Then $T\in\cl N_{[2]}$ in the sense of the Definition \ref{N_[k]Def} if and only if there exists an operator $K\in \ff S_\Phi=\cl B_1(\cl H)$ with $\|K\|_{\Phi}=1$ such that $|\tr(TK)|=\|T\|_{\Phi^*}$.
\end{thm}

\begin{proof}
First we assume that $T\in\cl N_{[2]}$. There exist $x_1,x_2\in \cl H$ with $\|x_1\|=\|x_2\|=1$ and $x_1\perp x_2$ such that $\|T\|_{\Phi^*}=\|T\|_{[2]}=\|Tx_1\|+\|Tx_2\|.$ Let 
$$
y_1=\frac{Tx_1}{\|Tx_1\|},\,\, y_2=\frac{Tx_2}{\|Tx_2\|} \text{ and define } K:=\sum_{j=1}^2x_j\otimes y_j.
$$
That $K\in \cl B_1(\cl H)$ and $s_1(K)=s_2(K)=1$ is obvious, so $\|K\|_\Phi=1$. Then 
\begin{multline*}
|\tr(TK)|=|\tr(T(\sum_{j=1}^2x_j\otimes y_j))|=|\tr(\sum_{j=1}^2Tx_j\otimes y_j)|\\
=|\sum_{j=1}^2\inner{Tx_j}{y_j}|=\left|\sum_{j=1}^2\inner{Tx_j}{\frac{Tx_j}{\|Tx_j\|}}\right|=\sum_{j=1}^2\|Tx_j\|=\|T\|_{\Phi^*},
\end{multline*}
finishes the proof of the forward implication. To see the backward implication, we assume that there exists an operator $K\in \cl B_1(\cl H)$ with $\|K\|_\Phi=1$ such that  $|\tr(TK)|=\|T\|_{\Phi^*}$. Let $\alpha:=\|T\|_{\Phi^*}=\|T\|_{[2]}=s_1(T)+s_2(T)$. Consequently,
\begin{multline*}
\alpha =\|T\|_{\Phi^*}=|\tr(TK)|
=\left|\tr\left(T\left(\sum_{j=1}^{\text{rank } K}s_j(K)(x_j\otimes y_j)\right)\right)\right|\\
=\left|\sum_{j=1}^{\text{rank } K}s_j(K)\inner{Tx_j}{y_j}\right|
\leq \sum_{j=1}^{\text{rank } K}s_j(K)|\inner{Tx_j}{y_j}|\\
=\inner{\begin{bmatrix}
s_1(K)\\
\vdots\\
s_j(K)\\
\vdots
\end{bmatrix}}{\begin{bmatrix}
|\inner{Tx_1}{y_1}|\\
\vdots\\
|\inner{Tx_j}{y_j}|\\
\vdots
\end{bmatrix}}
\leq \Phi ((s_j(K))_j)\Phi^* ((|\inner{Tx_j}{y_j}|)_j)\\
=\Phi^* ((|\inner{Tx_j}{y_j}|)_j)
\leq\|T\|_{\Phi^*}=\alpha.
\end{multline*}
This forces $\Phi^* ((|\inner{Tx_j}{y_j}|)_j)=\alpha$. That is, $\|(|\inner{Tx_j}{y_j}|)_j\|_{\Phi^*}=\alpha$. This observation along with the fact that the sequence $(s_j(K))_j$ is nonincreasing implies that the sequence $(|\inner{Tx_j}{y_j}|)_j$ is also nonincreasing, that is, $|\inner{Tx_1}{y_1}|\geq |\inner{Tx_2}{y_2}|\geq ... \geq |\inner{Tx_j}{y_j}|\geq ...$; for if it is not, then there exists $\ell \in \bb N$ such that $|\inner{Tx_\ell}{y_\ell}|<|\inner{Tx_{\ell+1}}{y_{\ell+1}}|$ which yields 
\begin{multline*}
\alpha=
\inner{\begin{bmatrix}
\vdots\\
s_\ell(K)\\
s_{\ell+1}(K)\\
\vdots
\end{bmatrix}}{\begin{bmatrix}
\vdots\\
|\inner{Tx_\ell}{y_\ell}|\\
|\inner{Tx_{\ell+1}}{y_{\ell+1}}|\\
\vdots
\end{bmatrix}}
<\inner{\begin{bmatrix}
\vdots\\
s_\ell(K)\\
s_{\ell+1}(K)\\
\vdots
\end{bmatrix}}{\begin{bmatrix}
\vdots\\
|\inner{Tx_{\ell+1}}{y_{\ell+1}}|\\
|\inner{Tx_\ell}{y_\ell}|\\
\vdots
\end{bmatrix}}\\
\leq \Phi \left(\begin{bmatrix}
\vdots\\
s_\ell(K)\\
s_{\ell+1}(K)\\
\vdots
\end{bmatrix}\right)\Phi^*\left(\begin{bmatrix}
\vdots\\
|\inner{Tx_{\ell+1}}{y_{\ell+1}}|\\
|\inner{Tx_\ell}{y_\ell}|\\
\vdots
\end{bmatrix}\right)\\
= \Phi \left(\begin{bmatrix}
\vdots\\
s_\ell(K)\\
s_{\ell+1}(K)\\
\vdots
\end{bmatrix}\right)\Phi^*\left(\begin{bmatrix}
\vdots\\
|\inner{Tx_\ell}{y_\ell}|\\
|\inner{Tx_{\ell+1}}{y_{\ell+1}}|\\
\vdots
\end{bmatrix}\right)=\alpha,
\end{multline*}
which is indeed a contradiction.
Consequently, 
\begin{multline*}
\alpha=\|(|\inner{Tx_j}{y_j}|)_j\|_{\Phi^*}=|\inner{Tx_1}{y_1}| + |\inner{Tx_2}{y_2}|\\
\leq \|Tx_1\|+\|Tx_2\|\leq s_1(T)+s_2(T)=\alpha,
\end{multline*}
which forces $\|Tx_1\|+\|Tx_2\|=s_1(T)+s_2(T)$ thereby establishing that $T\in \cl N_{[2]}$. This completes the proof.
\end{proof}

\subsection{$[\pi, 2]$-norming operators on $\left(\cl B(\cl H), \|\cdot\|_{[\pi, 2]}\right)$}

\begin{lemma}
Given $\pi=(\pi_j)_{j\in \bb N}\in \Pi$, if $\Phi$ and $\Psi$ are s.n. functions defined as
\begin{align*}
\Phi(\eta)&=\max \left\{ \eta_1, \frac{\sum_{j}\eta_j}{1+\pi_2}\right\};\\
\Psi(\xi) &=\xi_1+\pi_2\xi_2;
\end{align*}
where $\eta=(\eta_i)_{i\in \bb N},\text{ and } \xi=(\xi_j)_{j\in \bb N} \in c_{00}^*$, then $\Phi$ and $\Psi$ are mutually adjoint, that is, $\Phi^*=\Psi$ and $\Psi^*=\Phi$. Moreover, $\phi$ is equivalent to the maximal s.n. function and $\Psi$ to the minimal.
\end{lemma}

\begin{proof}
Without much hassle it can be shown that
$$\Psi^*(\eta)=\max \left\{\eta_1t+\frac{(\sum_{j\neq 1}\eta_j)(1-t)}{\pi_2}:  t\in \left[\frac{1}{1+\pi_2}, 1\right] \right\}.$$
But 
\begin{align*}
\eta_1t+\frac{(\sum_{j\neq 1}\eta_j)(1-t)}{\pi_2}&= \left(\eta_1-\left(\frac{\sum_{j\neq 1}\eta_j}{\pi_2}\right)\right)t+\left(\frac{\sum_{j\neq 1}\eta_j}{\pi_2}\right)\\
&=\begin{cases} 
     \eta_1 & \text{ when } t=1  \\
     \frac{\sum_{j}\eta_j}{1+\pi_2} & \text{ when } t=\frac{1}{1+\pi_2},
   \end{cases}
\end{align*}
which implies that 
$$
\Phi(\eta)=\Psi^*(\eta)=\max \left\{ \eta_1, \frac{\sum_{j}\eta_j}{1+\pi_2}\right\}.
$$
The final part of the assertion is trivial.
\end{proof}
Using this result we establish the following theorem.
\begin{thm}\label{KyFanPi2IsDone}
Given $\pi=(\pi_j)_{j\in \bb N}\in \Pi$, let $T\in \cl B(\cl H)$ and $\Phi$ be an s.n. function equivalent to the maximal s.n. function, defined by 
$$\Phi(\eta)=\max \left\{ \eta_1, \frac{\sum_{j}\eta_j}{1+\pi_2}\right\},$$ where $\eta=(\eta_i)_{i\in \bb N}\in c_{\Phi}$, and let $\Phi^*$ be its dual norm so that 
$$
\ff{S}_{\Phi}^*\cong  (\cl B(\cl H), \|\cdot\|_{\Phi^*}),
$$
 with $\|T\|_{\Phi^*}=\|T\|_{[\pi, 2]}$ for every $T\in \cl B(\cl H)$. 
Then $T\in\cl N_{[\pi,2]}$ in the sense of the Definition \ref{N_[pi,k]Def} if and only if there exists an operator $K\in \ff S_\Phi=\cl B_1(\cl H)$ with $\|K\|_{\Phi}=1$ such that $|\tr(TK)|=\|T\|_{\Phi^*}$.
\end{thm}

\begin{proof}
Let $\{x_1,x_2\}\in \cl H$ be an orthonormal set such that $\|T\|_{\Phi^*}=\|T\|_{[\pi, 2]}=\|Tx_1\|+\pi_2\|Tx_2\|$ and let $$y_1=\frac{Tx_1}{\|Tx_1\|},\,\, y_2=\frac{Tx_2}{\|Tx_2\|}.$$
Define $ K:=(x_1\otimes y_1)+\pi_2(x_2\otimes y_2)$. Clearly, $K\in \cl B_1(\cl H)$ and $s_1(K)=1,\, s_2(K)=\pi_2$ with $\|K\|_\Phi=1$. Then
\begin{multline*}
|\tr(TK)|
=\left|\inner{Tx_1}{\frac{Tx_1}{\|Tx_1\|}}+\pi_2\inner{Tx_2}{\frac{Tx_2}{\|Tx_2\|}}\right|\\
=\|Tx_1\|+\pi_2\|Tx_2\|=\|T\|_{\Phi^*},
\end{multline*}
proves the forward implication. 
Next we assume that there exists an operator $K\in \cl B_1(\cl H)$ with $\|K\|_\Phi=1$ such that  $|\tr(TK)|=\|T\|_{\Phi^*}$. Let $\alpha:=\|T\|_{\Phi^*}=\|T\|_{[\pi, 2]}=s_1(T)+\pi_2s_2(T)$. Slightly tweaking the proof of Theorem \ref{KyFan2IsDone} allows us to infer 
$\|(|\inner{Tx_j}{y_j}|)_j\|_{\Phi^*}=\alpha$  and that the sequence $(|\inner{Tx_j}{y_j}|)_j$ is nonincreasing. This yields
\begin{multline*}
\alpha=\|(|\inner{Tx_j}{y_j}|)_j\|_{\Phi^*}=|\inner{Tx_1}{y_1}| +\pi_2 |\inner{Tx_2}{y_2}|\\
\leq \|Tx_1\|+\pi_2\|Tx_2\|\leq s_1(T)+\pi_2s_2(T)=\alpha,
\end{multline*}
which forces $\|Tx_1\|+\pi_2\|Tx_2\|=s_1(T)+\pi_2s_2(T)$ thereby establishing that $T\in \cl N_{[\pi, 2]}$. This completes the proof.
\end{proof}

Given an arbitrary s.n. function $\Phi$ that is equivalent to the maximal s.n. function, we are now ready to establish the definition of operators in $\cl B(\cl H)$ that attain their $\Phi^*$-norm.

\begin{defn}\label{PhiNorming}
Let $\Phi$ be an s.n. function equivalent to the maximal s.n. function. An operator $T\in (\cl B(\cl H),\|\cdot\|_{\Phi^*})$ is said to be \emph{$\Phi^*$-norming} if there exists an operator $K\in \ff S_\Phi=\cl B_1(\cl H)$ with $\|K\|_{\Phi}=1$ such that $|\tr(TK)|=\|T\|_{\Phi^*}.$ We let $\cl N_{\Phi^*}(\cl H)$ denote the set of $\Phi^*$-norming operators in $\cl B(\cl H)$.
\end{defn}

The following proposition is a trivial observation and its prinicipal significance lies in the fact that it can be taken as a new equivalent definition of $\Phi^*$-norming operators in $\cl B(\cl H)$.

\begin{prop}\label{PhiNorming-alternative}
Let $\Phi$ be an s.n. function equivalent to the maximal s.n. function and let $\Phi^*$ be its dual norm so that $\ff S_{\Phi}^*\cong(\cl B(\cl H), \|\cdot\|_{\Phi^*})$. If $T\in \cl B(\cl H)$ is identified with $f_T\in \ff S_{\Phi}^*$, then the following statements are equivalent.
\begin{enumerate}
\item $T\in \cl N_{\Phi^*}(\cl H)$.
\item $f_T$ attains its norm.
\end{enumerate}
\end{prop}

\begin{notation}
We let $\cl N(\ff S_\Phi, \bb C)$ denote the set of functionals on $\ff S_{\Phi}$ that attain their norm. 
\end{notation}

Notice that Theorem \ref{TVM-TPM-k} is a reformulation of the definition of an absolutely $[k]$-norming operators in $\cl B(\cl H,\cl K)$ by identifying $T|_{\cl M}\in \cl B(\cl M, \cl K)$ with $TP_{\cl M}\in \cl B(\cl H, \cl K)$; and so are Theorems \ref{TVM-TPM-Pi-k} and \ref{TVM-TPM-p-k} for absolutely $[\pi, k]$-norming  and absolutely $(p,k)$-norming operators, respectively. These reformulations motivate the following definition.

\begin{defn}\label{Absolutely-PhiNorming}
Let $\Phi$ be an s.n. function equivalent to the maximal s.n. function. An operator $T\in (\cl B(\cl H),\|\cdot\|_{\Phi^*})$ is said to be \emph{absolutely $\Phi^*$-norming} if for every nontrivial closed subspace $\cl M$ of $\cl H$, $TP_{\cl M}\in \cl B(\cl H)$ is $\Phi^*$-norming. We let $\cl{AN}_{\Phi^*}(\cl H)$ denote the set of absolutely $\Phi^*$-norming operators in $\cl B(\cl H)$.
\end{defn}

\begin{exam}
For any $\pi\in \Pi$ and for any $k\in \bb N$, choose $\Phi$ to be the s.n. function such that $\Phi^*=\|\cdot\|_{[\pi, k]}$. Then $T\in (\cl B(\cl H),\|\cdot\|_{\Phi^*})$ belongs to $\cl{AN}_{\Phi^*}(\cl H)$ if and only if $|T|$ is of the form $|T|= \alpha I +F+K$, where $\alpha \geq 0, K$ is a positive compact operator and $F$ is self-adjoint finite-rank operator.

Given $p\in [1,\infty)$ and $k\in \bb N$, choose $\Phi$ to be the s.n. function such that $\Phi^*=\|\cdot\|_{(p,k)}$. Then $T\in \cl{AN}_{\Phi^*}(\cl H)$ if and only if $|T|$ is of the form $|T|= \alpha I +F+K$, where $\alpha \geq 0, K$ is a positive compact operator and $F$ is self-adjoint finite-rank operator. 
\end{exam}

\begin{proof}[\textbf{Proof of Proposition \ref{Identity-nonnorming}}]
Let $\{e_i\}_{i\in \mathbb{N}}$ be the canonical orthonormal basis of the Hilbert space $\ell^2(\mathbb{N})$, and let $\pi=(\pi_n)_{n\in \mathbb{N}}$ be a strictly decreasing convergent sequence of positive numbers with $\pi_1=1$ such that $\lim_{n}\pi_n>0$.
Let us define a symmetrically norming function $\Phi_\pi$ by $\Phi_\pi(\xi_1,\xi_2,...)=\sum_{j}\pi_j\xi_j$. Notice that for every $n\in \bb N,$ we have
$$\frac{n}{\Phi_{\pi}(\underbrace{1,...,1}_{n \text{ times}},0,0,...)}= \frac{n}{1+\pi_2+...+\pi_n}<\frac{1}{\lim_{n}\pi_n},$$ which implies
$$
\sup_n\left\{\frac{n}{\Phi_{\pi}(\underbrace{1,...,1}_{n \text{ times}},0,0,...)}\right\}< \sup_n\left\{\frac{1}{\lim_{n}\pi_n}\right\}<\infty.
$$
$\Phi_\pi$ is thus equivalent to the maximal symmetric norming function $\Phi_1.$ The dual $\ff{S}_{\Phi_\pi}^*$ of the symmetrically normed ideal $\ff{S}_{\Phi_\pi}$ is thus isometrically isomorphic to ($\cl B(\ell^2), \|\cdot\|_{\Phi_\pi^*}$), that is, $\ff{S}_{\Phi_\pi}^*\cong  (\cl B(\ell^2), \|\cdot\|_{\Phi_\pi^*}),$
and the $\|\cdot\|_{\Phi_\pi^*}$ norm for any operator $T\in \cl B(\ell^2)$ is given by 
$\|T\|_{\Phi_\pi^*}= \sup\{|\tr(TK)|:K\in \ff S_{\Phi_\pi},\, \|K\|_{\Phi_\pi}= 1\}.$ But the ideal $\cl B_1(\cl H)$ and $\ff S_{\Phi_\pi}$ coincide elementwise and hence 
$\|T\|_{\Phi_\pi^*}= \sup\{|\tr(TK)|:K\in \cl B_1(\ell^2),\, \|K\|_{\Phi_\pi}= 1\}.$
We will show that $I$ does not attain its $\Phi_\pi^*$-norm in $\cl B(\ell^2)$. To show this, we assume that $I\in \cl N_{\Phi_\pi^*}(\ell^2)$, and we deduce a contradiction from this assumption. 

We first claim that $
\alpha:=\sup\{|\tr(K)|:K\in \cl B_1(\ell^2),\, \|K\|_{\Phi_\pi}=1\}=\sup\{|\tr(K)|:K\in \cl B_1(\ell^2),\,K=\text{diag}\{s_1(K),s_2(K),...\},\, \|K\|_{\Phi_\pi}= 1\}=:\beta.
$
That $\beta \leq \alpha$, is a trivial observation. Let us choose an operator $T\in \cl B_1(\ell^2)$ with $\|T\|_{\Phi_\pi}= 1.$ We define 
$$
\tilde T:=\begin{pmatrix}
s_1(T)\\
&s_2(T)&&\text{\huge 0}\\
&&\ddots&&\\
&\text{\huge 0}& & s_j(T)&\\
&&&&\ddots
\end{pmatrix}.
$$
 Notice that for every $j$, we have $s_j(\tilde T)=s_j(T)$ and thus $\|\tilde T\|_{\Phi_\pi} =\|T\|_{\Phi_\pi}$ which implies that $\tilde T\in \cl B_1(\ell^2)$. Consequently,
 $
 |\tr(T)|\leq |\tr(\tilde T)|
 \leq \sup\{|\tr(K)|:K\in \cl B_1(\ell^2),\,K=\text{diag}\{s_1(K),s_2(K),...\},\, \|K\|_{\Phi_\pi}= 1\}
=\beta.
 $
 It follows then that $\alpha \leq \beta$, thereby establishing our claim. Next we observe that since the trace of a positive trace class diagonal operator is precisely the sum of its singular values, we have 
$
\sup\{|\tr(K)|:K\in \cl B_1(\ell^2),\,K=\text{diag}\{s_1(K),s_2(K),...\},\, \|K\|_{\Phi_\pi}= 1\}=\sup\{\sum_js_j(K):K\in \cl B_1(\ell^2),\,K=\text{diag}\{s_1(K),s_2(K),...\},\, \|K\|_{\Phi_\pi}= 1\}.
$
 
 The above two observations leads us to realize that if $I\in \cl N_{\Phi_\pi^*}$, then the supremum, 
 $\sup\{\sum_{j}s_j(K):K\in \cl B_1(\ell^2),\, K=\text{diag}\{s_1(K),s_2(K),...\},\,\|K\|_{\Phi_\pi}=1\},$ 
 is attained, that is, 
  there exists $K_0=\text{diag}\{s_1(K_0),s_2(K_0),...\}\in \cl B_1(\cl H)$ with $\sum_{j}a_js_j(K_0)=1$ such that $\|I\|_{\Phi_\pi^*}=|\tr(K_0)|=\sum_js_j(K_0).$
Since $K_0\in \cl B_1(\ell^2)\subseteq \cl B_0(\ell^2)$, we have $\lim_{j\rightarrow \infty} s_j(K_0)=0$. This forces the existence of a natural number $M$ such that $s_M(K_0)>s_{M+1}(K_0).$ All that remains is to show the existence of an operator $\tilde K\in \cl B_1(\ell^2), \,\|\tilde K\|_{\Phi_\pi}=1$ of the form $\tilde K=\text{diag}\{s_1(\tilde K),s_2(\tilde K),...\}$ such that $\sum_is_i(\tilde K)> \sum_js_j(K_0)$. 
If we define a sequence $(t_i)_{i\in \bb N}$ by 
$$
t_i= \begin{cases} 
    \frac{\sum_{j=M}^{M+1}\pi_js_j(K_0)}{\sum_{j=M}^{M+1}\pi_j} & \text{ if } M\leq i\leq  M+1, \\
      s_i & \text{ if } i<  M \text{ or }i> M+1,
   \end{cases}
$$
then it follows that $s_{M+1}(K_0)< t_M=t_{M+1}<s_M(K_0),$ and that $
\sum_i \pi_it_i=\sum_j\pi_js_j(K_0)=1$ so that
\begin{align*}
\sum_{i=M}^{M+1}t_i =2\left(\frac{\sum_{j=M}^{M+1}\pi_js_j(K_0)}{\sum_{j=M}^{M+1}\pi_j}\right)>\sum_{j=M}^{M+1}s_j(K_0),
\end{align*}
which implies that 
$$
\sum_it_i=\sum_{i=1}^{M+1}t_i+\sum_{i>M+1}t_i
>\sum_{j=1}^{M+1}s_j(K_0)+\sum_{j>M+1}s_j(K_0)
=\sum_js_j(K_0).
$$
Setting $$
\tilde K:=\begin{pmatrix}
t_1\\
&t_2&&\text{\huge 0}\\
&&\ddots&&\\
&\text{\huge 0}& & t_i&\\
&&&&\ddots
\end{pmatrix},
$$
we observe that $\|\tilde K\|_{\Phi_\pi}=1<\infty$ so that $\tilde K\in \cl B_1(\ell^2)$ and that it is of the form $\tilde K=\text{diag}(s_1(\tilde K),s_2(\tilde K),...)$ where 
$s_i(\tilde K)=t_i$ for every $i$. But then
$
|\tr(\tilde K)|=\sum_is_i(\tilde K)=\sum_it_i>\sum_js_j(K_0)=|\tr(K_0)|=\|I\|_{\Phi_\pi^*},
$
which contradicts the assumption that $|\tr(K_0)|$ is the supremum of the set 
$$
\left\{\sum_{j}s_j(K):K\in \ff S_{\Phi_\pi},\, K=\text{diag}\{s_1(K),s_2(K),...\}\,\|K\|_{\Phi_\pi}=1\right\}.
$$ Since the operator $K_0$ with which we began our discussion is arbitrary, it follows that for any given operator in $\ff S_{\Phi_\pi}$ with unit norm, one can find another operator in $ \ff S_{\Phi_\pi}$ with unit norm with  trace of larger magnitude and hence the supremum of the above set can never be attained. This shows that the identity operator $I$ does not attain its norm.   
\end{proof}

We wish to prove Theorem \ref{compacts-are-norming}, which can be thought of as an analogue of Proposition \ref{compacts-are-absolutely[k]norming} except that we are in the setting of $(\cl B(\cl H),\|\cdot\|_{\Phi^*})$ instead of $(\cl B(\cl H, \cl K), \|\cdot\|)$ with $\Phi$ being an s.n. function equivalent to the maximal s.n.function. Before proving this theorem we will need the following lemma.

 \begin{lemma}\label{prereq-doubledual}
 Let $\Phi$ be an arbitrary s.n. function equivalent to the maximal s.n. function. Then 
$$
 (\cl B_0(\cl H), \|\cdot\|_{\Phi^*})^{**}\cong (\cl B(\cl H),\|\cdot\|_{\Phi^*}).
$$
  \end{lemma} 
\begin{proof}
Since $\Phi$ is equivalent to  the maximal s.n. function $\Phi_1$, the first part of the Proposition \ref{Gohb-Krein} guarantees $ (\ff S_\Phi, \|\cdot\|_\Phi)^*\cong (\cl B(\cl H),\|\cdot\|_{\Phi^*})$ and $(\ff S_{\Phi^*}, \|\cdot\|_{\Phi^*})^*\cong (\ff S_\Phi,\|\cdot\|_{\Phi})$.
Moreover, $\ff S_\Phi$ and $\cl B_1(\cl H)$ coincide elementwise and so does 
$\ff S_{\Phi^*}$ and $\cl B_0(\cl H)$. Consequently, 
$(\cl B_1(\cl H), \|\cdot\|_\Phi)^*\cong (\cl B(\cl H),\|\cdot\|_{\Phi^*})$ and $(\cl B_0(\cl H), \|\cdot\|_{\Phi^*})^*\cong (\cl B_1(\cl H),\|\cdot\|_{\Phi})$ which implies $(\cl B_0(\cl H), \|\cdot\|_{\Phi^*})^{**}\cong (\cl B(\cl H),\|\cdot\|_{\Phi^*}).$ This completes the proof.
\end{proof}

\begin{thm}\label{compacts-are-norming}
Let $\Phi$ be an arbitrary s.n. function equivalent to the maximal s.n. function. If $T\in (\cl B(\cl H),\|\cdot\|_{\Phi^*})$ is a compact operator, then $T\in \cl{AN}_{\Phi^*}.$
\end{thm}

\begin{proof}
If $T\in \cl B_0(\cl H)$, then $TP_{\cl M}\in \cl B_0(\cl H)$ for any closed subspace $\cl M$ of $\cl H$. So it suffices to show that $T\in \cl N_{\Phi^*}(\cl H)$. Since $(\cl B_0(\cl H), \|\cdot\|_{\Phi^*})$ and $(\cl B_0(\cl H), \|\cdot\|_{\Phi^*})^{**}$ are Banach spaces, the Banach space theory guarantees the existence of the canonical map $\wedge:(\cl B_0(\cl H), \|\cdot\|_{\Phi^*})\longrightarrow (\cl B_0(\cl H), \|\cdot\|_{\Phi^*})^{**}$ given by $T\mapsto \hat T$ and $\|\hat T\|=\max\{|\hat T(\varphi)|:\varphi\in (\cl B_0(\cl H), \|\cdot\|_{\Phi^*})^*, \|\varphi\|=1\}$ so that there exists $\varphi_0\in (\cl B_0(\cl H), \|\cdot\|_{\Phi^*})^*$ with $\|\varphi_0\|=1$ such that $\|\hat T\|=|\hat T(\varphi_0)|=|\varphi_0(T)|$. Corresponding to this $\varphi_0$ there exists a unique $A_0\in (\cl B_1(\cl H), \|\cdot\|_{\Phi})$ so that $\varphi_0(X)=\tr(A_0X)$ for every $X\in (\cl B_0(\cl H), \|\cdot\|_{\Phi^*})$ with $\|\varphi_0\|=\|A_0\|_\Phi$. 
Since the diagram below commutes
$$
\begin{tikzcd}[font=\large]
(\cl B_0(\cl H), \|\cdot\|_{\Phi^*})\arrow{rr}{\wedge} \arrow[hookrightarrow]{ddrr}{i} & & (\cl B_0(\cl H), \|\cdot\|_{\Phi^*})^{**}\arrow{dd}{f} \\
& & \\
& & (\cl B(\cl H), \|\cdot\|_{\Phi^*})
\end{tikzcd}
$$
where $\wedge$ is the canonical map between the space $(\cl B_0(\cl H), \|\cdot\|_{\Phi^*})$ and its double dual, $f$ is the isometric isomorphism resulting from Lemma \ref{prereq-doubledual}, and $i$ is the inclusion map. The operator in $(\cl B(\cl H), \|\cdot\|_{\Phi^*})$ associated with $\hat T\in (\cl B_0(\cl H), \|\cdot\|_{\Phi^*})^{**}$ is the operator $T$ itself. So, $\|T\|_{\Phi^*}=\|\hat T\|$ which implies that there exists  $A_0\in (\cl B_1(\cl H), \|\cdot\|_{\Phi})$ such that $\|\hat T\|=|\tr(A_0T)|$ with $\|A_0\|_\Phi=1.$ This proves that $T\in \cl N_{\Phi^*}(\cl H)$.
\end{proof}

\section{Acknowledgement}

Results in this paper were announced at the Analysis Seminar at the Department of Pure Mathematics at University of Waterloo in July 2016. The author would like to thank Kenneth R. Davidson, the organizer of this seminar, for providing the opportunity to present this work. We are very much indebted to Laurent W. Marcoux for his critical reading of the paper and his many helpful suggestions for improvements. The author would also like to express his deep gratitude for the guidance of his Ph.D. advisor, Vern I. Paulsen. His suggestions have been invaluable and without his guidance none of this would have been possible.

\providecommand{\bysame}{\leavevmode\hbox to3em{\hrulefill}\thinspace}
\providecommand{\MR}{\relax\ifhmode\unskip\space\fi MR }
\providecommand{\MRhref}[2]{%
  \href{http://www.ams.org/mathscinet-getitem?mr=#1}{#2}
}
\providecommand{\href}[2]{#2}

\hfill Satish K. Pandey

\hfill Department of Pure Mathematics

\hfill University of Waterloo

\hfill Waterloo, Ontario, N2L 3G1, Canada

\hfill \emph{E-mail address:} satish.pandey@uwaterloo.ca

\end{document}